\documentclass[11pt,a4paper,twoside]{article}
\usepackage{fancyhdr}
\usepackage{amsbsy,amsfonts,amssymb,amsmath}
\usepackage{amsthm}
\usepackage{stmaryrd} 
\usepackage{multicol}
\usepackage{graphicx}
\graphicspath{{../Figures/}}
\usepackage{color}
\definecolor{bleuf}{rgb}{0.1,0.1,0.6}
\usepackage[round]{natbib}
\usepackage{hyperref}
\hypersetup{
       colorlinks = true,           
       breaklinks = true,          
       urlcolor = bleuf,           
       linkcolor =bleuf,                  
       citecolor = bleuf,                     
       linktocpage = true,
       pdftitle={Diffusion approximation of a multilocus model with assortative mating},   
       pdfauthor={A.M. Etheridge and S. Lemaire},                  
    }

\makeatletter
\renewcommand{\section}{\@startsection{section}{1}{0mm}%
{\baselineskip}{\baselineskip}%
{\large\bfseries\centering}}

\renewcommand{\subsection}{\@startsection{subsection}{2}{0mm}%
{0.5\baselineskip}{0.2\baselineskip}%
{\large\bfseries}}

\renewcommand{\subsubsection}{\@startsection{subsubsection}{1}{0mm}%
{0.4\baselineskip}{0.1\baselineskip}%
{\slshape\bfseries}}
\renewcommand{\subsubsection}{\@startsection{subsubsection}{1}{0mm}%
{0.4\baselineskip}{0.1\baselineskip}%
{\slshape\bfseries}}
\makeatletter

\newcommand*{\andname}{\&}

\numberwithin{equation}{section} 
\DeclareMathOperator{\NN}{{I\hspace{-.15em}N}}

\DeclareMathOperator{\R}{{I\hspace{-.15em}R}}
\DeclareMathOperator{\PP}{{I\hspace{-.15em}P}}
\DeclareMathOperator{\EE}{{I\hspace{-.15em}E}}
\DeclareMathOperator{\un}{1\hspace{-.29em}I}

\def\st{\text{ s. t. }}

\newcommand\bb[1]{\boldsymbol{#1}}
\newcommand\mc[1]{\mathcal{#1}}
\newcommand{\intervalle }[4]{\mathopen{#1}#2
\mathclose{}\mathpunct{};#3
\mathclose{#4}}
\newcommand\enum[2]{{\intervalle\llbracket {#1}{#2}\rrbracket}}
\newcommand\enu[1]{{\intervalle\llbracket {1}{#1}\rrbracket}}
\newcommand\eq{\equiv}

\newtheorem{thm}{Theorem}[section]
\newtheorem{lem}{Lemma}[section]
\newtheorem{prop}{Proposition}[section]
\newtheorem{cor}{Corollary}[section]

\theoremstyle{definition}
\newtheorem{ex}{Example}[section]
\newtheorem{exs}[ex]{Examples}
\newtheorem{rem}{Remark}[section]
\newcounter{hyp}
\newenvironment{hyp}[1][]{\refstepcounter{hyp}\par\noindent \textbf{ Assumption~H\thehyp: #1} \itshape}{\rmfamily\par}
\newcommand{\Href}[1]{H\ref{#1}}
\newcommand{\Sref}[1]{\S\ref{#1}}

\begin{document}
\title{\textsc{Diffusion approximation of a multilocus model\\ with assortative mating}} 
\author{A. M. Etheridge\thanks{Department of Statistics, University of Oxford, 1 South Parks Road, Oxford OX1 3TG, UK; email:  
\url{etheridg@stats.ox.ac.uk}. AME supported in part by EPSRC grant no. EP/G052026/1} \and S. Lemaire\thanks{Univ. Paris-Sud, Laboratoire de Math\'ematiques, UMR 8628, Orsay
F-91405; CNRS, Orsay, F-91405;  email: \url{sophie.lemaire@math.u-psud.fr}.}}
\date{}
\maketitle
\bigskip
\begin{abstract}To understand the effect of assortative mating on the genetic evolution of a population, 
we consider a finite population in which each individual has a type, determined by a sequence of $n$ diallelic loci.  
We assume that the population evolves according to a Moran model with weak assortative mating, strong recombination 
and  low mutation rates. With an appropriate rescaling of  time, we obtain that the evolution of the genotypic frequencies  
in a large population can be approximated by the evolution of the product of the allelic frequencies at each locus, 
and the vector of the allelic frequencies is approximately governed by a diffusion.  
We present some features of the limiting 
diffusions (in particular their boundary behaviour and  conditions under which the allelic frequencies 
at different loci evolve independently). If mutation rates are strictly positive then the limiting 
diffusion is reversible and, under some assumptions, the critical points of the stationary density 
can be characterised.  
\end{abstract}
\bigskip
\textbf{AMS 2000 Subject Classification}: 60J20, 92D25, 60J70
\bigskip\\
\textbf{Key words}: Moran model, population genetics, multilocus models, assortative mating, diffusion approximation.
\newpage
\section{Introduction}
The aim of this paper is to construct and analyse a diffusion approximation for a diallelic multilocus 
reproduction model with assortative mating, recombination and mutation. 
Our starting point is a variant of the Moran model.  We suppose that the population is monoecious\footnote{Every individual has both male and female sexual organs.},  
haploid\footnote{Each cell has one copy of each chromosome.} and of constant size $N$.  This will be
an overlapping generation model, but, in contrast to the usual Moran framework, we suppose that reproduction takes place 
at discrete times  1, 2, \ldots.
 In each time step, a mating event occurs between two individuals $I_1$ and $I_2$;  $I_1$ is replaced 
by an offspring, so that the size of the population is kept constant.  The genotype of the offspring 
is obtained from those of $I_1$ and $I_2$ through a process of recombination followed by mutation
which we make precise in \Sref{sectmodel}.
In the classical Moran model, the two individuals $I_1$ and $I_2$ are chosen at random from the population. 
Here, to study the effects of assortative mating, we assume that the 
first individual, $I_1$, is still chosen at random, but the second individual, $I_2$, is sampled with a 
probability that depends on its genotype and on the genotype of the first selected individual. 
The genotype of an individual is composed of a finite number, $n$, of loci  with two alleles per locus 
denoted by $0$ and $1$. To characterise the assortative mating, we introduce a real parameter 
$s_{\bb{i},\bb{j}}$ for every pair of genotypes $(\bb{i},\bb{j})$.  
If $I_1$ has genotype $\bb{i}$, then, in the draw of $I_2$,
an individual with genotype $\bb{j}$ has a probability proportional to $1+\frac{1}{N}s_{\bb{i},\bb{j}}$ of 
being selected. 

Diffusion approximations for different selection-mutation models have been studied 
extensively in the 
one-locus case (see, for example, \citealt{EthierKurtz}, Chapter 10). 
The coefficients $1+\frac{1}{N}s_{\bb{i},\bb{j}}$ of our model play the same r\^{o}le as the 
(viability) selection coefficients in a Wright-Fisher model for a diploid population.
Since they depend on the types of both parents they result in nonlinear (frequency
dependent) selection (see \Sref{secttwolocusWF}). \cite{EthierNagylaki89} 
study two-locus Wright-Fisher models for a 
panmictic\footnote{Every individual is equally likely to mate with every other.}, monoecious, diploid population of constant size under various assumptions on selection and 
recombination. Depending on the strength of the linkage between the two loci, they obtain different 
types of diffusion approximation: limiting diffusions for gametic\footnote{Gametes are produced during reproduction.
A gamete contains a single copy of each chromosome, composed of segments of the two chromosomes in the corresponding parent.  Two gametes, one from each parent, fuse to produce an offspring.} frequencies if the recombination 
fraction multiplied by the population size tends to a constant as the size tends to $+\infty$ 
(so-called \emph{tight linkage}) and  limiting diffusions for allelic frequencies if the recombination 
fraction multiplied by the population size tends to $+\infty$ (so-called \emph{loose linkage}).
To our knowledge, this work has not been extended to more general multilocus models with
recombination and assortative mating.
Nevertheless, there is a large 
body of work on multilocus genetic systems.  Most theoretical investigations assume that the size of 
the population is infinite, so that the random genetic drift can be ignored; the evolution of genotypic 
frequencies is then described by recursive equations or by differential equations 
(see \citealt{Christiansen} and references within). A comparison between infinite and finite population models with random mating is presented in  \cite{BaakeHerms}. 
A review of several simulation studies can be found 
in the introduction of \cite{Devaux}. Among these,  
the `species formation model', introduced by \cite{Higgs} inspired our work. In their model, mating is only possible between 
individuals with sufficiently similar genotypes, so that from the point of view of reproduction
the population is split into isolated subgroups.  Their simulations display a succession of divisions 
and extinctions of subgroups.  In this paper we generalise their assortative mating criterion to one 
defined through the family of parameters $s_{\bb{i},\bb{j}}$ and provide a general theoretical 
treatment.

To give an overview of our results, we first consider a particular pattern of  assortative mating. 
Let us assume that the frequency of matings between two individuals of types $\bb{i}$ and $\bb{j}$ 
depends only on the number of loci at which their allelic types differ (and not on the positions of
those loci along the genome).  We then have a model with $n+1$ assortment parameters, denoted by 
$s_0,\ldots,s_n$, obtained by setting $s_{\bb{i},\bb{j}}=s_{k}$ if the genotypes $\bb{i}$, $\bb{j}$ 
are different at exactly $k$ loci (regardless of their positions). This mating criterion will be called 
the \emph{Hamming criterion} in what follows. A decreasing sequence $s_0\geq s_1\geq \ldots \geq s_n$ 
will describe a positive assortative mating (individuals mate preferentially with individuals that are 
similar).  An increasing sequence $s_0\leq s_1\leq \ldots \leq s_n$ will describe a negative assortative 
mating (individuals mate preferentially with individuals that are dissimilar).

We establish a weak convergence of the Markov chain describing the genetic evolution of the population 
as its size tends to $+\infty$, under a hypothesis on the recombination distribution that corresponds to 
loose linkage (during each reproduction event, 
recombination between any pair of loci occurs with a positive probability)  and under the assumption that  mutations occur independently at each locus with the same 
rates (at each locus, the rate of mutation of a type $0$ allele to a type $1$ allele is $\frac{\mu_{0}}{N}$ 
and the rate of mutation of a type $1$ to a type $0$ is $\frac{\mu_{1}}{N}$). 
In particular, while mutation and assortment parameters are rescaled with population size, 
recombination is not.  As a result, we see a {\em separation of timescales}.
Due to recombination, the genotypic frequencies rapidly converge to a product distribution which is 
characterised by its marginals, that is by the $0$-allelic frequencies at each locus. 
We show that, at a slower rate, the allelic frequencies converge to a multidimensional diffusion,
whose components are 
coupled only through an infinitesimal drift term (in the mathematical sense) arising from the assortative mating. \\

 Let us describe some  features of the limiting diffusion. 
If  $s_{1}-s_{0}=s_{2}-s_{1}=\ldots=s_{n}-s_{n-1}$ then the frequencies of the $0$-allele
at each locus evolve according to independent Wright-Fisher diffusions with mutation rates $\mu_0$ and 
$\mu_1$ and symmetric balancing selection with strength $\frac{1}{2}(s_1-s_0)$; that is they solve
the following stochastic differential equation:
\begin{multline*}
dx_t=\sqrt{x_t(1-x_t)}dW_t+\Big(\mu_1(1-x_t)-\mu_0x_{t}+(s_1-s_0)(1/2-x_t)x_t(1-x_t)\Big)dt,
\end{multline*}
where $(W_t)_{t\geq 0}$ is a standard Brownian motion. 
In all other cases, the allelic frequencies at different loci no longer evolve independently.  Instead
the vector of $0$-allelic frequencies $(x_t(1),\ldots,x_t(n))$ is governed by the stochastic differential 
equation:
\begin{multline}
\label{hammingsde}
dx_t(i)=\sqrt{x_t(i)(1-x_t(i))}dW_t(i)\\+\Big(\mu_1(1-x_t(i))-\mu_0x_{t}(i)+(1/2-x_t(i))x_t(i)(1-x_t(i))P_{i,s}(x_t)\Big)dt,
\end{multline}
where $(W_t(1))_{t\geq 0}$,\ldots, $(W_t(n))_{t\geq 0}$ denote $n$ independent standard Brownian motions 
and $P(\hat{x}^{(i)})$ is a symmetric polynomial function of the $n-1$ variables 
$x(j)(1-x(j))$,  $j\in\{1,\ldots,n\}\setminus\{i\}$ whose coefficients depend only on the parameters 
$s_1-s_0,\ldots,s_n-s_{n-1}$. More precisely, $P(\hat{x}^{(i)})$ is an increasing function of each 
parameter $s_1-s_0,\ldots,s_n-s_{n-1}$ (see Theorem \ref{thgen} for an explicit formula of $P(\hat{x}^{(i)})$). 
When the mutation rates $\mu_0$ and $\mu_1$ are strictly positive, the limiting diffusion has a reversible 
stationary measure, the density of which is explicit. When the two mutation rates are equal to $\mu>0$,  
we describe the properties of the critical points of the density of the stationary measure. 
In particular, we find sufficient conditions on $\mu$ and $s_1-s_0,\ldots,s_n-s_{n-1}$ for the state 
where the frequencies of the two alleles are equal to $1/2$ at each locus to be a global maximum and 
for the stationary measure to have $2^n$ modes. These sufficient conditions generalise the independent case. 
For example, when $\mu>1/2$  they imply the following results:
\begin{enumerate}
\item if $s_{\ell}-s_{\ell-1}\geq -(8\mu-4)$ for every $\ell\in\{1,\ldots,n\}$, 
then $(1/2,\ldots,1/2)$ is the only mode of the stationary measure;
\item if $s_{n}-s_{n-1}\leq \ldots \leq s_{1}-s_{0} < -(8\mu-4)$ then the stationary measure 
has $2^{n}$ modes. 
\end{enumerate}
 
These results can be extended to other patterns of assortative mating. In fact, we need only make the 
following assumption on the parameters $s_{\bb{i},\bb{j}}$: the value of the assortment parameter 
$s_{\bb{i},\bb{j}}$ between two genotypes $\bb{i}$ and $\bb{j}$ is assumed to be the same as the value 
of $s_{\bb{j},\bb{i}}$ and to depend only on the loci at which $\bb{i}$ and $\bb{j}$ differ. 
In particular this implies that the value of $s_{\bb{i},\bb{i}}$ is the same for every genotype $\bb{i}$. 
This generalises the Hamming criterion and allows us to consider more realistic situations in which
the influence on mating choice differs between loci (see \Sref{sec: assumptions}).
It transpires that, under these assumptions, the limiting diffusion does not depend on the whole family 
of assortment parameters, but only on one coefficient per subgroup of loci $L$.  We denote this
coefficient $m_{L}(s)$.  It is the mean of the assortment parameters for pairs of genotypes that 
carry different alleles on each locus in $L$ and identical alleles on all other loci. 
The stochastic differential equation followed by the limiting diffusion can still be described by 
equation~\eqref{hammingsde} if the symmetric polynomial term $P(\hat{x}^{(i)})$ in the drift of the 
$i$-th coordinate is replaced by a non-symmetric polynomial term
$P_i(\hat{x}^{(i)})$ in the coefficients of which the quantities
$m_{L\cup \{i\}}(s)-m_{L}(s)$ for $L\subset \{1,\ldots,n\}$ 
replace $s_1-s_0$,\ldots, $s_{n}-s_{n-1}$. \\

The rest of the paper is organized as follows. In \Sref{sectmodel}, we present our multilocus Moran model.
In \Sref{sectonelocus}, we describe the diffusion approximation for the one-locus model and compare it with a diffusion 
approximation for a population undergoing mutation and `balancing selection'.  We recall some 
well-known properties of this diffusion, in particular the boundary behaviour and the
form of the stationary measure, for later comparison with the multilocus case. In \Sref{secconv}, we state our main
result concerning convergence to a diffusion approximation in the multilocus case (Theorem~\ref{thgen})
and give two equivalent expressions for the limiting diffusion.  We then compare with the two-locus 
diffusion approximation obtained in \cite{EthierNagylaki89}. The proof of Theorem~\ref{thgen} 
is postponed until \Sref{secconvproof}. In \Sref{sectdescrlim}, we derive some general properties of the limiting diffusion.  
\Sref{sectstatmeas} is devoted to the study of the density of the stationary measure.  An appendix collects some 
technical results used in the description of the limiting diffusion. 
\section{The discrete model\label{sectmodel}}
This section is devoted to a detailed presentation of the individual based model. 
The assumptions on assortative mating, recombination and mutation that we will require 
to establish a diffusion approximation for the allelic frequencies are discussed at the end of the section.  

\subsection{Description of the model}
We consider  a monoecious and haploid population of size $N$ where the type of each individual is 
described by a sequence of $n$ diallelic loci. For the sake of brevity, let the set of loci be 
identified with the set of integers $\enu{n}:=\{1,\ldots,n\}$ and let the two alleles at each locus 
be labelled $0$ and $1$. The type of an individual is then identified by an $n$-tuple
$\bb{k}:=(k_1,...,k_n)\in\{0,1\}^n$. Let $\mathcal{A}=\{0,1\}^n$ be the set of possible types. 
The proportion of individuals of type $\bb{k}$ at time $t\in\NN$ will be  denoted by 
$Z^{(N)}_t(\bb{k})$ so that the composition of the population is described by the set 
$Z^{(N)}_t=\{Z^{(N)}_t(\bb{k}),\ \bb{k}\in \mathcal{A}\}$. At each unit of time the population evolves 
under the effect of assortative mating, recombination and mutation as follows.
\paragraph{Assortative mating:} 
at each time $t$, two individuals are sampled from the population in such a way that: 
\begin{enumerate}
\item  the first individual has probability $Z^{(N)}_t(\bb{i})$ of being of type $\bb{i}$;  
\item  given that the first individual chosen is of type $\bb{i}$, the probability 
that the second individual is of type $\bb{j}$ is
$$\frac{\big(1+s^{(N)}_{\bb{i},\bb{j}}\big)Z_t^{(N)}(\bb{j})}{\sum_{\bb{k}\in\mathcal{A}}
\big(1+s^{(N)}_{\bb{i},\bb{k}}\big)Z_t^{(N)}(\bb{k})},$$
where the {\em assortment parameters}
$\{s^{(N)}_{\bb{i},\bb{j}}, \bb{i},\bb{j}\in\mc{A}\}$ are fixed nonnegative real numbers\footnote{We are allowing a small chance of self-fertilisation.}. 
\end{enumerate}
The population at time $t+1$ is obtained by replacing the first chosen individual with an 
offspring whose type is the result of the following process of recombination followed by mutation.
\paragraph{Recombination:} for each subset $L$ of $\enu{n}$, let $r_L$ denote the probability that 
the offspring inherits the genes of the first chosen parent at loci $\ell\in L$ and the genes of the 
second parent at loci $\ell\not\in L$. The family of parameters $\{r_L,\ L\subset\enu{n}\}$ defines a 
probability distribution, called the \emph{recombination distribution}, on the power set 
$\mc{P}(\enu{n})$ (it was first introduced in this manner by \cite{Geiringer} 
to describe 
the recombination-segregation of gametes in a diploid population).  It is natural to assume that the 
two parents contribute symmetrically to the offspring genotype, that is:
\begin{hyp}\label{Hrecombsym} for each subset $L$ of $\enu{n}$, $r_{L}=r_{\bar{L}}$ where $\bar{L}$ 
denotes the complementary set of loci, $\enu{n}\setminus L$.  
\end{hyp}
With this notation, the probability that, before mutation, the offspring of a pair of individuals of types
$(\bb{i},\bb{j})$ is of type $\bb{k}$ is 
$$q((\bb{i},\bb{j});\bb{k})=\sum_{L\subset \enu{n}}r_L\un_{\{\bb{k}=(\bb{i}_{|L},\bb{j}_{|\bar{L}})\}}.$$
Let us express some classical examples of recombination distributions in this notation:
\begin{exs}\label{egrecomb}\
\begin{enumerate}
\item $r_{\emptyset}=r_{\enu{n}}=\frac{1}{2}$ (no recombination, also called absolute linkage)
\item $r_{I}=2^{-n}$ for each $I\in\mathcal{P}(\enu{n})$ (free recombination)
 \item $r_{\enu{x}}=r_{\enum{x+1}{n}}=\frac{r}{2(n-1)}$ for $1\leq x \leq n-1$ and $r_{\emptyset}=r_{\enu{n}}=1/2(1-r)$ where $r$ denotes an element of $]0,1]$ (at most one exchange between the sequence of loci  which occurs with equal probability at each position).
\end{enumerate}
\end{exs}
Finally we superpose mutation.
\paragraph{Mutation:}  we assume that mutations occur independently and at the same rate at 
each locus: $\mu^{(N)}_1$ will denote the probability that an allele $1$  at a given locus 
of the offspring changes into allele $0$ and $\mu^{(N)}_0$ the probability of the reverse mutation. 
The probability that the mutation process changes a type $\bb{k}$ into a type $\bb{\ell}$ is: 
$$\mu^{(N)}(\bb{k},\bb{\ell}):=\prod_{i=1}^{n}(\mu^{(N)}_{k_i})^{|\ell_i-k_i|}(1-\mu^{(N)}_{k_i})^{1-|\ell_i-k_i|}.$$ 
\subsection{Expression for the transition probabilities}
It is now elementary to write down an expression for the transition probabilities of our model.
In the notation above, if $z=\{z(\bb{k}),\bb{k}\in \mathcal{A}\}$ describes the proportion of individuals 
of each type in the population at a given time, then the probability that, in the
next time step, the
number of individuals of type $\bb{j}$ increases by one and the number of individuals of type $\bb{i}\neq \bb{j}$ 
decreases by one is 
$$
f_N(z,\bb{i},\bb{j}):=\sum_{\bb{k},\bb{\ell}\in\mathcal{A}}z(\bb{i})z(\bb{k})w^{(N)}(z,\bb{i},\bb{k})q((\bb{i},\bb{k});\bb{\ell})\mu^{(N)}(\bb{\ell},\bb{j})
$$
where 
$$
w^{(N)}(z,\bb{i},\bb{k})=\frac{1+s^{(N)}_{\bb{i},\bb{k}}}{\sum_{\bb{\ell}\in\mathcal{A}}(1+s^{(N)}_{\bb{i},\bb{\ell}})z(\bb{\ell})}.
$$
\subsection{Assumptions on assortative mating, recombination and mutation}
\label{sec: assumptions}
In order to obtain a diffusion approximation for a large population, we assume that mutation and 
assortment parameters are both $O(N^{-1})$, so we set
\begin{hyp}\label{Hscaling} $\mu^{(N)}_\epsilon=\frac{\mu_\epsilon}{N}$ for 
$\epsilon\in\{0,1\}$ and $s^{(N)}_{\bb{i},\bb{j}}=\frac{s_{\bb{i},\bb{j}}}{N}$ for $\bb{i},\bb{j}\in\mc{A}$. 
\end{hyp}
Just as in the two-locus case studied by \cite{EthierNagylaki89}, we can expect 
diffusion approximations to exist under two quite different assumptions on 
recombination, corresponding to tight and loose linkage.  Here we focus
on loose linkage.
More precisely, we assume that the recombination distribution does not depend 
on the size of the population and that recombination can occur between any pair of loci:
\begin{hyp}\label{Hlinkage}
For every $I\in\mathcal{P}(\enu{n})$, $r_I$ does not depend on $N$ and for any distinct 
integers $h,k\in\enu{n}$, there exists a subset $I \in\mc{P}(\enu{n})$ such that $h\in I$,  
$k\not \in I$ and $r_I>0$. 
\end{hyp}

This assumption is satisfied for the last two examples of recombination distribution presented 
in Example~\ref{egrecomb}, but not in the absolute linkage case. 
In infinite population size multilocus models with random mating, and without selection, 
this condition is known to ensure that in time the genotype frequencies will converge to 
linkage equilibrium, where they are products of their 
respective marginal allelic frequencies (see \citealt{Geiringer} and 
\citealt{Nagylaki93}  for a study of 
the evolution of multilocus linkage disequilibria under weak selection). 

In order that the generator of the limiting diffusion has a tractable form, we shall make 
two further assumptions on the family of assortment coefficients
$s=\{s_{\bb{i},\bb{j}}, (\bb{i},\bb{j})\in\mc{A}^2\}$:

\begin{hyp}\label{Hpairing} for every $(\bb{i},\bb{j})\in\mc{A}^2$, 
\begin{enumerate}
 \item $s_{\bb{i},\bb{j}}=s_{\bb{j},\bb{i}}$
\item the value of $s_{\bb{i},\bb{j}}$ depends only on the set of 
loci $k$ at which $i_k=0$ and $j_k=1$ and on the set of loci $\ell$ at 
which $i_\ell=1$ and $j_\ell=0$. 
\end{enumerate}
\end{hyp} 
These conditions mean that the probability of mating between two individuals at a fixed time 
depends only on the difference between their types. 
In particular,  two individuals of the same type will have a probability of mating 
that does not depend on their common type: $s_{\bb{i},\bb{i}}=s_{\bb{j},\bb{j}}$ for 
every $\bb{i},\bb{j}\in\mc{A}$. 
In the one-locus case, this assumption means that the model distinguishes only two 
classes of pairs of individuals since $s_{0,1}=s_{1,0}$ and $s_{0,0}=s_{1,1}$.  \\
\newpage
In the two-locus case, this assumption leads to a model with five assortment parameters:
\begin{itemize}
 \item 
 one parameter, $s_{00,00}=s_{11,11}=s_{10,10}=s_{01,01}$, for pairs of individuals having the same genotype,
\item  one parameter $s_{00,10}=s_{10,00}=s_{11,01}=s_{01,11}$ for pairs of individuals whose genotypes only differ on the first locus,
\item one parameter, $s_{00,01}=s_{01,00}=s_{11,10}=s_{10,11}$, for pairs of individuals whose genotypes only differ on the second locus,  
\item two parameters $s_{01,10}=s_{10,01}$ and $s_{00,11}=s_{11,00}$ for pairs of individuals whose genotypes  differ on the two loci. 
\end{itemize}

To describe positive or negative  assortative mating we have to choose how to quantify  similarities between two types. Let us present two  criteria that provide  assortment parameters for which assumption \Href{Hpairing} is fulfilled:
\begin{enumerate}
\item {\bf Hamming Criterion.} One simple measure to quantify  similarities between 
two types is the number of loci with distinct alleles: $s_{\bb{i},\bb{j}}$ will be defined 
as nonnegative reals depending only on the Hamming distance between $\bb{i}$ and $\bb{j}$ 
denoted by ${\displaystyle d_h(\bb{i},\bb{j}):=\sum_{\ell=1}^{n}|i_{\ell}-j_{\ell}|}$. A positive assortative 
mating will be described by a sequence of $n+1$ nonnegative reals $s_0\geq s_1\geq \ldots \geq s_n$ 
by setting $s_{\bb{i},\bb{j}}=s_{d_{h}(\bb{i},\bb{j})}$ for every $\bb{i},\bb{j}\in\mc{A}$. 
This criterion will be called \emph{Hamming criterion}. 
\item {\bf Additive Criterion.}  If we assume that the assortment is based on a phenotypic 
trait which is determined by the $n$ genes whose effects are similar and additive, then a 
convenient measure of the difference between individuals of type  $\bb{i}$ and $\bb{j}$ 
is \linebreak ${\displaystyle d_a(\bb{i},\bb{j}):=|\sum_{\ell=1}^{n}(i_{\ell}-j_{\ell})|}$.  A positive assortative mating 
will be described by a sequence of $n+1$ nonnegative reals $s_0\geq s_1\geq \ldots \geq s_n$ 
by setting $s_{\bb{i},\bb{j}}=s_{d_{a}(\bb{i},\bb{j})}$ for every $\bb{i},\bb{j}\in\mc{A}$. 
This criterion will be called \emph{additive criterion}.
\end{enumerate}

The assortative mating in the species formation model of 
\cite{Higgs}
is a special case of the Hamming criterion.  The additive criterion is widely used in
models in which assortative mating is determined by an additive genetic trait.  For
example, \cite{Devaux} use it to investigate speciation in 
flowering plants due to assortative mating determined by flowering time.
Flowers can only be pollinated by other flowers that are open at the same time.  Modelling
flowering time as an additive trait, they observe an effect that is qualitatively similar 
to that observed in the simulations of \cite{Higgs}  for the Hamming criterion, 
namely continuous creation of reproductively isolated subgroups. 

With the Hamming and additive
criteria, every locus is assumed to have an identical positive or negative 
influence  on the assortment.  As we have defined a general family of 
assortment parameters, it is 
possible to consider more complex situations.  For instance, we can take into account that some 
loci have a greater influence on the mating choice than others by dividing the set of loci into 
two disjoint subgroups $G_1$ and $G_{2}$; we introduce two sets of 
assortment parameters $s^{(1)}$ 
and $s^{(2)}$ that satisfy assumption \Href{Hpairing} for the subgroups of loci $G_1$ and $G_2$ 
respectively. If we assume that the effects of the two subgroups are additive we set 
${s_{\bb{i},\bb{j}}=s^{(1)}_{\bb{i}_{|G_1},\bb{j}_{|G_1}}+s^{(2)}_{\bb{i}_{|G_2},\bb{j}_{|G_2}}}$ 
for every $\bb{i},\bb{j}\in\mc{A}$. 
This defines a set of assortment parameters that satisfies assumption \Href{Hpairing}. More generally, 
any set of assortment parameters defined as a function of $s^{(1)}$ and $s^{(2)}$  
satisfies assumption \Href{Hpairing}.
\section{The one-locus diffusion approximation\label{sectonelocus}}
Before studying the multilocus case, for later comparison, in this section we record
some properties of the one-locus model.
\subsection{The generator of the one-locus diffusion}
In the case of one locus ($n=1$), under assumption \Href{Hscaling}, the frequency of $0$-alleles 
satisfies:

\begin{alignat*}{2}
 \EE_z[Z^{(N)}_1(0)-z] = &\ \frac{1}{N^2}\Big((1-z)\mu_1-z\mu_0\\ 
&\hspace{1cm} +\frac{1}{2}z(1-z)((s_{1,0}-s_{1,1})(1-z)-(s_{0,1}-s_{0,0})z)\Big)+O(1/N^3)\\
\EE_z[(Z^{(N)}_1(0)-z)^2] = &\ \frac{1}{N^2}z(1-z)+O(1/N^3)\\
\EE_z[(Z^{(N)}_1(0)-z)^4] = &\ O(1/N^4) 
\end{alignat*}
uniformly in $z$.\\
Therefore  the distribution of the frequency of $0$-alleles at time $[N^2t]$ is approximately 
governed, when $N$ is large, by a diffusion with generator:  
\begin{multline}\label{onelocusgenerator}
\mathcal{G}_{1,s}=\frac{1}{2}x(1-x)\frac{d^2}{dx^{2}}+\Big((1-x)\mu_1-x\mu_0\\+1/2x(1-x)((s_{1,0}-s_{1,1})(1-x)-(s_{0,1}-s_{0,0})x)\Big)\frac{d}{dx}.
\end{multline}
More precisely,  if $Z^{(N)}_0$ converges in distribution in $[0,1]$ as $N$ tends to $+\infty$, then $(Z^{(N)}_{[N^2t]})_{t\geq 0}$ converges in distribution  in the Skorohod space of c\`{a}dl\`{a}g functions $D_{[0,1]}([0,+\infty))$ to a diffusion with generator $\mathcal{G}_{1,s}$ (see, for example, \citealt{EthierKurtz}, Chapter 10).\\

If we assume that $s$ satisfies assumption \Href{Hpairing}, that is $s_{0,0}=s_{1,1}$ 
and $s_{0,1}=s_{1,0}$, and if we denote their common values by $s_0$ and $s_{1}$ respectively, 
then the drift has a simpler form and we obtain
$$ 
\mathcal{G}_{1,s}
=\frac{1}{2}x(1-x)\frac{d^2}{dx^{2}}+\Big((1-x)\mu_1-x\mu_0+(s_1-s_0)(1/2-x)x(1-x)\Big)\frac{d}{dx}.
$$
\begin{rem}\label{rem:WFonelocus}
This diffusion can also be obtained as an approximation of a diploid  model with random mating, 
mutation and weak selection in favour of homozygosity\footnote{A diploid individual is homozygous at a gene locus when its cells contain two identical alleles at the locus.} (when $s_0-s_1 > 0$) or in favour of 
heterozygosity (when $s_0-s_1 < 0$) (see, for example, \citealt{EthierKurtz}, Chapter 10). 
\end{rem}\medskip

\subsection{Properties of the one-locus diffusion}
\paragraph{Stationary measure.}
If $\mu_0$ and $\mu_1$ are strictly positive, this diffusion has a reversible stationary
measure. Its density with respect to Lebesgue measure on $[0,1]$ is given by Wright's formula:
$$g_{\mu,s}(x)=C_{\mu,s}x^{2\mu_1-1}(1-x)^{2\mu_0-1}\exp\Big(-1/2((s_{1,0}-s_{1,1})(1-x)^2+(s_{0,1}-s_{0,0})x^2)\Big)$$
 where the constant $C_{\mu,s}$ is chosen so 
that  ${\displaystyle \int_{0}^{1}g_{\mu,s}(x)dx=1}$.  This is plotted, for various parameter values, in 
Fig.~\ref{loiinv1} under the assumptions $\mu_1=\mu_0=\mu$, $s_{0,0}=s_{1,1}=s_0$ and 
$s_{0,1}=s_{1,0}=s_1$.
\begin{figure}[htb]
\centering\includegraphics[scale=0.47]{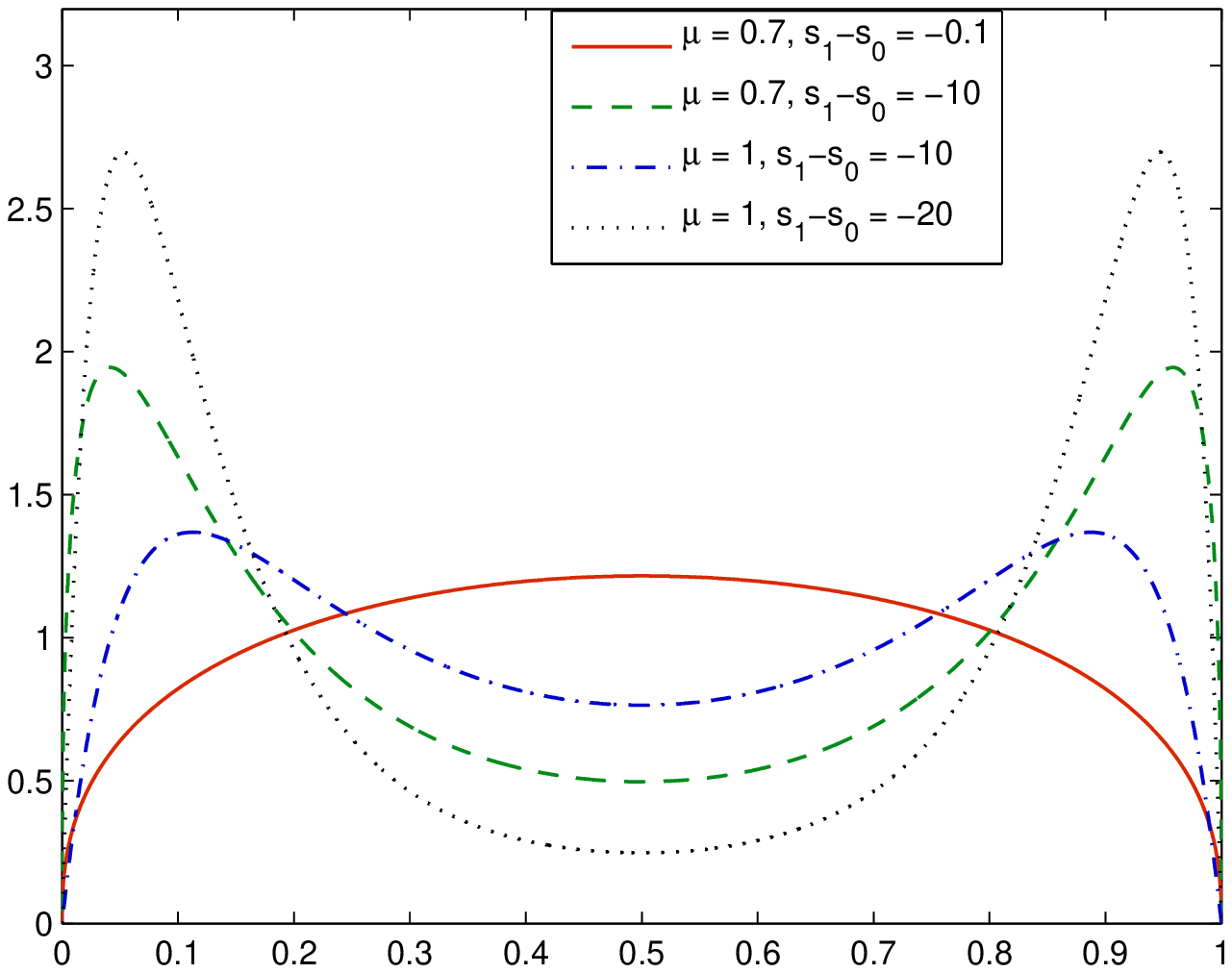}\includegraphics[scale=0.47]{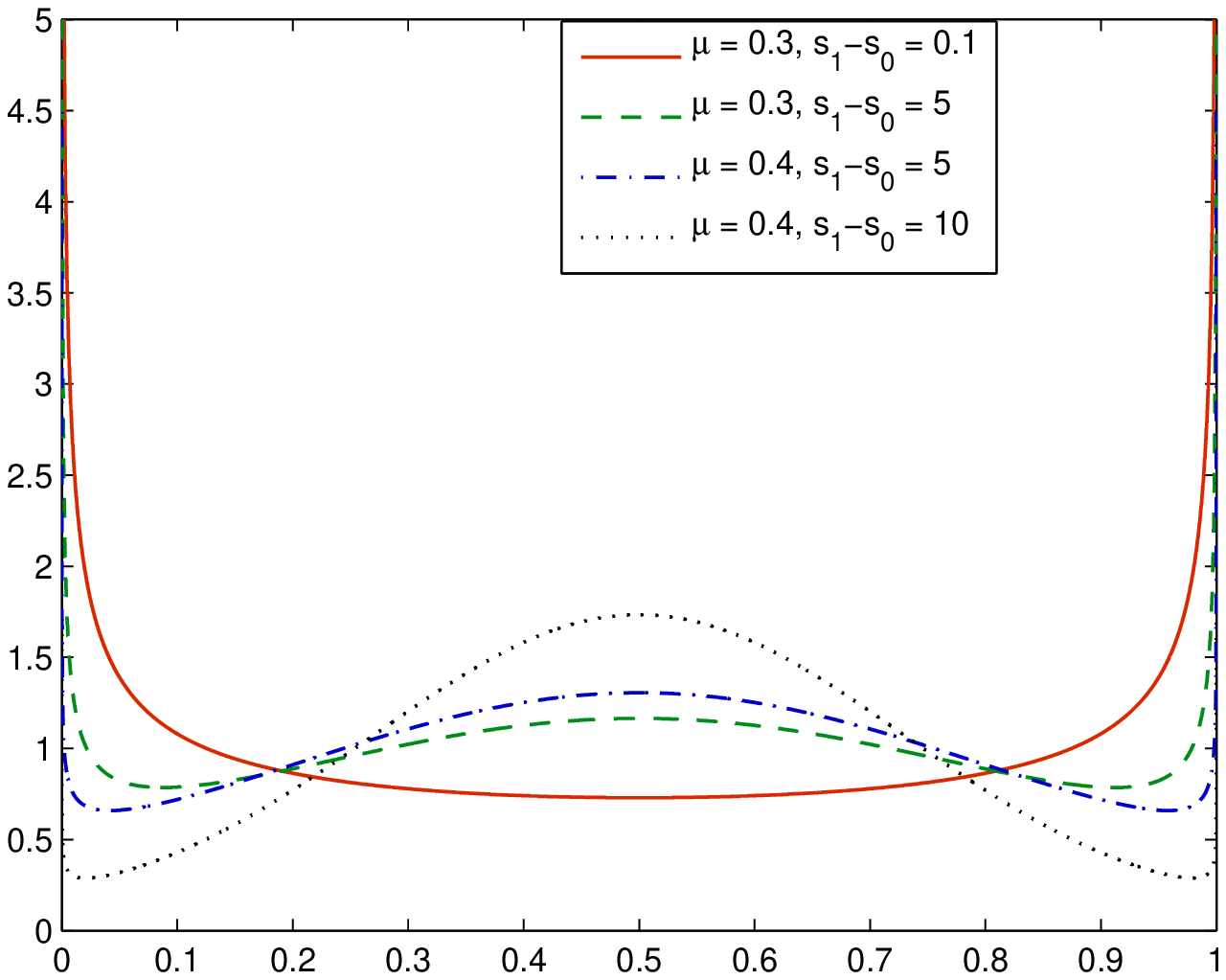}
\caption{Representation of the invariant density $g_{\mu,s}$ for the one-locus diffusion when 
the two mutations rates $\mu_1=\mu_0=\mu$, $s_{0,0}=s_{1,1}=s_0$ and 
$s_{0,1}=s_{1,0}=s_1$.
In the figure on the left, $\mu > 1/2$ and matings between individuals of the same allelic type are
favoured.  The density is bimodal if and only if $s_0-s_1 > 8\mu-4$.  In the figure on 
the right, $0<\mu<1/2$ and matings between individuals of different allelic types are favoured.
The density tends to $+\infty$ at the boundaries and has a global minimum at $1/2$ if and only 
if $s_1-s_0 \leq 4-8\mu$. \label{loiinv1}  }
\end{figure}

\paragraph{Boundary behaviour.}
According to the Feller boundary classification for one-dimensional diffusions 
(see e.g. \citealt{EthierKurtz}):
\begin{itemize}
\item[(i)] if $\mu_1=0$  then $0$  is an absorbing state 
and the diffusion exits from $]0,1[$ in a finite time almost surely; 
 \item[(ii)] if $\mu_1\geq 1/2$ then $0$ is an entrance boundary (started from a point in
$]0,1[$ the diffusion will not reach $0$ in finite time, but the process started from $0$
is well-defined);
\item[(iii)] if $0<\mu_1<1/2$ then $0$ is a regular boundary (starting from a point 
$z_0\in ]0,1[$ the diffusion has a positive probability of reaching $0$ before any point 
$b\in ]z_0,1]$ in a finite time and the diffusion started from $0$ is well-defined);
\end{itemize}
with the obvious symmetric definitions at $1$.
\section{Convergence to a diffusion in the multilocus case\label{secconv}}
%
In the case of several loci,  under assumptions \Href{Hscaling} and \Href{Hlinkage},
a Taylor expansion shows that the drift 
$\EE[Z^{(N)}_{t+1}(\bb{i})-Z^{(N)}_{t}(\bb{i}) \mid Z^{(N)}_{t}=z]$ is of order $\frac{1}{N^2}$ only 
inside the set of product distributions on $\{0,1\}^n$. This set is often called  the \emph{Wright manifold} or the
\emph{linkage-equilibrium manifold} and denoted by $\mathcal{W}_n$ (a population is said to be in linkage equilibrium if the genotype distribution  $z$ is in $\mc{W}_n$, that is if 
$ z(\bb{i})=z_{1}(i_1)\cdots z_{n}(i_n)\  \forall \bb{i}=(i_1,\ldots,i_n)\in \mathcal{A}$ 
where $z_j(x)=\sum_{\bb{k}\in\{0,1\}^{n-1}}z(k_1,\ldots,k_{j-1},x,k_{j+1},\ldots,k_n)$ 
denotes the frequency of  individuals having the allele $x$ at the $j$-th locus). \\ 
Outside this manifold, the drift pushes the process towards the Wright manifold at an 
exponential speed. Therefore to extend the diffusion approximation to the $n$-locus case, we introduce a change of variables composed of the $n$  0-allelic frequencies  and of $2^n-n-1$  processes that measure the deviation from the linkage equilibrium.\\
For a nonempty subset $L$ of $\{1,\ldots,n\}$,
\begin{itemize}
 \item let $X^{(N)}_{t}(L)=\sum_{\bb{j}\in\mc{A},\ \bb{j}_{|L}\eq 0}Z^{(N)}_{t}(\bb{j})$ 
denote the  proportion of individuals having the allele $0$ on all loci in $L$ at time $t$;
\item let $Y^{(N)}_{t}(L)=\prod_{i\in L}X^{(N)}_{t}(\{i\})-X^{(N)}_{t}(L)$ for $|L|\geq 2$ 
describe the linkage disequilibrium between the loci in $L$ at time $t$.
(This is just one of many ways to measure the linkage disequilibrium, see for example \cite{Burger00}, 
Chapter V.4.2, for other measures.)
\end{itemize}
The vector  of 0-allelic frequencies at time $t$ is  $X^{(N)}_{t}:=\big(X^{(N)}_{t}(\{1\}),...,X^{(N)}_{t}(\{n\})\big).$\\
The process $Y^{(N)}$ defined by $Y^{(N)}_t:=\big\{Y^{(N)}_t(L),\ L\subset\enu{n} \text{ such that } |L|\geq 2\big\} \text{ for } t\geq 0$  vanishes on the Wright manifold.\\
We shall show that if $t_N$ tends to $+\infty$ faster than $N$ then $Y^{(N)}_{[t_N]}$ converges to 0 
while if time is sped up by  $N^2$ then $X^{(N)}$ converges to a diffusion as $N$ tends to $+\infty$.  \\
Before giving a precise statement of the convergence result for the two processes $X^{(N)}$ and $Y^{(N)}$ (Theorem \ref{thgen}), 
let us introduce some notation in which to express the parameters of the limiting diffusion.
\subsection{Mean assortment parameters}
For a subset $L$ of loci, consider the set of pairs of genotypes that differ at each locus 
$\ell\in L$ and are equal at each locus $\ell\notin L$: 
$$F_L=\left\{(\bb{i},\bb{j})\in\mc{A}^2:\ i_{u}=1-j_u\ 
\forall u\in L \text{ and } i_u=j_u\ \forall u\in \bar{L}\right\}.$$ 
Let $m_L(s)$ denote the mean value of the assortment parameters for all pairs in this set $F_L$: 
$$m_L(s)=2^{-n}\sum_{(\bb{i},\bb{j})\in F_L}s_{\bb{i},\bb{j}}.$$
\newpage
\begin{exs}\
\label{exmeancoef}
\begin{enumerate}
\item In the two-locus case, 
\begin{eqnarray*}m_{\emptyset}(s)&=&\frac{1}{4}(s_{00,00}+s_{01,01}+s_{10,10}+s_{11,11}),\\
 m_{\{1\}}(s)&=&\frac{1}{4}(s_{00,10}+s_{10,00}+s_{01,11}+s_{11,01}),
\end{eqnarray*}
\begin{eqnarray*}
m_{\{2\}}(s)&=&\frac{1}{4}(s_{00,01}+s_{01,00}+s_{11,10}+s_{10,11}). 
\end{eqnarray*}
In each of these expressions the four coefficients  are equal by assumption \Href{Hpairing}.
$$m_{\{1,2\}}(s)=\frac{1}{4}(s_{00,11}+s_{11,00}+s_{01,10}+s_{10,01}).$$ In this expression the first (resp. last) two coefficients are equal by \Href{Hpairing}.  
\item  With the Hamming criterion, $m_L(s)=s_{|L|}$ for every $L\subset \enu{n}$, where $|L|$ denotes the number of loci in $L$. 
\item With the additive criterion, $m_{\emptyset}(s)=s_0$, $m_{\{\ell\}}(s)=s_1$ $\forall \ell\in\enu{n}$ and more generally $m_{L}(s)=2^{-|L|}\sum_{k=0}^{|L|}\binom{|L|}{k}s_{|2k-|L||}$ for every $L\subset \enu{n}$. 
\end{enumerate}
\end{exs}
\subsection{Convergence to a diffusion}
The following theorem provides convergence results for the two processes $X^{(N)}$ and $Y^{(N)}$ as the population size $N$ tends to $+\infty$. The proof, based on Theorem~3.3 of  \cite{EthierNagylaki80}, is postponed to \Sref{secconvproof}.
\begin{thm}\label{thgen}
Assume that hypotheses \Href{Hrecombsym}, \Href{Hscaling}, \Href{Hlinkage} and   \Href{Hpairing} hold.  
\begin{itemize}
\item[(a)] For $i\in\enu{n}$ let $P_{i,s}(x)$ denote a polynomial function in the $n-1$ 
variables $x_{k}(1-x_{k})$ for $k\in\enu{n}\setminus\{i\}$. 
Then, the operator \begin{multline}
\label{expresGn}
\mathcal{G}_{n,s}=\frac{1}{2}\sum_{i=1}^{n}x_i(1-x_i)\frac{\partial^{2}}{\partial_{x_i}\partial_{x_i}}\\+\sum_{i=1}^{n}\Big((1-x_i)\mu_1-x_i\mu_0+(1/2-x_i)x_i(1-x_i)P_{i,s}(x)\Big)\frac{\partial}{\partial_{x_i}}
\end{multline}
with domain $\mathcal{D}(\mathcal{G}_{n,s})=C^2([0,1]^n)$ is closable in $C([0,1]^n)$ and its closure is the generator 
of a strongly continuous semigroup of contractions.
\item[(b)]
If $X^{(N)}_{0}$ converges in distribution in $[0,1]^n$, then $(X^{(N)}_{[N^2t]})_t$ converges in 
distribution in the Skorohod space of c\`adl\`ag functions $D_{[0,1]^n}([0,\infty))$ to a diffusion process $X$  with generator $\mathcal{G}_{n,s}$ 
where the polynomial function $P_{i,s}(x)$ has the following expression: 
\begin{multline}
\label{expresP}
  P_{i,s}(x)=\sum_{A\in\mc{P}(\enu{n}\setminus\{i\})}\big(m_{A\cup\{i\}}(s)-m_{A}(s)\big)\\
\prod_{k\in A}(2x_k(1-x_k))\prod_{\ell\in \enu{n}\setminus\{A\cup \{i\}\}}(1-2x_\ell(1-x_\ell)).
\end{multline}
\item[(c)] For every positive sequence $(t_N)_N$ that converges to $+\infty$, $Y^{(N)}_{[Nt_N]}$ 
converges in distribution to 0. 
\end{itemize}
\end{thm}
\begin{rem}\
\begin{enumerate}
\item The recombination distribution $(r_{I})_{I\subset\enu{n}}$ does not appear in the 
expression for the limiting diffusion.  Nevertheless, the proof of Theorem~\ref{thgen} will 
show that it has an influence on the speed of convergence of the linkage disequilibrium to $0$. 
\item 
The limiting diffusion depends on the assortment parameters only via the quantities $m_A(s)$ for every  $A\subset \enu{n}$. 
A set of assortment parameters for which $${m_{A\cup \{i\}}(s)-m_A(s)<0} \text{ for 
every }i\in\enu{n}\text{ and }A\subset \enu{n}\setminus \{i\}$$ favours homozygous mating with respect to the genotype 
at the $i$-th locus.  It is therefore no surprise that by increasing the value of 
$m_{A\cup \{i\}}(s)-m_A(s)$ for a fixed subset $A$, we increase the value of the $i$-th 
coordinate of the drift at a point $x$ for which $x_i<1/2$ and decrease it at a point 
$x$ for which $x_i>1/2$. 
\end{enumerate}
\end{rem}
\subsection{Another expression for the polynomial term $P_{i,s}(x)$ of the drift}
An expansion of the polynomial function $P_{i,s}(x)$ in terms of the variables 
$x_{k}(1-x_{k})$, ${k\in\enu{n}\setminus \{i\}}$ yields the following expression: 
\begin{equation}\label{Pgeneral}
P_{i,s}(x)=\sum_{L\in\mathcal{P}(\enu{n}\setminus\{i\})}\alpha_{i,L}(s)\prod_{\ell \in L}x_{\ell}(1-x_\ell)
\end{equation}
with 
\begin{equation}
\label{alphageneral}
\alpha_{i,L}(s)=2^{|L|}\sum_{A\subset L}(-1)^{|L|-|A|}(m_{A\cup\{i\}}(s)-m_{A}(s)).
\end{equation}
The details of the proof are provided in \Sref{sectproofdrift}.

The coefficients $\alpha_{i,L}(s)$ can be compactly expressed using difference operators. 
Let us introduce some notation: for a function $f$ defined on the subsets of a finite set $E$ 
and for an element $i$ of $E$, we denote by $\delta_{i}[f]$ the function on ${\mathcal P}(E)$
defined by 
$$\delta_{i}[f](A)=f(A\cup \{i\})-f(A),\ \forall A\in\mathcal{P}(E).$$ 
Since $\delta_{i}\circ \delta_{j}=\delta_{j}\circ \delta_{i}$ for every $i,j\in E$, 
we can, more generally, introduce a difference operator $\delta_{B}$ for each subset 
$B\in \mathcal{P}(E)$ by setting $\delta_{\emptyset}=Id$, and 
$\delta_B=\delta_{b_1}\circ \cdots \circ\delta_{b_r}$ if $B=\{b_1,\ldots,b_r\}$. 
A proof by induction on $|B|$ provides the following formula for $\delta_{B}$:
\begin{equation}
\delta_{B}[f](A)=\sum_{J\subset B}(-1)^{|B|-|J|}f(A\cup J)\quad \forall A\subset E\label{expdiffset}.
\end{equation}
Let $m(s)$ denote the function $A\mapsto m_A(s)$ defined on the subsets of $\enu{n}$. In this notation, for every $A\subset \enu{n}\setminus \{i\}$, 
\begin{equation} m_{A\cup \{i\}}(s)-m_{A}(s)=\delta_{i}[m(s)](A)\text{ and }  \alpha_{i,A}(s)=2^{|A|}\delta_{A\cup \{i\}}[m(s)](\emptyset). 
 \label{expalpdiffset}
\end{equation}

If, for each subset $A$ of loci, the coefficient $m_{A}(s)$ depends only on  
the number of loci in $A$, then it follows from expression \eqref{Pgeneral} that $P_{i,s}(x)$ 
is a symmetric polynomial function, the coefficients of which do not depend on $i$.  
This is the case for instance with the Hamming and additive criteria (see Example~\ref{exmeancoef} 
for the corresponding expressions for $m_{A}(s)$). 
Let us give the expanded form of $P_{i,s}(x)$ for the Hamming criterion:
\begin{equation}\label{PHamming}
P_{i,s}(x)=\sum_{\ell=0}^{n-1}\tilde{\alpha}_{\ell}\sum_{L\subset \enu{n}\setminus\{i\},\  |L|=\ell\ }\prod_{\ell \in L}x_{\ell}(1-x_\ell).
\end{equation}
where $\tilde{\alpha}_{k}(s)=2^{k}\sum_{\ell=0}^{k}(-1)^\ell\binom{k}{\ell}(s_{k-\ell+1}-s_{k-\ell})$.\\
As in the general case, the coefficient $\tilde{\alpha}_{k}(s)$ has a compact expression 
in terms of difference operators. 
Let  $\delta^{(1)}$ denote the forward difference operators: $\delta^{(1)}[s](i)=s_{i+1}-s_{i}$ 
for every ${i\in\enum{0}{n-1}}$.  The forward difference operators of higher orders are 
defined iteratively:  $\delta^{(k+1)}[s]=\delta^{(k)}\circ \delta^{(1)}[s]$ for $k\in\NN^*$. 
With this notation, $\tilde{\alpha}_{k}(s)=2^{k}\delta^{(k+1)}[s](0)$  for $k\in \enum{0}{n-1}$.
\subsection{Comparison with the two-locus Wright-Fisher diffusion \label{secttwolocusWF}}
\cite{EthierNagylaki89} established convergence results for a general 
multiallelic two-locus Wright-Fisher model of a panmictic, monoecious, diploid population of 
$N$ individuals (identified with $2N$ haploids) undergoing mutation and selection.  In their model, a gamete is described by a pair 
$\bb{i}=(i_1,i_2)\in\enu{r_1}\times \enu{r_2}$ where $r_1$ is the number of alleles 
in the first locus and $r_2$ is the number of alleles in the second locus. 
The parameters of their 
model are:
\begin{enumerate}
 \item the viability of a pair of gametes $(\bb{i},\bb{j})$ denoted by $w_{N,\bb{i},\bb{j}}=1-\sigma_{N,\bb{i},\bb{j}}$ with the assumption $\sigma_{N,\bb{i},\bb{j}}=\sigma_{N,\bb{j},\bb{i}}$ and $\sigma_{N,\bb{i},\bb{i}}=0$ for every $\bb{i},\bb{j}\in\enu{r_1}\times \enu{r_2}$ (after viability selection the proportion  of a pair of gametes $(\bb{i},\bb{j})$ is assumed to be  ${\displaystyle P^{*}_{\bb{i},\bb{j}}=\dfrac{w_{N,\bb{i},\bb{j}}P_{\bb{i}}P_{\bb{j}}}{\sum_{\bb{k},\bb{\ell}}w_{N,\bb{k},\bb{\ell}}P_{\bb{k}}P_{\bb{\ell}}}}$ if $P_{\bb{k}}$ denotes the frequency of gametes $\bb{k}$ in the population $\forall \bb{k}\in\enu{r_1}\times \enu{r_2}$);
\item the recombination fraction $c_N$;
\item the probability $(2N)^{-1}\nu^{(i)}_{j,k}$ that the $j$-th allele in the $i$-th locus mutates to the $k$-th allele.
\end{enumerate}
The population at the  generation $t+1$ is obtained by choosing $2N$ gametes uniformly at random with replacement from the pool of gametes of the generation  $t$ after the steps of viability selection, recombination and mutation.

They studied the diffusion approximation under several assumptions on selection and recombination 
coefficients.  In the case of weak selection ($2N\sigma_{N,\bb{i},\bb{j}}$ converges to a 
real number
denoted by $\sigma_{\bb{i},\bb{j}}$ for every $\bb{i},\bb{j}$) and loose linkage ($c_N$ converges 
to a finite limit and $Nc_N$ tends to $+\infty$) they obtained a limiting diffusion for the allelic 
frequencies $(p_1,\ldots,p_{r_1-1},q_1,\ldots,q_{r_2-1})$ of the alleles $1,\ldots,r_1-1$ in the 
first locus and the alleles $1,\ldots,r_2-1$ in the second locus. 
In the case of two alleles at each locus ($r_1=r_2=2$), the generator of the limiting diffusion is
$$\mathcal{L}=\frac{1}{2}p_1(1-p_1)\partial^{2}_{p_1,p_1}+\frac{1}{2}q_1(1-q_1)\partial^{2}_{q_1,q_1}+b_1(p_1,q_1)\partial_{p_1} +b_2(p_1,q_1)\partial_{q_1}$$ with 
 \begin{alignat*}{2}b_1(p_1,q_1)= & \nu^{(1)}_{2,1}(1-p_1)-\nu^{(1)}_{1,2}p_1\\  &-p_1(1-p_1)(1-2p_1)\Big((\sigma_{12,21}+\sigma_{11,22})q_1(1-q_1)+\sigma_{11,21}q_{1}^2+\sigma_{12,22}(1-q_{1})^2\Big)\\
&-2p_1(1-p_1)q_1(1-q_1)\Big(\sigma_{11,12}p_1-\sigma_{21,22}(1-p_1)\Big).\\
 b_2(p_1,q_1)=  & \nu^{(2)}_{2,1}(1-q_1)-\nu^{(2)}_{1,2}q_1\\&-q_1(1-q_1)(1-2q_1)\Big((\sigma_{12,21}+\sigma_{11,22})p_1(1-p_1)+\sigma_{11,12}p_{1}^2+\sigma_{21,22}(1-p_1)^2\Big)\\
&-2q_1(1-q_1)p_1(1-p_1)\Big(\sigma_{11,21}q_1-\sigma_{12,22}(1-q_1)\Big).
\end{alignat*}
Accordingly, the generator $\mathcal{L}$ coincides with  $\mathcal{G}_{2,s}$ if we assume
\begin{itemize}
 \item[(a)] 
that the mutation rates $\nu^{(i)}_{j,k}$ do not depend on the locus $i$ and set $\nu^{(i)}_{1,2}=\mu_{0}$ and $\nu^{(i)}_{2,1}=\mu_{1}$,
\item[(b)] that the coefficients of selection satisfy $\sigma_{11,21}=\sigma_{12,22}$ 
and $\sigma_{11,12}=\sigma_{21,22}$ (second condition of assumption \Href{Hpairing})
\end{itemize} 
and set $\sigma_{\bb{i},\bb{j}}=-\frac{1}{2}s_{\bb{i}-\bb{1},\bb{j}-\bb{1}}$, for every $\bb{i},\bb{j}\in\{1,2\}^{2}$ (with the notation $\bb{1}=(1,\ldots,1)$).\\
This comparison suggests that the effect of  assortative mating  on the genotype evolution of a large population in our model is similar to the effect of weak viability selection  in a diploid Wright-Fisher model with mutation. 
%
\section{Description of the limiting diffusion\label{sectdescrlim}}
This section collects some properties that can be deduced from the form of the generator,
$\mc{G}_{n,s}$, of the limiting diffusion. 
\subsection{The set of generators arising from the model}
\begin{lem}
Any generator on $C^{2}([0,1]^n)$ of the form 
\begin{multline*}\frac{1}{2}\sum_{i=1}^{n}x_i(1-x_i)\frac{\partial^{2}}{\partial_{x_i}\partial_{x_i}}+\\\sum_{i=1}^{n}\Big((1-x_i)\mu_1-x_i\mu_0+(1/2-x_i)x_i(1-x_i)\sum_{L\in\mc{P}(\enu{n}\setminus\{i\})}\alpha_{L\cup\{i\}}\prod_{k\in L}x_{k}(1-x_k)\Big)\frac{\partial}{\partial_{x_i}},
\end{multline*}
where $\{\alpha_A,\ A\subset \enu{n},\ A\neq \emptyset\}$ is a family of real numbers,
can be interpreted as the generator of the diffusion approximation of an $n$-locus Moran model 
as defined in \Sref{sectmodel}.  
\end{lem}
\begin{proof}
We may, for instance, take the following set of assortment parameters 
$\{s_{\bb{i},\bb{j}},\ \bb{i},\bb{j}\in\mc{A}\}$:
\begin{itemize}
 \item 
$s_{\bb{i},\bb{i}}=0$ for every $\bb{i}\in\mc{A}$.
\item $s_{\bb{i},\bb{j}}=\sum_{B\subset L,\ |B|\geq 1}2^{-|B|+1}\alpha_B$ for every $(\bb{i},\bb{j})\in F_L$ and for every  nonempty subset $L$ of $\enu{n}$.
\end{itemize}
Let us check that this family satisfies $2^{|L|-1}\delta_{L}[m(s)](\emptyset)=\alpha_{L}$ for 
every nonempty subset $L$ of $\enu{n}$. 
First, $m_L(s)=\sum_{B\subset L,\ |B|\geq 1}2^{-|B|+1}\alpha_B$. For every $i\in\enu{n}$ and $L\subset\enu{n}\setminus\{i\}$ 
$$\delta_{L\cup \{i\}}[m(s)](\emptyset)=\sum_{A\subset L}(-1)^{|L|-|A|}(m_{A\cup\{i\}}(s)-m_{A}(s))
=\sum_{A\subset L}(-1)^{|L|-|A|}\sum_{B\subset A}2^{-|B|}\alpha_{B\cup \{i\}}.
$$
We invert the double sum and use the formula ${\displaystyle \sum_{A\subset L,\st B\subset A}(-1)^{|L|-|A|}=\un_{\{L=B\}}}$ to obtain:
$$\delta_{L\cup \{i\}}[m(s)](\emptyset)=
\sum_{B\subset L}2^{-|B|}\alpha_{B\cup \{i\}}\un_{\{L=B\}}=2^{-|L|}\alpha_{L\cup \{i\}}.
$$
\end{proof}

In particular, the $n$-locus Moran model with assortative mating based on the Hamming criterion 
allows us to obtain, through diffusion approximation, any generator on $C^2([0,1]^n)$ of the form: 
\begin{multline*}\frac{1}{2}\sum_{i=1}^{n}x_i(1-x_i)\frac{\partial^{2}}{\partial_{x_i}\partial_{x_i}}+\\\sum_{i=1}^{n}\Big((1-x_i)\mu_1-x_i\mu_0+(1/2-x_i)x_i(1-x_i)\sum_{\ell=0}^{n-1}\alpha_{\ell}\sum_{\underset{\text{ s.t.}|L|=\ell}{L\subset\enu{n}\setminus\{i\}}}\prod_{k\in L}x_{k}(1-x_k)\Big)\frac{\partial}{\partial_{x_i}}.
\end{multline*}
To see this, given any sequence $\alpha_0,\ldots,\alpha_{n-1}$ of $n$ reals, we have to find $n+1$ 
real numbers $s_0,\ldots,s_{n}$ such that ${\alpha_{\ell}=2^{\ell}\delta^{(\ell+1)}[s](0)}$. 
These are given by the inversion formula \eqref{deltarel} in the Appendix, from which we see that we may
set $s_0=0$ and $s_k=\sum_{\ell=1}^{k}2^{1-\ell}\binom{k}{\ell}\alpha_{\ell-1}$ 
for $k\in\enu{n}$. 

\subsection{The generator for two groups of loci}
Let us consider a partition of the set of loci into two subgroups, 
$G_1=\enu{k}$ and \linebreak[4] ${G_2=\enum{k+1}{n}}$, say.  We introduce two sets of assortment 
parameters $s^{(1)}$ and $s^{(2)}$ depending on subgroups of loci from $G_1$ and 
from $G_2$ respectively and satisfying 
assumption \Href{Hpairing}.  If we assume that the assortment parameter between two individuals 
of type $\bb{i}$ and $\bb{j}$ is 
${s_{\bb{i},\bb{j}}=s^{(1)}_{\bb{i}_{|G_1},\bb{j}_{|G_1}}+s^{(2)}_{\bb{i}_{|G_2},\bb{j}_{|G_2}}}$ 
for every $\bb{i},\bb{j}\in\mc{A}$, then $m_{L}(s)=m_{L\cap G_1}(s^{(1)})+m_{L\cap G_2}(s^{(2)})$ 
for every subset $L$ of $\enu{n}$.
This implies that  the  first $k$ coordinates of diffusion limit evolve independently of the 
last $n-k$ coordinates and that the generator of the diffusion limit is:
$$\mc{G}_{n,s}=\mc{G}_{k,s_1}\otimes\mc{G}_{n-k,s_2}.$$ 
Therefore, with these choices we can reduce our study to subgroups of loci having the same 
influence on assortment. 
\subsection{Conditions for independent coordinates}
For some patterns of assortment, the allelic frequencies at each locus in a large population 
evolve approximately as independent diffusions:
\begin{prop}\label{propindep}
Assume that the assortment parameters $s=\{s_{\bb{i},\bb{j}},\ \bb{i},\bb{j}\in\mc{A}\}$ 
satisfy the assumption \Href{Hpairing}. 
\begin{enumerate}
\item
The $n$ coordinates of the diffusion associated with the generator $\mc{G}_{n,s}$ are 
independent if and only if the following condition holds:
\begin{description}
\item{\refstepcounter{hyp}(H\thehyp)\label{Hindep}}
for every $i\in\enu{n}$, $m_{L\cup\{i\}}(s)-m_{L}(s)$ does not depend on the choice of the 
subset $L$  of $\enu{n}\setminus\{i\}$.
\end{description}
\item If condition (\Href{Hindep}) holds, the $i$-th coordinate behaves as the one-locus 
diffusion with assortment coefficients $s_0=s_{\bb{1},\bb{1}}$ and 
$s_{1}=s_{u_i,\bb{1}}$ where $u_i=(0_{\{i\}},\bb{1}_{\enu{n}\setminus\{i\}})$ denotes the genotype which differs from the genotype $\bb{1}$ only on the locus $i$; its generator is 
$$\frac{1}{2}x(1-x)\frac{d^2}{dx^2}+\Big((1-x)\mu_1-x\mu_0+
(s_{u_i,\bb{1}}-s_{\bb{1},\bb{1}})(1/2-x)x(1-x)\Big)\frac{d}{dx}.$$
\item 
In particular,
\begin{itemize}
\item[(a)] with the Hamming criterion, $\mc{G}_{n,s}$ is the generator of $n$ independent one-dimensional 
diffusions if and only if the value of $s_{\ell+1}-s_{\ell}$ does not depend on $\ell$;
\item[(b)] with the additive criterion, $\mc{G}_{n,s}$ is the generator of $n$ independent one-dimensional
diffusions if and only if there exists a constant $c$ such that $s_{\ell+1}-s_{\ell}=c(2\ell+1)$ 
for every $\ell\in\enum{0}{n-1}$. 
\end{itemize}
\end{enumerate}
\end{prop}
\begin{proof}
First note that $\mc{G}_{n,s}$ is the generator of $n$ independent diffusions if and only if 
the polynomial term $P_{i,s}(x)$ is a constant function for every $i\in\enu{n}$. 
\begin{enumerate}
\item According to the formula \eqref{expresP}, the polynomial term $P_{i,s}(x)$ is a constant 
function for every $i\in\enu{n}$ whenever condition \Href{Hindep} holds.
Conversely, assume that  the polynomial term $P_{i,s}(x)$ is a constant function 
for every $i\in\enu{n}$. By  formulae \eqref{Pgeneral} and \eqref{expalpdiffset},  $\delta_{L}[m(s)](\emptyset)=0$ for every subset $L$ of $\enu{n}$ 
having at least two elements.  We derive, from the inversion formula \eqref{invdifset} stated in the 
Appendix, that for every subset $A\in \mc{P}(\enu{n})$,  
$$\delta_{i}[m(s)](A)=\sum_{B\subset A}\delta_{B\cup \{i\}}[m(s)](\emptyset)=\delta_{i}[m(s)](\emptyset). $$
Therefore, condition \Href{Hindep} is satisfied. 
\item With the Hamming criterion,  $m_A(s)=s_{|A|}$ and condition \Href{Hindep} is equivalent 
to $$s_{k+1}-s_{k}=s_{1}-s_{0} \text{ for every }k\in\enum{1}{n-1}.$$
\item With the additive criterion, for a subset $L$ with $\ell$ elements 
$m_L(s)=2^{-\ell}\sum_{j=0}^{\ell}\binom{\ell}{j}s_{|2j-\ell|}$. 
After some computation, we obtain for $i\not\in L$,  
\begin{multline}  \label{deltamadd} 
m_{L\cup\{i\}}(s)-m_{L}(s)=2^{-\ell}\sum_{j=0}^{\ell}\binom{\ell}{j}(s_{|2j-\ell+1|}-s_{|2j-\ell|})\\
  =\begin{cases}2^{-\ell}\underset{j=1}{\overset{\frac{\ell+1}{2}}{\sum}}\binom{\ell}{\frac{\ell+1}{2}-j}\delta^{(2)}[s](2j-2) & \text{ if }\ell \text{ is odd},\\
2^{-\ell}\Big(\underset{j=1}{\overset{\frac{\ell}{2}}{\sum}}\binom{\ell}{\frac{\ell}{2}-j}\delta^{(2)}[s](2j-1)+\binom{\ell}{\frac{\ell}{2}}\delta^{(1)}[s](0)\Big) & \text{ if }\ell \text{ is even}.
\end{cases}
\end{multline}
It follows from \eqref{deltamadd} that for every $c\in\R$, the system defined 
by $$m_{L\cup\{i\}}(s)-m_{L}(s)=c\text{ for every }i\in\enu{n} 
\text{ and }L\subset\enu{n}\setminus\{i\}$$
has a unique solution which is $\delta^{(1)}[s](k)=c(2k+1)$ for every $k\in\enum{0}{n-1}$. 
\end{enumerate}
\end{proof}
\subsection{Behaviour at the boundaries\label{boundariemultilocus}}

In this section the trajectories of the coordinates of the limiting diffusion are compared with 
those of one-dimensional diffusions in order to investigate whether an allele can be 
(instantaneously) fixed at one of the loci. 

Consider the stochastic differential equations associated with the generator $\mc{G}_{n,s}$:
\begin{equation} \label{sde}
dx_t(i)=\sqrt{x_t(i)(1-x_t(i))}dW_t(i)+b_{i}(x_t)dt\quad \forall i\in\enu{n},
\end{equation}
where 
$(W_t(1))_{t\geq 0}$,\ldots, $(W_t(n))_{t\geq 0}$ denote $n$ independent standard Brownian motions,
and $$b_{i}(x)=\mu_1(1-x(i))-\mu_0x(i)+(1/2-x(i))x(i)(1-x(i))P_{i,s}(x)\text{ for }i\in\enu{n}.$$
Theorem~1 of \cite{Yamada} 
ensures pathwise uniqueness for the stochastic differential equation 
\eqref{sde}, since the drift is Lipschitz and the diffusion matrix is a diagonal matrix of the 
form  
$$\sigma(x)=\text{diag}(\sigma_{1}(x(1)),\ldots,\sigma_n(x(n))),$$ 
where the functions $\sigma_i$ are $1/2$-H\"{o}lder continuous functions. 

The following proposition shows that, just as for the one-locus case, the 
boundary behaviour of the solution to \eqref{sde} depends only on the values of the 
mutation rates $\mu_0$ and $\mu_1$.
\begin{prop}\label{boundary}
Let $(x_t)_{t\geq 0}$ denote a solution of the stochastic differential equation 
\eqref{sde} starting from a point $x_0\in]0,1[^n$. 
\begin{itemize}
\item[(i)] If $\mu_1=\mu_0=0$ then  the diffusion process $(x_t)_t$ exits from $]0,1[^n$ in a finite time almost surely. 
\item[(ii)] If $\mu_1=0$ and $\mu_0>0$ then  each coordinate of $(x_t)_t$ reaches the point $0$ in a finite time almost surely. 
\item[(iii)] If $0<\mu_1<1/2$ then $0$ is attainable for each coordinate of the diffusion process: 
  $$\PP[\exists t>0,\ x_{t}(i)=0]>0\quad \forall i\in\enu{n}.$$
\item[(iv)] If $\mu_1\geq 1/2$ then $0$ is inaccessible for each coordinate of the diffusion 
process: $$\PP[\exists t>0,\ x_{t}(i)=0]=0\text{ and }\PP[\lim_{t \rightarrow +\infty}x_t(i)=0]=0\text{ for every }i\in\enu{n}.$$
\end{itemize}
Similar statements to (ii), (iii) and (iv) hold for the point $1$ on exchanging the r\^oles
of $\mu_1$ and $\mu_0$. 
\end{prop}
\begin{proof}
Let $i\in\enu{n}$. On $[0,1]^n$ the polynomial function $P_{i,s}$ is bounded above by   
\[ M_{i}^{+}=\sum_{A\subset\enu{n}\setminus\{i\}}2^{-|A|}\max\big\{m_{A\cup \{i\}}(s)-m_{A}(s),0
\big\}
 \]
and is bounded below by
\[ M_{i}^{-}=-\sum_{A\subset\enu{n}\setminus\{i\}}2^{-|A|}\max\big\{-(m_{A\cup \{i\}}(s)-m_{A}(s)),0
\big\}.\]
Let $b_{i}^{+}$ and $b_{i}^{-}$ denote the  functions defined on $[0,1]$ by 
\begin{eqnarray*} b_{i}^{+}(u)&=&\mu_1(1-u)-\mu_0u
 +(1/2-u)u(1-u)(M_{i}^{+}\un_{\{u<1/2\}}+M_{i}^{-}\un_{\{u>1/2\}}),\\
b_{i}^{-}(u)&=&\mu_1(1-u)-\mu_0 u +(1/2-u)u(1-u)(M_{i}^{+}\un_{\{u>1/2\}}+M_{i}^{-}\un_{\{u<1/2\}}),
\end{eqnarray*}
for every  $u\in[0,1]$.
 For every $i\in\enu{n}$,  pathwise uniqueness holds for the following two stochastic differential equations: 
\begin{equation}
 du_t=\sqrt{u_t(1-u_t)}dW_t(i)+b^{+}_{i}(u_t)dt\label{sdeplus}
\end{equation}
and 
 \begin{equation}
 du_t=\sqrt{u_t(1-u_t)}dW_t(i)+b^{-}_{i}(u_t)dt.\label{sdeminus}
\end{equation}
Let $\xi^{+}_{t}(i)$ and $\xi^{-}_{t}(i)$ be the solution starting from  $x_0(i)$ of the stochastic differential equations \eqref{sdeplus} and \eqref{sdeminus} respectively. 
As the $i$-th coordinate of the drift is bounded above by $b_{i}^{+}$ and is bounded below by  $b_{i}^{-}$, the comparison theorem of \cite{Ikeda} ensures that 
the following inequalities hold with probability one:
\begin{equation}\label{comparison}\xi^{-}_{t}(i)\leq x_t(i)\leq \xi^{+}_{t}(i)\quad  \forall t\geq 0,\ \forall i\in\enu{n}.
\end{equation}
The  nature of the points $0$ and $1$ as described by the Feller classification is the same for 
$(\xi^{-}_t(i))_t$ and $(\xi^{+}_{t}(i))_t$ and depends only on $\mu_1$ and $\mu_0$. 
To describe their behaviours near $0$,  let $\tau^{\pm,i}_z(a,b)$ denote the first time  
the process $(\xi^{\pm}_t(i))_t$, starting from $z$, exits $(a,b)$ for $0\leq a<z <b \leq 1$.
\begin{enumerate}
\item If $\mu_1=\mu_0=0$ then $0$ and $1$ are absorbing points;  $(\xi^{\pm}_t(i))_t$ reaches $0$ or $1$ in a finite time with probability one and 
$$\PP\left[\lim_{t\rightarrow \tau^{\pm,i}_z(0,1)}\xi^{\pm}_{t}(i)=0\right]
=\dfrac{\int_{z}^{1}\exp\Big(-\int_{1/2}^{x}\frac{2b^{\pm}_i(u)}{u(1-u)}du\Big)dx}{\int_{0}^{1}\exp\Big(-\int_{1/2}^{x}\frac{2b^{\pm}_i(u)}{u(1-u)}du\Big)dx}.$$
\item If $\mu_1=0$ and $\mu_0>0$ then $0$ is the only absorbing point and $(\xi^{\pm}_t(i))_t$ reaches $0$ in a finite time with probability one.
\item If $0<\mu_1<1/2$ and $\mu_0>0$ then 0 is attainable: for every $0<z<b<1$, 
$$\PP\Big[\tau^{\pm,i}_z(0,b)<+\infty 
\text{ and }\lim_{t\rightarrow \tau^{\pm,i}_z(0,b)}\xi^{\pm}_t(i)=0\Big]>0.$$
\item If $\mu_1\geq 1/2$ and $\mu_0>0$ then $0$ is inaccessible: for every $0<z<1$, 
$$\PP_z[\exists t>0,\ \xi^{\pm}_t(i)=0]=0\text{ and }\PP_z\big[ \lim_{t\rightarrow +\infty}\xi^{\pm}_t(i)=0 \big]=0.$$
\end{enumerate}
Similar properties hold for the behaviour near the point 1.\\
These properties are sufficient to prove the boundary behaviour claimed for $(x_t)_t$. 
Since ${x_t(i)\leq \xi^{+}_{t}(i)}$ for every $t\geq 0$, if $(\xi^{+}_{t}(i))_t$ reaches $0$ in a 
finite time then so must $(x_{t}(i))_t$.  Similarly, if $0$ is attainable 
for $(\xi^{+}_{t}(i))_t$ then $0$ is also attainable for $(x_{t}(i))_t$. 
In the same way, since $x_t(i)\geq \xi^{-}_{t}(i)$ for every $t\geq 0$, if $0$ is inaccessible 
for $(\xi^{-}_{t}(i))_t$ then  $0$ is also inaccessible for $(x_{t}(i))_t$. 

It remains to prove that $(x_t)_t$ exits from $]0,1[$ in a finite time with probability one 
if $\mu_1=\mu_0=0$.  Let $\epsilon>0$ be small enough that $x_0\in[\epsilon,1-\epsilon]^n$.  
The diffusion $x_t$ exits from  the compact $[\epsilon,1-\epsilon]^n$ in a finite time with 
probability one.  Let $x_{\epsilon}$ be a point on the boundary of $[\epsilon,1-\epsilon]^n$. 
There exists $i\in\enu{n}$ such that $x_{\epsilon}(i)\in\{\epsilon,1-\epsilon\}$.  For $z\in]0,1[$, 
set $\phi_{i}^{\pm}(z) :=\PP_{z}[\lim_{t\rightarrow \tau^{\pm,i}(0,1)}\xi^{\pm}_t(i)=0]$. 
By the comparison theorem applied to the solutions of the stochastic differential equations 
\eqref{sde}, \eqref{sdeplus} and \eqref{sdeminus} starting from $x_{\epsilon}$, the probability 
that the solution of \eqref{sde} starting at $x_{\epsilon}$ reaches the boundary of $[0,1]^n$ 
in a finite time is greater than $\phi_{i}^{+}(\epsilon)$ if $x^{\epsilon}_i =\epsilon$ 
and is greater than $1-\phi_{i}^{-}(1-\epsilon)$ if $x^{\epsilon}_i =1-\epsilon$. 
By the strong Markov property, the probability that $(x_t)$ reaches the boundary in a 
finite time is greater than $\min\{\min(\phi_{i}^{+}(\epsilon),1-\phi_{i}^{-}(1-\epsilon)), 
i\in\enu{n}\}$ for every $\epsilon>0$.  Therefore $(x_t)_t$ reaches the boundary in a finite 
time with probability one. 
\end{proof}
\section{The stationary measure of the limiting diffusion\label{sectstatmeas}}
%
%
\subsection{Existence of a stationary distribution and an expression for its density}
As in the one-locus case, when the mutation rates are strictly positive, the diffusion has 
a reversible stationary distribution:
\begin{prop}\label{existinv}
 Assume that the hypothesis \Href{Hpairing} holds and that the mutation rates $\mu_0$ and $\mu_1$ 
are strictly positive.  Set $\tilde{s}_{\bb{i},\bb{j}}=s_{\bb{i},\bb{j}}-s_{\bb{1},\bb{1}}$ 
for every pair of types $\bb{i},\bb{j}\in\mc{A}$.
The diffusion with generator $\mc{G}_{n,s}$ has a unique reversible stationary distribution 
which has the following density with respect to the Lebesgue measure on $[0,1]^n$: 
$$g_{n,\mu,s}(x)=C_{n,\mu,s}\prod_{i=1}^{n}x_{i}^{2\mu_1-1}(1-x_i)^{2\mu_0-1}\exp(H_{n,s}(x))$$ where 
\begin{itemize}
 \item $\displaystyle{H_{n,s}(x)=\frac{1}{2}\underset{L\subset\enu{n},\ |L|\geq 1}{\sum}m_L(\tilde{s})\prod_{\ell\in L}(2x_{\ell}(1-x_{\ell}))\prod_{k\in\enu{n}\setminus L}(1-2x_{k}(1-x_k))}$;
\item 
$C_{n,\mu,s}$ is chosen so that $\displaystyle{\int_{[0,1]^n}g_{n,\mu,s}(x_1,\ldots,x_n)dx_1\cdots dx_n=1}$.
\end{itemize}
\end{prop}
\begin{rem}
An expansion of the polynomial function $H_{n,s}$ yields: 
$$H_{n,s}(x)=\sum_{L\subset\enu{n},\ |L|\geq 1}2^{|L|-1}\delta_{L}[m(s)](\emptyset)\prod_{\ell\in L}x_{\ell}(1-x_{\ell}).$$
\end{rem}
\begin{proof}[Proof of Proposition~\ref{existinv}]
Let $\mathcal{G}_{n,0}$ denote the generator of the limiting diffusion in the random mating 
case ($s_{\bb{i},\bb{j}}=0$ for every $\bb{i},\bb{j}\in\mc{A}$). The diffusion associated 
with this generator is ergodic and has a reversible stationary distribution $m_{\mu,0}$ 
which is the product of Beta distributions: 
$m_{\mu,0}:=(\text{Beta}(2\mu_0, 2\mu_1))^{\otimes n}$.
In the general case, the generator $\mathcal{G}_{n,s}$ can be decomposed as 
$$\mathcal{G}_{n,s}=\mathcal{G}_{n,0}+\frac{1}{2}\sum_{i=1}^{n}x_{i}(1-x_{i})\partial_i h(x)\partial_i$$
where $$h(x)=\sum_{L\subset\enu{n},\ |L|\geq 1}2^{|L|-1}\delta_{L}[m(s)](\emptyset)\prod_{\ell\in L}x_{\ell}(1-x_{\ell}).$$
Therefore, as explained in \cite{EthierNagylaki89}, we can apply a result of  \cite{FukushimaStroock} 
to deduce that
the diffusion associated with $\mc{G}_{n,s}$ has a unique reversible 
stationary distribution $m_{\mu,s}$ given by 
$$m_{\mu,s}(dx)=C\exp(h(x))m_{\mu,0}(dx),$$ 
where $C$ is chosen so that $m_{\mu,s}$ is a probability distribution. 
\end{proof}
\subsection{Description of the density of the stationary measure}
We analyse the density of the stationary measure under two supplementary assumptions:
\begin{hyp}\label{Hmutsym} The two mutation rates $\mu_0$ and $\mu_1$ are assumed to be equal 
to a strictly positive real number $\mu$.
\end{hyp}
\begin{hyp}\label{Hassortdens}  For every $L\in\mc{P}(\enu{n})$, $m_L(s)$ depends only on $|L|$.
We write $m(\ell)$ for the common value of $m_L(s)$ for $L\in \mc{P}(\enu{n})$ such that $|L|=\ell$.  
\end{hyp}
Assumption \Href{Hassortdens} holds if  the assortment parameters satisfy the Hamming 
criterion or the  additive criterion. 

Under the hypotheses \Href{Hrecombsym}, \Href{Hscaling}, \Href{Hlinkage}, \Href{Hpairing}, 
\Href{Hmutsym} and \Href{Hassortdens}, the density of the invariant measure can be written as
$g_{n,\mu,s}(x)=C\exp(h_{n,\mu,s}(x))$ with  
$$h_{n,\mu,s}({x})= (2\mu-1)\sum_{i=1}^{n}\ln(\rho(x_i))+\sum_{\ell=0}^{n-1}\alpha_{\ell}\sum_{L\subset\enu{n},\ |L|=\ell+1}\prod_{k\in L}\rho(x_k), $$
where  $\rho(x_i)=x_i(1-x_i)$ and  $\alpha_{\ell}=2^{\ell}\delta^{(\ell+1)}[m](0)$.

The study of the invariant measure in the one-locus case already provides a precise image of the 
graph of $g_{n,\mu,s}$ when the $n$ coordinates of the diffusion are independent, that is when the 
assortment coefficients are chosen so that $$\text{for every }\ell\in\{0,\ldots,n-1\},\  
m(\ell+1)-m(\ell)=m(1)-m(0).$$  There are then at least four different types of graph depending 
on the respective contributions to allelic diversity of mutations ($\mu>1/2$ or $0<\mu<1/2$) and  
assortment parameters ($m(1)-m(0)$ smaller than $| 8\mu-4 |$ or not) as shown in Fig.~\ref{loiinv1}. 

Proposition \ref{descriunptcrit} gives conditions on the assortment parameters 
under which $(1/2,\ldots,1/2)$ is the only critical point of the density, as
in the random mating case.   Proposition \ref{descloirev} deals with  situations far from the 
random mating case (the proofs are postponed to \Sref{subsectdproofstation}). 
\begin{prop}\label{descriunptcrit}
We assume that the hypotheses \Href{Hrecombsym}, \Href{Hscaling}, \Href{Hlinkage}, 
\Href{Hpairing}, \Href{Hmutsym} and \Href{Hassortdens} hold.
Set  $V_n=2\mu-1+2^{-(n+1)}\sum_{k=0}^{n-1}\binom{n-1}{k}\delta^{(1)}[m](k)$. 
\begin{enumerate}
 \item If $V_{n}>0$, then $(1/2,\ldots,1/2)$ is a local maximum of $g_{n,\mu,s}$.
\item If $V_{n}<0$, then $(1/2,\ldots,1/2)$ is a local minimum of $g_{n,\mu,s}$.
\item If $\mu>1/2$ and if  $\delta^{(1)}[m](\ell)\geq -(8\mu-4)$ $\forall \ell\in\enum{0}{n-1}$, then $(1/2,\ldots,1/2)$ is a global maximum and is the only critical point of $g_{n,\mu,s}$.  
\item If $0<\mu<1/2$ and if  $\delta^{(1)}[m](\ell)\leq -(8\mu-4)$ $\forall \ell\in\enum{0}{n-1}$, then $(1/2,\ldots,1/2)$ is a global minimum and is the only critical point of $g_{n,\mu,s}$. 
\end{enumerate}
\end{prop}
\begin{ex}
 Let us consider the additive criterion with the assortment sequence $s_{\ell}=b\ell$ for 
$\ell\in\enum{0}{n}$. Then $\delta^{(1)}[m](\ell)=2^{-\ell}\binom{\ell}{\ell/2}b\un_{\{\ell 
\text{ is even}\}}$. As $2^{-\ell}\binom{\ell}{\ell/2}$ is a strictly decreasing sequence smaller 
than $1$, $b<0$ implies $V_n>2\mu-1+\frac{1}{8}b$. Thus, it follows from 
Proposition~\ref{descriunptcrit} that if $\mu>1/2$ and $b\geq -8(2\mu-1)$, the point 
$(1/2,\ldots,1/2)$ is a local maximum of $g_{n,s,\mu}$. Let us note that if we consider the 
same sequence $s_{\ell}=b\ell$  but with the Hamming criterion, then for $\mu>1/2$ and 
$b<-4(2\mu-1)$, $(1/2,\ldots,1/2)$ is a local minimum of $g_{n,s,\mu}$.  
\end{ex}
\begin{rem}
The statement of Proposition~\ref{descriunptcrit} can be easily extended to a family of 
assortment parameters for which \Href{Hassortdens} does not hold: $V_n$ must be replaced 
by $$V_{n,i}=2\mu-1 +2^{-(n+1)}\sum_{B\subset\enu{n}\setminus \{i\}}\delta_{i}[m(s)](B)$$ for 
every $i\in\enu{n}$ and the conditions on $\delta[m](\ell)$ in assertions 3 and 4 
are replaced by a condition on  $\delta_{i}[m(s)](A)$ for every $i\in\enu{n}$ and 
$A\in \mc{P}(\enu{n}\setminus\{i\})$. 
\end{rem}
The following proposition describes the properties of the critical points of the density in 
two cases, (1) $\mu>1/2$ and a condition on the assortment parameters which strongly favours 
mating between individuals carrying similar types: 
$$\delta^{(1)}[m](n-1)\leq \delta^{(1)}[m](n-2)\leq\ldots \leq  \delta^{(1)}[m](0)\leq 0 \text{ and  } {\delta^{(1)}[m](n-2)<0},$$
and (2) $0<\mu<1/2$ and a condition on the assortment parameters which strongly favours mating 
between individuals with dissimilar types: 
$$\delta^{(1)}[m](n-1)\geq \delta^{(1)}[m](n-2)\geq \ldots \geq  \delta^{(1)}[m](0)\geq 0\text{ and  }{\delta^{(1)}[m](n-2)>0}.$$ 
To simplify the statement, the description is limited to the hypercube $[0,1/2]^{n}$. The description
on the whole space $[0,1]^n$ can be deduced from this since
 $g_{n,\mu,s}(x)$ is invariant if we replace any coordinate $x_i$ with $1-x_i$. 
\begin{prop}\label{descloirev}
Assume that conditions \Href{Hrecombsym}, \Href{Hscaling}, \Href{Hlinkage}, 
\Href{Hpairing}, \Href{Hmutsym} and \Href{Hassortdens} hold.  Set 
$$V_n=2\mu-1+2^{-(n+1)}\sum_{k=0}^{n-1}\binom{n-1}{k}\delta^{(1)}[m](k).$$
\begin{enumerate}
\item \underline{Case $\mu>1/2$}. Assume furthermore that:\\
 $\delta^{(1)}[m](n-1)\leq \delta^{(1)}[m](n-2)\leq\ldots \leq  \delta^{(1)}[m](0)\leq 0$ and  ${\delta^{(1)}[m](n-2)<0}$. 
 \begin{enumerate}
 \item If $V_n > 0$ then $(1/2,\ldots,1/2)$ is a global maximum and is the only critical point of the density $g_{n,\mu,s}$. 
\item If $V_n < 0$ then 
\begin{enumerate}
 \item $g_{n,\mu,s}$ has a local minimum at $(1/2,\ldots,1/2)$. 
\item In $[0,1/2]^n$, $g_{n,\mu,s}$ takes its maximum value at a unique point of the form $(\xi_0,\ldots,\xi_0)$.
\item The other critical points of $g_{n,\mu,s}$ in $[0,1/2]^n$ are saddle points: for every \linebreak[4]${\ell\in\enu{n-1}}$,  $g_{n,\mu,s}$ has $\binom{n}{\ell}$ saddle points of index $n-\ell$ in $[0,1/2]^n$. The saddle points of index $n-\ell$ have $\ell$ coordinates equal to $1/2$ and the other coordinates have the same value  denoted by $\xi_{\ell}$.
\item The relative positions of the coordinates of the critical points in $[0,1/2]^n$  satisfy  $0<\xi_{n-1}<\cdots<\xi_0<1/2$. 
\item The value of $g_{n,\mu,s}$ is the same at any saddle point of index $n-\ell$ and decreases as $\ell$ increases. 
\end{enumerate}
\end{enumerate}
\item \underline{Case $0< \mu < 1/2$}. Assume furthermore that:
 $$\delta^{(1)}[m](n-1)\geq \delta^{(1)}[m](n-2)\geq \ldots \geq  \delta^{(1)}[m](0)\geq 0\text{ and  }{\delta^{(1)}[m](n-2)>0}.$$ 
 \begin{enumerate}
 \item If $V_n < 0$ then $(1/2,\ldots,1/2)$ is a global minimum and is the only critical point of the density $g_{n,\mu,s}$. 
\item If $V_n > 0$ then 
\begin{enumerate}
\item $g_{n,\mu,s}$ has a local maximum at $(1/2,\ldots,1/2)$. 
\item In $[0,1/2]^n$, $g_{n,\mu,s}$ takes its minimum value at a unique point of the form \linebreak[4]$(\xi_0,\ldots,\xi_0)$.
\item The other critical points of $g_{n,\mu,s}$ in $[0,1/2]^n$ are saddle points: for every \linebreak[4]$\ell\in\enu{n-1}$,  $g_{n,\mu,s}$ has $\binom{n}{\ell}$ saddle points of index $\ell$ in $[0,1/2]^n$. The saddle points of index $\ell$ have $\ell$ coordinates equal to $1/2$ and the other coordinates have the same value  denoted by $\xi_{\ell}$.
\item The relative positions of the coordinates of the critical points in $[0,1/2]^n$  satisfy  $0<\xi_{n-1}<\cdots<\xi_0<1/2$. 
\item The value of $g_{n,\mu,s}$ is the same at any saddle point of index $n-\ell$ and increases as $\ell$ increases. 
\end{enumerate}
\end{enumerate}
\end{enumerate}
\end{prop} 
\begin{rem} \label{remptcrit}\
 \begin{enumerate}
\item 
 $\xi_0=1/2 - 1/2\sqrt{1-4\lambda_0}$ where $\lambda_0$ is the unique solution in $]0,1/4[$ of the equation:
\begin{equation}
 2\mu-1+x\sum_{i=0}^{n-1}\delta^{(1)}[m](i)\binom{n-1}{i}(2x)^{i}(1-2x)^{n-1-i}=0 \label{eqcrit}\tag{$\mc{E}_0$}
\end{equation}
More generally, for every $\ell\in\enum{0}{n-1}$, 
 $\xi_\ell=1/2 - 1/2\sqrt{1-4\lambda_\ell}$ where $\lambda_\ell$ is the unique solution in $]0,1/4[$ of the equation:
\begin{equation}
 2\mu-1+x\sum_{i=0}^{n-1}B_{n-1,\ell,i}(2x)\delta^{(1)}[m](i)=0 \label{eqcritn}\tag{$\mc{E}_\ell$}
\end{equation}
and $\displaystyle{B_{n,\ell,i}(x)=2^{-\ell}\sum_{j=\max(0,i-n+\ell)}^{\min(i,\ell)}\binom{\ell}{j}\binom{n-\ell}{i-j}x^{i-j}(1-x)^{n-\ell-(i-j)}}$.\\
Let us note that $(B_{n,\ell,i}(x))_{i=0,...,n}$ are positive on $]0,1[$ and their sum is equal to 1. 
\item The assumption that $\delta^{(1)}[m](i)$ is a decreasing function of $i$ cannot be removed since one can find examples of assortment parameters satisfying $\delta^{(1)}[m](i)<0$ for every ${i\in\enum{0}{n-1}}$ and such that:
\begin{itemize}
 \item[(a)] 
 $\mu>1/2$,  $V_{n}>0$, but  $(1/2,\ldots,1/2)$ is not the only local maximum, 
\item[(b)] $V_n<0$ and 
$g_{n,\mu,s}$ has more  than $2^n$  local maxima.
\end{itemize}
\item
If $x_i$ is the proportion of the population with allele  0 at the $i$-th locus, 
$2x_i(1-x_i)$ is the probability that two individuals sampled at random from the population
carry different alleles at the $i$th locus.
The density function of the reversible measure takes its maximum value at a point $x$ such that for each $i\in\enu{n}$, $x_i(1-x_i)=\lambda_0$. 
\end{enumerate}
\end{rem}
\begin{ex}\label{exquadraticassort}Let us consider a quadratic sequence of parameters  $s_{\ell}=s_0-(b\ell+c\ell^2)$ ${\forall \ell\in\enum{0}{n}}$ and let us define the assortment with this sequence by  means of the  Hamming criterion. 
If $c>0$, $b+c\geq 0$ and $\mu>1/2$ then $g_{n,\mu,s}$ has $3^n$ critical points if and only if $b+nc > 8\mu-4$. 
In this case, $\lambda_0=n^{-1/2}\sqrt{\frac{2\mu-1}{4c}}+O(n^{-1})$. If $h_{n,k}$ denotes the value of the function $h_{n,\mu,s}$ at a critical point of index $n-k$ then   $h_{n,0}-h_{n,n}\underset{n\rightarrow +\infty}{\sim} \frac{c}{8}n^2$ and $h_{n,0}-h_{n,1}\underset{n\rightarrow+\infty}{\sim} n^{1/2}1/2\sqrt{c(2\mu-1)}$ (see Appendix \ref{annexexquadraticassort} for more details). 
\end{ex}
%
\subsection{Graphs of the density and simulations of trajectories in the two and three locus cases}
Figures~\ref{figdensity020406} to \ref{figdens2continuum} show graphs of the  density of the 
reversible stationary measure in the two-locus case for $\mu=0.6$ and for several values of
$s_1-s_0$ and $s_2-s_1$, the assortative mating being defined by the  Hamming distance. 
Figures~\ref{figdensity020406} and~\ref{figdensity0226} illustrate the two  situations 
considered in Proposition~\ref{descloirev} when $\mu>1/2$.  When $s_1-s_0=0$, the  density may 
have  a continuum of critical points as in Fig.~\ref{figdens2continuum}; this corresponds to 
a case in which the assumption $\delta^{(1)}[m](n-2)<0$ of Proposition~\ref{descloirev} 
is not satisfied. 

To illustrate the evolution of the $0$-allelic frequency when $\mu>1/2$ and the assortative mating 
strongly favours pairing between similar types,  simulations were run in a population of size 
$N=10^3$ with the two-locus model (Fig.~\ref{figsimul2}) and with the three-locus model 
(Fig.~\ref{figsimul3}). For these simulations, every individual initially carries the allele $0$ 
at every locus, recombination occurs independently at each locus and the assortative mating is defined by the Hamming criterion. The trajectory is plotted at intervals of size $N$ between the iterations $N^2$ 
and $33N^2$.  To help to visualize the evolution, the colour of the plot changes every 
$\frac{1}{2}N^2$ iterations.
The form of the density of the stationary measure here is highly reminiscent of that of the fitness 
landscapes studied in the adaptive evolution literature in modelling
additive traits under frequency dependent intraspecific competition, 
see e.g. \cite{schneider2007}
and references therein.  In the deterministic setting the existence of
multiple `long term equilibria' renders the behaviour of the system very sensitive to assumptions
about the initial conditions.  In our setting, the presence of genetic drift is sufficient for the
population to (eventually) explore the neighbourhoods of all the maxima, irrespective of its 
starting point. The time spent by the population in the neighbourhood of a maximum depends on the assortment parameters (Fig.~\ref{figsimul3}a and \ref{figsimul3}b). 
\begin{figure}[htb]
\begin{minipage}[t]{.46\linewidth}
\centering \includegraphics[scale=0.4]{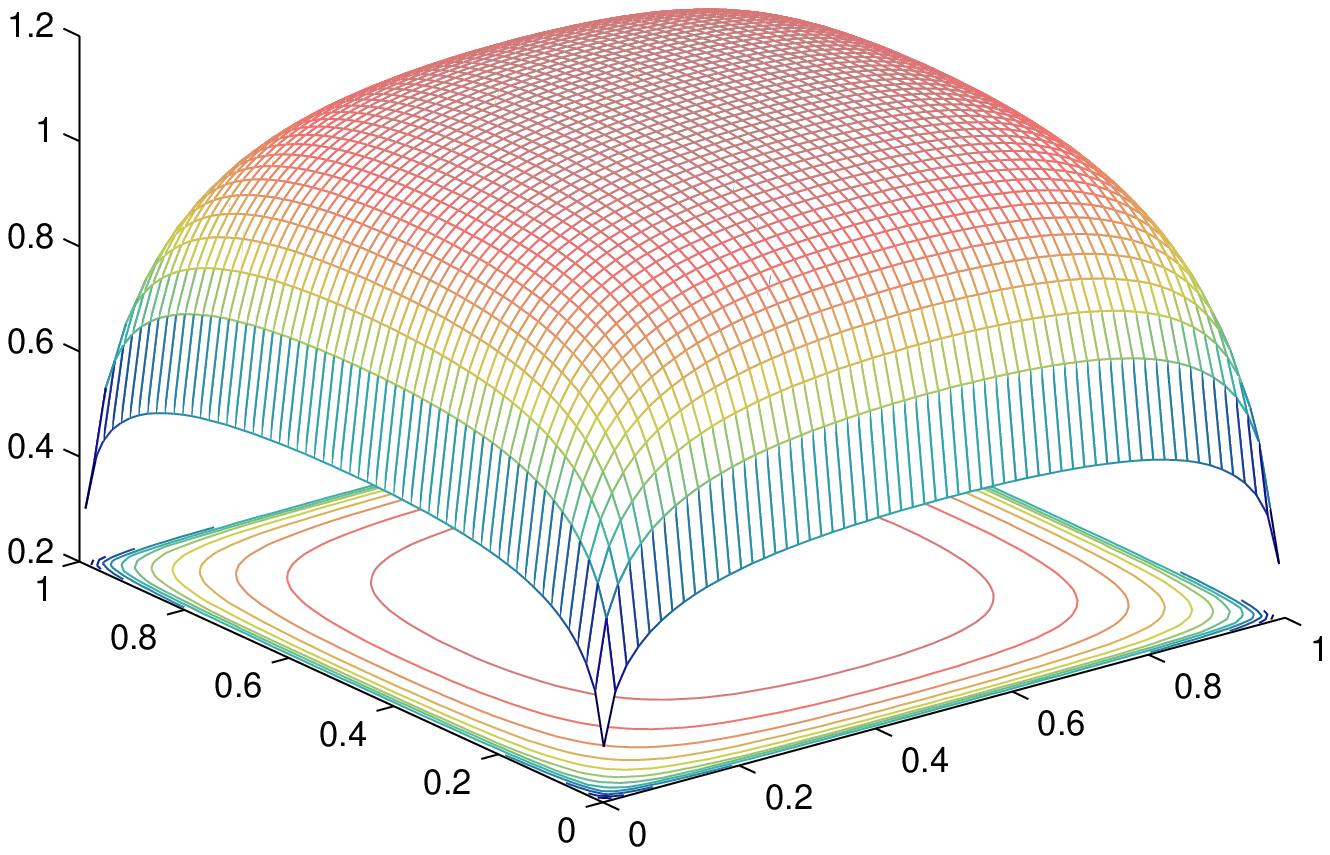}
\caption{Graph of $g_{2,\mu,s}$ when $\mu=0.6$, $s_1-s_0=-0.4$ and $s_2-s_1=-0.6$ so that the point $(1/2,1/2)$ is the only critical point of the density $g_{2,s,\mu}$. \label{figdensity020406}}
\end{minipage}\hfill
\begin{minipage}[t]{.46\linewidth}
\centering \includegraphics[scale=0.4]{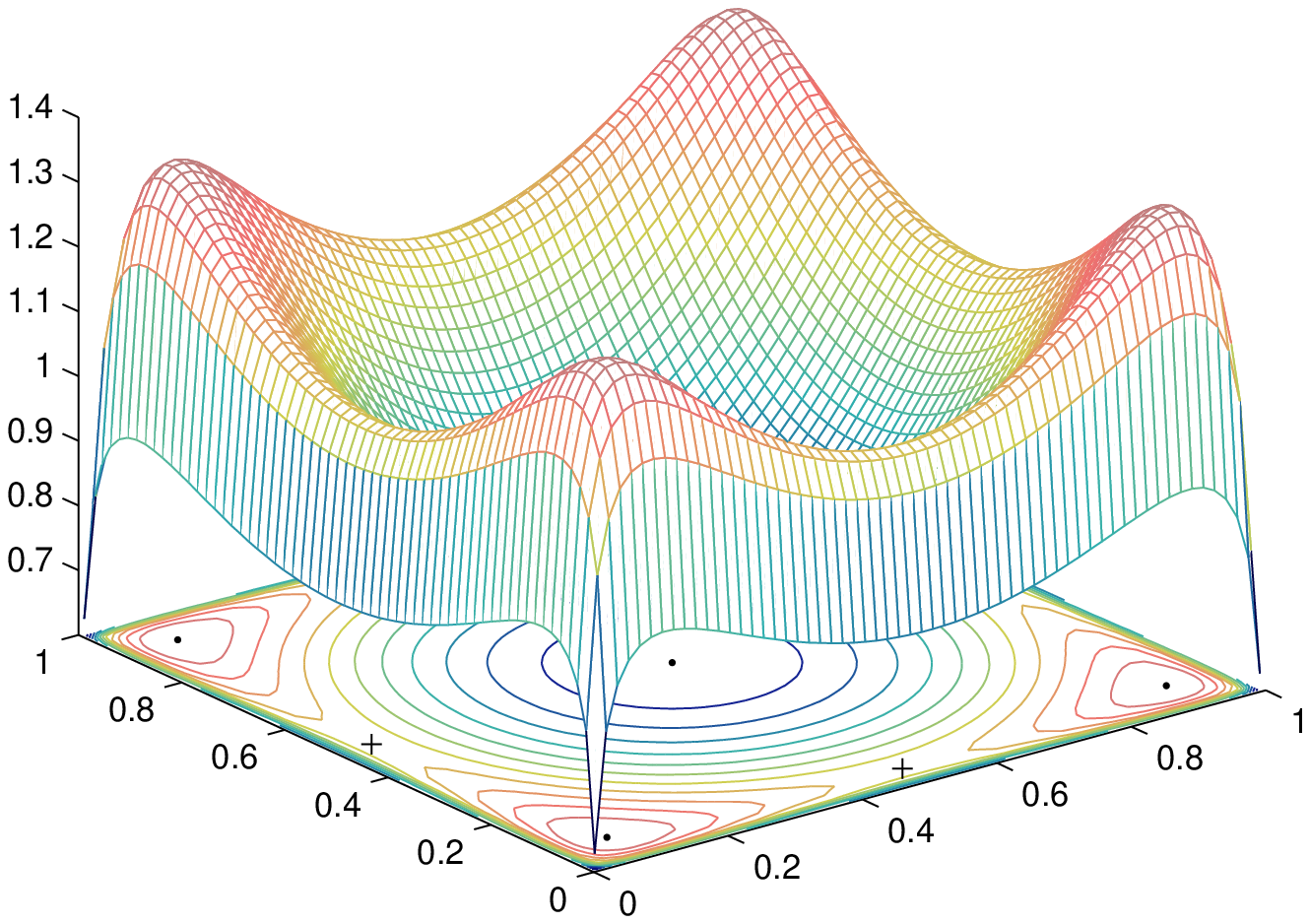}
\caption{Graph of $g_{2,\mu,s}$ when $\mu=0.6$, $s_1-s_0=-2$ and $s_2-s_1=-6$ so that $\lambda_0\simeq 0.0766$.  A black dot marks the position of each extremum and  a cross is plotted at each saddle point.\label{figdensity0226}}
\end{minipage}
\end{figure}
\begin{figure}[htb]
\begin{minipage}[t]{.46\linewidth}
\centering
\includegraphics[scale=0.4]{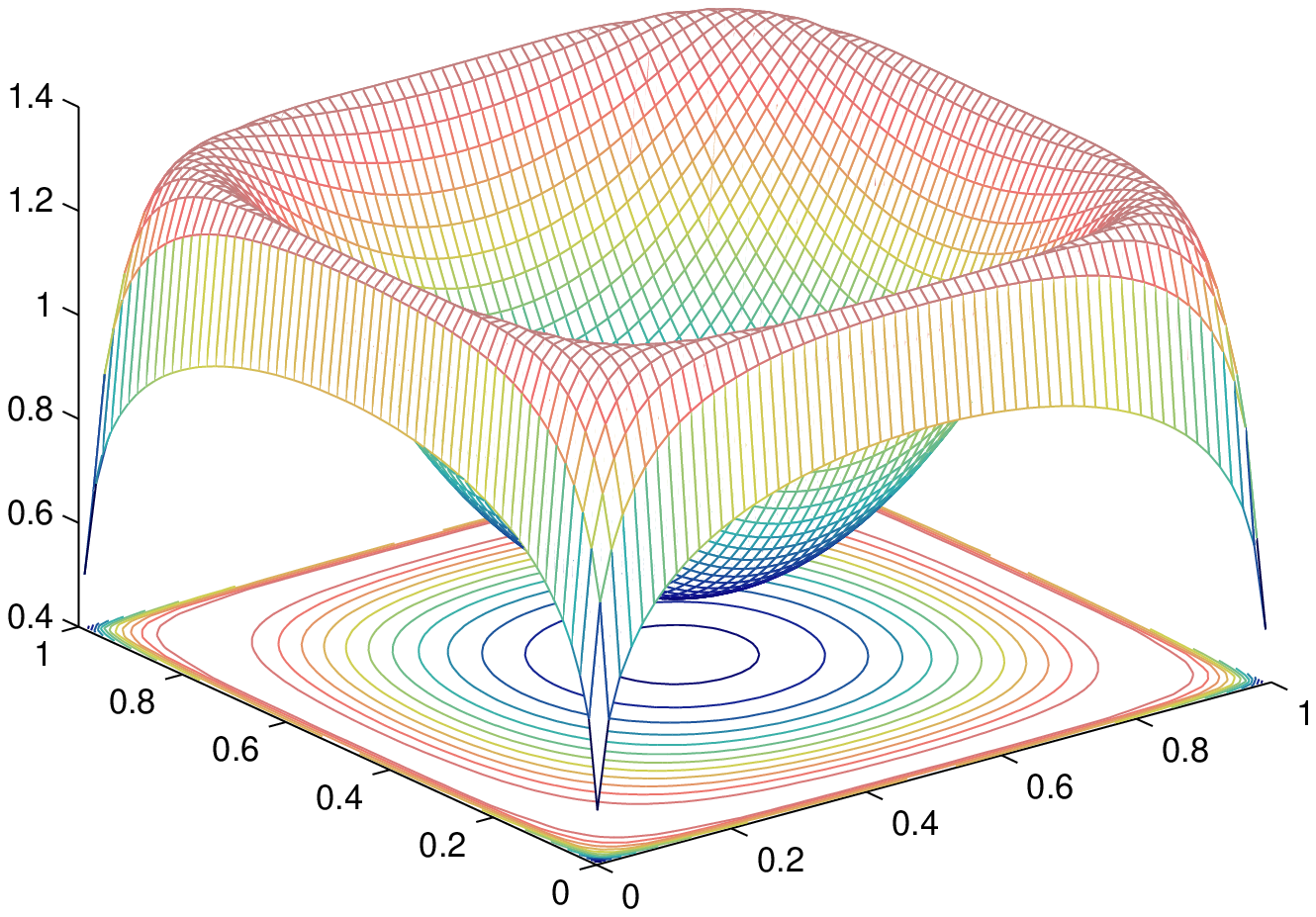}
\caption{Graph of $g_{2,\mu,s}$ when $\mu=0.6$, $s_1-s_0=0$ and $s_2-s_1=-12$; there is a continuum of critical points.\label{figdens2continuum} }
\end{minipage}\hfill
\begin{minipage}[t]{.46\linewidth}
\centering \includegraphics[scale=0.2]{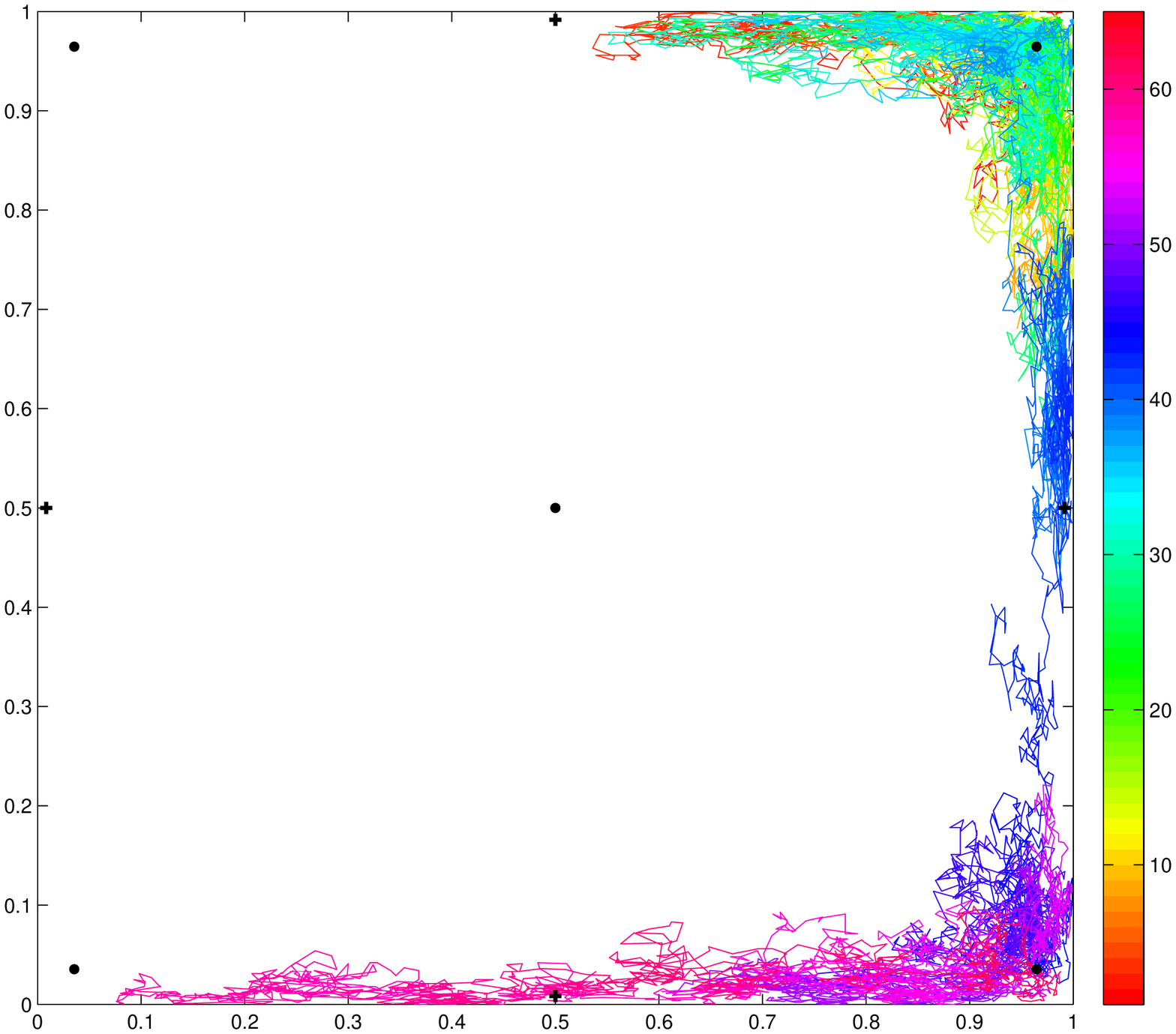}
\caption{Simulation of the evolution of the 0-allelic frequency in the two-locus model.
The population size is $N=10^3$, $\mu=1$, $s_1-s_0=-15$, $s_2-s_1=-210$.  A black dot marks the position of each extremum and  a cross is plotted at each saddle point. In this example,  $\lambda_0\simeq 0.034$  and $\lambda_1\simeq 0.008$. \label{figsimul2} }
\end{minipage}\hfill
\end{figure}
\begin{figure}[h]
\begin{minipage}[h]{0.55\linewidth}
\includegraphics[scale=0.55]{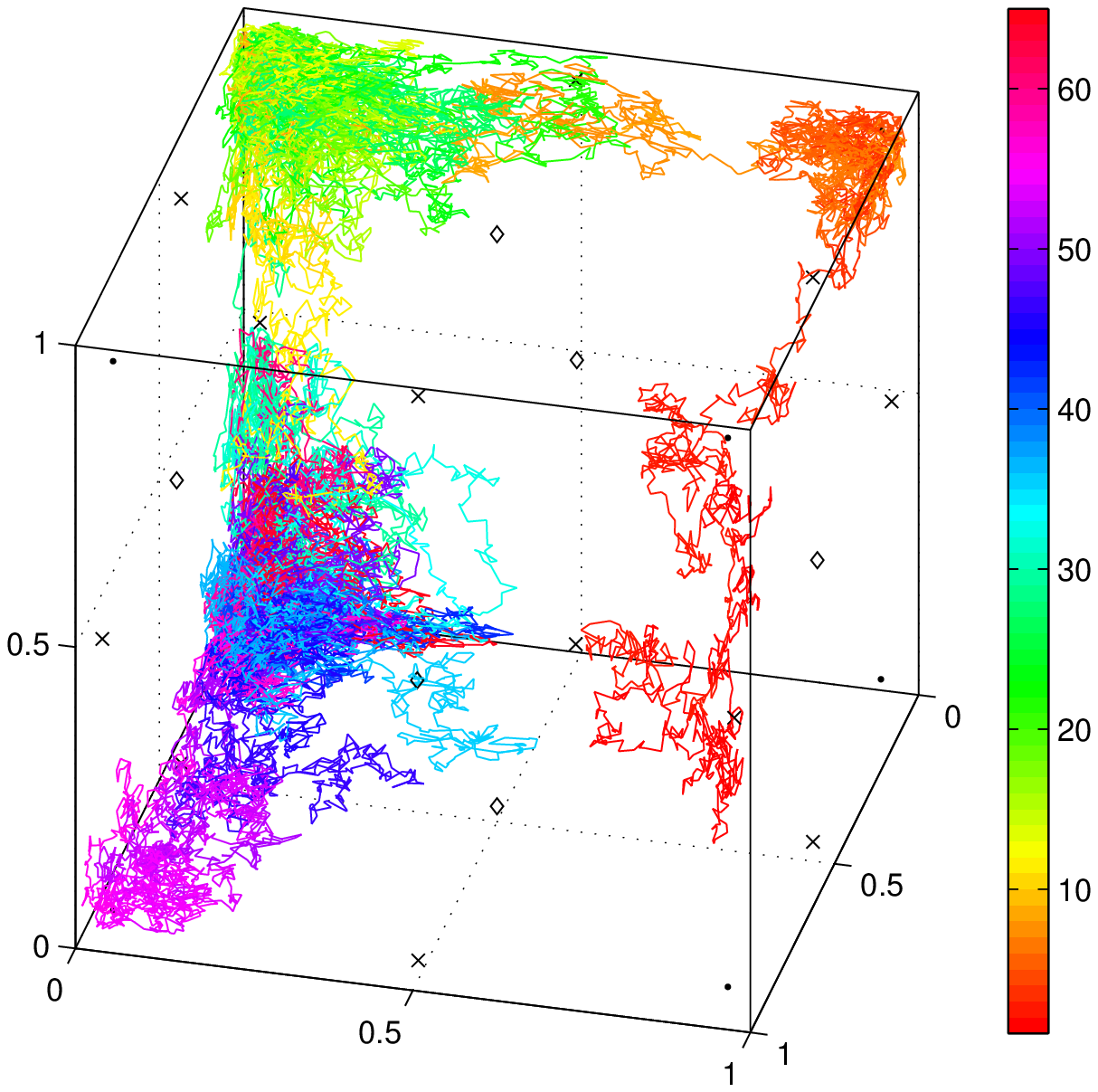}
\end{minipage}
\begin{minipage}[h]{0.45\linewidth}
 (a) \\
Assortment parameters: $s_1-s_0=-20$, $s_2-s_1=-40$ and $s_3-s_2=-60$.\medskip\\
Characteristics of the stationary density: 
$\lambda_0\simeq 0.043$, $\lambda_1\simeq 0.031$  and $\lambda_2\simeq 0.025$.\\
$h_{0}-h_{1}=7.9$, $h_{0}-h_{2}\simeq 24.3$ and ${h_{0}-h_{3}\simeq 49.8}$. 
\end{minipage}
\begin{minipage}{0.55\linewidth}
\includegraphics[scale=0.55]{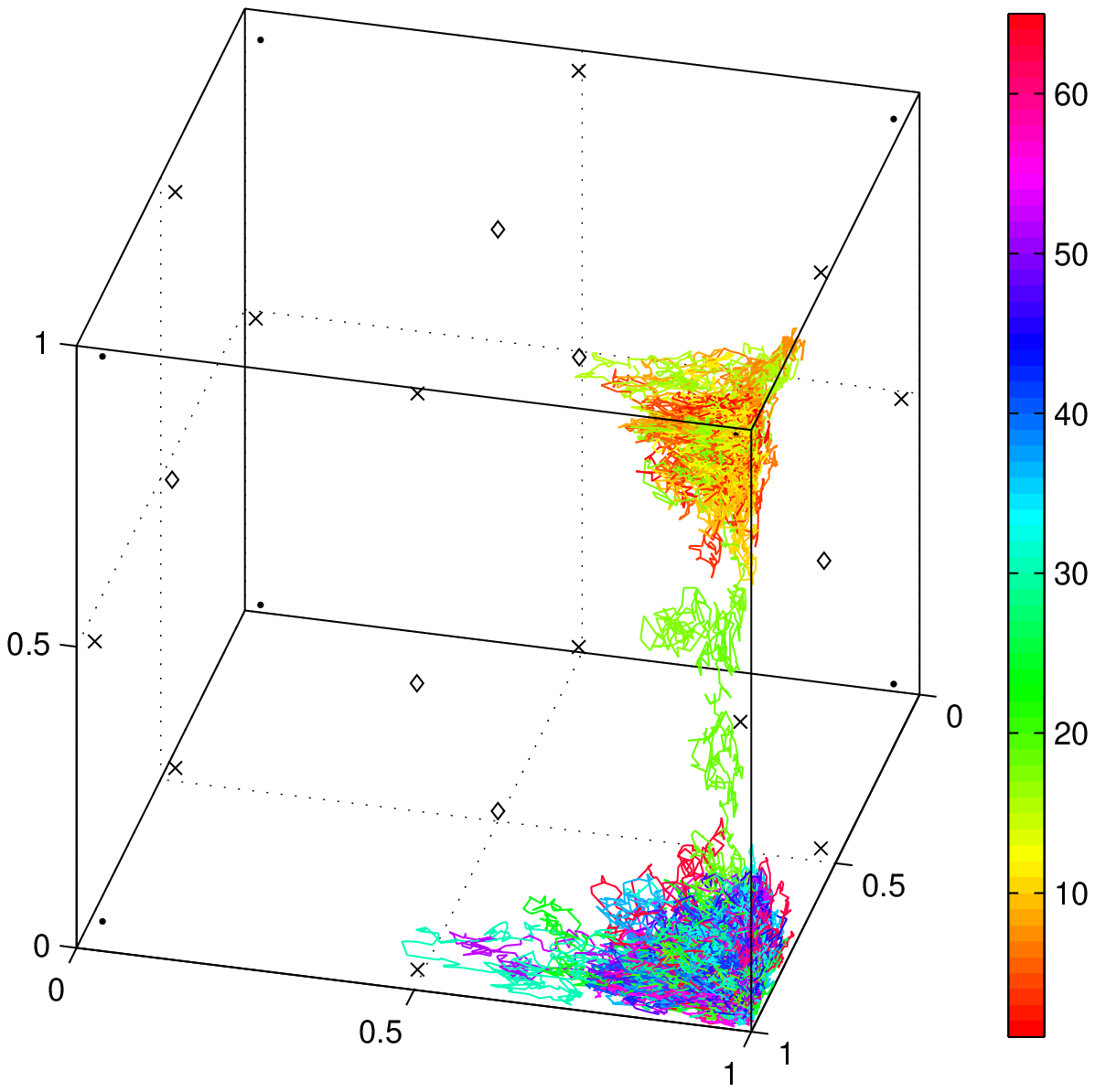}
\end{minipage}
\begin{minipage}{0.45\linewidth}
(b) \\
Assortment parameters:  $s_1-s_0=-30$, $s_2-s_1=-60$ and $s_3-s_2=-90$.
\medskip\\
Characteristics of the stationary density: 
$\lambda_0\simeq 0.030$, $\lambda_1\simeq 0.021$  and $\lambda_2\simeq 0.017$.\\
$h_{0}-h_{1}=12.6$, $h_{0}-h_{2}\simeq 38.6$ and ${h_{0}-h_{3}\simeq 78.7}$.
\end{minipage}

\caption{
Simulations of the evolution of the 0-allelic frequency  with the three-locus model  for two different sets of assortment parameters. The assortative mating favours more strongly pairing between similar types in Fig. \ref{figsimul3}b. The size of the population is  $N=10^3$ and the mutation rate is $\mu=1$.  A black dot marks the position of each global maximum of the stationary density,  a cross the position of each saddle point of index 2 and  a diamond the position of each saddle point of  index 1. Some numerical characteristics of the stationary density are presented to the right of each figure: for $i\in\{1,2,3\}$, the value of $\lambda_i=\xi_i(1-\xi_i)$ provides the position of the critical points of index $3-i$ (see Proposition \ref{descloirev}) and $h_{i}$ is the value of the log-density $h_{n,\mu,s}$ at a critical point of index $3-i$.   \label{figsimul3}
}
\end{figure}
\clearpage
\subsection{Proofs of Propositions~\ref{descriunptcrit} and~\ref{descloirev}\label{subsectdproofstation}}
\paragraph{Proof of Proposition~\ref{descriunptcrit}}
Let us introduce some notation in order to shorten the expressions. 
We set $\nu=2\mu-1$, $\rho(u)=u(1-u)$ for $u\in[0,1]$,  $\bb{\rho}(x)=(\rho(x_1),\ldots,\rho(x_n))$,
$$\bar{h}(x)=\nu\sum_{i=1}^{n}\log(x_i)+\frac{1}{2}\sum_{\ell=1}^{n}(m(\ell)-m(0))\sum_{L\subset\enu{n},\ |L|=\ell\ }\prod_{j\in L}(2x_{j})\prod_{k\in\enu{n}\setminus L}(1-2x_{k})$$
and $h(x)=\bar{h}(\bb{\rho}(x))$ for $x=(x_1,\ldots,x_n)\in]0,1[^n$. 
With this notation, $g_{n,\mu,s}(x)=C_{n,\mu,s}\exp(h(x))$. 
\begin{enumerate}
 \item For every $x\in]0,1[^n$ and $i\in\enu{n}$,  $\partial_{i}h(x)=(1-2x_i)\partial_{i}\bar{h}(\bb{\rho}(x))$
where $$\partial_{i}\bar{h}(x)=\frac{\nu}{x_i}+\sum_{\ell=0}^{n-1}\delta^{(1)}[m](\ell)\sum_{L\subset\enu{n}\setminus\{i\},|L|=\ell\ }\prod_{j\in L}(2x_j)\prod_{k\in\enu{n}\setminus (L\cup\{i\})}(1-2x_{k}). $$
First, the point $u_n=(1/2,\ldots,1/2)$ is a critical point of  $h_{n,\mu,s}$ and the Hessian matrix at this point is the diagonal matrix 
$-2\Delta I_n$ where  $$\Delta=4\nu+2^{-(n-1)}\sum_{i=0}^{n-1}\binom{n-1}{i}\delta^{(1)}[m](i)=4V_n.$$
This proves the first two assertions of the proposition. 
\item The last two assertions follow from the fact that $\partial_{i}\bar{h}(x)$ and $\Delta$ are increasing  
functions of $\delta^{(1)}[m](\ell)$ for every $\ell$.  Let us prove assertion 3 to illustrate the method. First, if $\delta^{(1)}[m(s)](\ell)=-(8\mu-4)$ for every $\ell\in \enum{0}{n-1}$ then the $n$ coordinates of the diffusion are independent. In this case,   $\Delta=0$ and the stationary density has only one critical point at $(1/2,\ldots,1/2)$ which is  a maximum.   If $\{s_{\bb{i},\bb{j}}, (\bb{i},\bb{j})\in\mc{A}^2\}$ is a family of 
assortment parameters such that $\partial_{i}\bar{h}(x)$ is nonnegative for every $x\in]0,1/4]^n$ and the density $g_{n,s,\mu}$ has a unique critical point at $(1/2,\ldots,1/2)$ which is a maximum, then  the same is true for any  family of assortment parameters 
$\{\hat{s}_{\bb{i},\bb{j}}, (\bb{i},\bb{j})\in\mc{A}^2\}$ such that 
$\delta^{(1)}[m(\hat{s})](\ell)\geq \delta^{(1)}[m(s)](\ell)$ for every $\ell\in\enum{0}{n-1}$. 
\end{enumerate}
\paragraph{Proof of Proposition~\ref{descloirev}}
We retain the notation introduced in the proof of Proposition~\ref{descriunptcrit}. 
For $k\in\enu{n}$, we set $\alpha_{k}=2^{k}\delta^{(k+1)}[m](0)$ and denote by $e_{n,k}$ the 
elementary symmetric polynomial function in $n$ variables of degree $k$:   
$$e_{n,0}(x)=1 \text{ and }e_{n,k}(x)=\sum_{{L\subset\enu{n},}\atop {|L|=k}}\prod_{\ell\in L}x_{\ell}\quad \text{ for } k\in\enu{n}.$$
For instance, $e_{n,1}(x)=x_1+\ldots+x_n$, $e_{n,2}(x)=\sum_{1\leq i<j\leq n}x_ix_j$.\\
With this notation 
$$\bar{h}(x)=\nu\sum_{i=1}^{n}\ln(x_i)+\sum_{\ell=0}^{n-1}\alpha_\ell e_{n,\ell+1}(x). $$ 
In the proof we shall use (several times) the following identity for elementary symmetric 
polynomial functions:
\begin{lem}\label{ESP}
Let  $n$ be an integer greater than $1$ and let $k\in\enum{0}{n-2}$. For every $x\in\R^n$, set  $\hat{x}^{(i)}=(x_1,\ldots,x_{i-1},x_{i+1},\ldots,x_n)$ for  $i\in\enu{n}$ and $$\hat{x}^{(i,j)}=\hat{x}^{(j,i)}=(x_1,\ldots,x_{i-1},x_{i+1},\ldots,x_{j-1},x_{j+1},\ldots,x_n) \text{ for  }i,j\in\enu{n} \text{ such that }i<j.$$ Then, 
\begin{equation}
 x_{i}e_{n-1,k}(\hat{x}^{(i)})-x_{j}e_{n-1,k}(\hat{x}^{(j)})=(x_i-x_j)e_{n-2,k}(\hat{x}^{(i,j)}). 
\end{equation}
\end{lem}
We shall also use the following alternative expression for symmetric polynomial functions that 
are similar to the polynomial term in $h$:
\begin{lem}
\label{polydelta}
Let $n\in\NN^*$ and let $a_0,\ldots,a_n$ be real numbers. Then for every $x\in\R^n$,  
$$
\sum_{k=0}^{n}2^{k}\delta^{(k)}[a](0)e_{n,k}(x)=\sum_{i=0}^{n}a_i\sum_{I\subset \enu{n},\ |I|=i\ }\prod_{i\in I}2x_i\prod_{j\not\in I}(1-2x_j).
$$
In particular, for every $y\in \R$ and $\ell\in\enum{0}{n}$, $$\sum_{k=0}^{n}2^k\delta^{(k)}[a](0)e_{n,k}((1/4)^{\otimes \ell},y^{\otimes(n-\ell)})=\sum_{i=0}^{n}a_iB_{n,\ell,i}(2y)$$
where $\displaystyle{B_{n,\ell,i}(y)=2^{-\ell}\sum_{j=\max(0,i-n+\ell)}^{\min(i,\ell)}\binom{\ell}{j}\binom{n-\ell}{i-j}y^{i-j}(1-y)^{n-\ell-(i-j)}}$.
\end{lem}
\begin{proof} See Corollary \ref{invformulacst}.
\end{proof}
\begin{enumerate}
\item
Let us assume that $x=(x_1,\ldots,x_n)$ is a critical point of  $g_{n,\mu,s}$ different from $u_n$. Let $\ell$ denote the number of coordinates equal to $1/2$ ($\ell\in\enum{0}{n-1}$).  Every coordinate $x_i$ different from $1/2$ has to satisfy: $\partial_i h_{n,\mu,s}(\bb{\rho}(x))=0$, that is
$$\nu+\rho(x_i)\sum_{k=0}^{n-1}\alpha_k e_{n-1,k}(\widehat{\bb{\rho}(x)}^{(i)})=0. $$
In particular, it follows from Lemma~\ref{ESP} that if $x_{i}$ and $x_j$ are two coordinates 
of the critical point $x$ not equal to 1/2 then $$\rho(x_i)=\rho(x_j)\text{ or }\sum_{k=0}^{n-2}\alpha_ke_{n-2,k}(\widehat{\bb{\rho}(x)}^{(i,j)})=0. $$
By Lemma~\ref{polydelta}, 
$$\sum_{k=0}^{n-2}\alpha_ke_{n-2,k}(x)=\sum_{\ell=0}^{n-2}\delta^{(1)}[a](\ell)Q_{\ell}(x),$$ where $Q_{\ell}$ denotes a polynomial function which is positive on $x\in]0,1/4[^{n-2}$ for every $\ell\in\enum{0}{n-2}$. Thus this sum  cannot vanish in $]0,1/4[^{n-2}$ under the assumption that all coefficients $\delta^{(1)}[m](i)$ have the same sign and that for at least one $i\leq n-2$, $\delta^{(1)}[m](i)$ is non-zero. 
Therefore, such a critical point exists only if there exists a solution in the interval $]0, 1/4[$ of 
\begin{equation}
\nu+y\sum_{k=0}^{n-1}\alpha_k e_{n-1,k}\left((\frac{1}{4})^{\otimes \ell},y^{\otimes (n-\ell-1)}\right)=0. \tag{$\mathcal{E}^{'}_{\ell}$}\label{eqcritnalpha}
\end{equation}
In order to study the solutions of  \eqref{eqcritnalpha}, let $\phi_{\ell}(y)$ denote  the left-hand side of \eqref{eqcritnalpha}:
\begin{equation}\label{dfnPhi}
 \phi_{\ell}(y)=\nu+y\sum_{k=0}^{n-1}\alpha_k e_{n-1,k}((\frac{1}{4})^{\otimes \ell},y^{\otimes (n-\ell-1)})
\end{equation}
By Lemma~\ref{polydelta}, \begin{equation}\phi_{\ell}(y)=\nu+y\sum_{i=0}^{n-1}B_{n-1,\ell,i}(2y)\delta^{(1)}[m](i). 
 \end{equation}
Therefore, \eqref{eqcritnalpha} coincides with \eqref{eqcritn} of Remark~\ref{remptcrit}. The derivative of $\phi_{\ell}$ is equal to:
\begin{multline*}
\phi^{'}_{\ell}(y)=\sum_{i=0}^{n-1}B_{n-1,\ell,i}(2y)\delta^{(1)}[m](i)\\+2y(n-1-\ell)\sum_{i=0}^{n-2}B_{n-2,\ell,i}(2y)(\delta^{(1)}[m](i+1)-\delta^{(1)}[m](i)).
\end{multline*}
If $\delta^{(1)}[m](n-1)\leq\cdots \leq \delta^{(1)}[m](0)\leq 0$ 
 (respectively $\delta^{(1)}[m](n-1)\geq\cdots \geq \delta^{(1)}[m](0)\geq 0$),   
$\phi_{\ell}$ is a decreasing function on the interval $[0,1/2]$
(resp. an increasing function on the interval $[0,1/2]$). The value of $\phi_{\ell}$ at $0$ is 
$\nu$  and the value at $1/4$ is $V_n$. Therefore, under the assumptions of 1 or 2 of the proposition,
for every $\ell\in\{0,\ldots,n-1\}$ \eqref{eqcritnalpha} has no solution in $]0,1/4[$ if $V_n$ 
and $\nu$ have the same sign and  has exactly one solution in $]0,1/4[$ denoted by $\lambda_{\ell}$ 
if $V_n$ and $\nu$ have opposite signs. This proves assertions 1.(a) and 2.(a).  \\
For every pair of disjoint subsets $I$ and $J$  of $\enu{n}$, let us introduce the following 
point:  $u_{I,J}=(x_1,\ldots,x_n)$ with $$x_i=\begin{cases} 1/2 \text{ if } i\in I,\\
1/2 +1/2\sqrt{1-4\lambda_{|I|}} \text{ if } i\in J,\\
1/2 -1/2\sqrt{1-4\lambda_{|I|}} \text{ if } i\in \enu{n}\setminus( I\cup J).
\end{cases}$$ 
We have shown that if $V_n$ and $\nu$ have opposite signs, then every point $u_{I,J}$ is a critical point and any critical point is one of these points $u_{I,J}$. 
\newcounter{enumeration}
\setcounter{enumeration}{\theenumi}
\end{enumerate}
So that we may use our conclusions above, 
from now on, we assume that the hypotheses stated in point 1 of the proposition are satisfied. 
However, the computations that follow do not depend on these hypotheses, and so our proof is 
easily modified to the setting of point 2.
\begin{enumerate}
\setcounter{enumi}{\theenumeration}
\item Let us study the Hessian matrix of $h_{n,\mu,s}$ at a critical point $u_{I,J}$  such that $|I|\leq n-1$. For that,  set $\ell=|I|$, $\ell^{+}=|J|$ and $\ell^{-}=n-\ell-\ell^{+}$ and let us introduce the following notations: 
\begin{alignat*}{2}
a_{\ell}&=\partial_{1}\bar{h}((\frac{1}{4})^{\otimes \ell},(\lambda_\ell)^{\otimes (n-\ell)}),\ &b_{\ell}=-(1-4\lambda_{\ell})\frac{\nu}{\lambda_{\ell}^{2}},\\ c_\ell&=(1-4\lambda_{\ell})\partial^{2}_{n,n-1}\bar{h}((\frac{1}{4})^{\otimes \ell},(\lambda_\ell)^{\otimes (n-\ell)}).
   \end{alignat*}
 The Hessian matrix of $h_{n,\mu,s}$ at $u_{I,J}$ is permutation-similar to the following block matrix: 
$$\mc{H}_{I,J}=\begin{pmatrix}
\bb{A}_{\ell} &\bb{0} & \bb{0}\\
\bb{0} & \bb{B}_{\ell,\ell^{+}} & \bb{C}_{\ell}\\
\bb{0}& \bb{C}_{\ell} & \bb{B}_{\ell,\ell^{-}}\\ 
              \end{pmatrix}
$$
where 
\begin{itemize}
 \item $\bb{A}_{\ell}$  denotes the scalar matrix $-2a_{\ell}I_{\ell}$ with $a_{\ell}=\partial_{1}\bar{h}((\frac{1}{4})^{\otimes \ell},(\lambda_\ell)^{\otimes (n-\ell)})$, 
\item $\bb{B}_{\ell,k}$ denotes  the following $k$-by-$k$ matrix  : $\bb{B}_{\ell,k}=\begin{pmatrix} b_{\ell}&c_{\ell}&\cdots&c_{\ell}\\
c_{\ell}&\ddots&\ddots &\vdots\\
\vdots &\ddots& \ddots & c_{\ell}\\
c_{\ell}&\cdots&c_{\ell}&b_{\ell}\\
\end{pmatrix}$,
\item $\bb{C}_{\ell}$ denotes the $\ell^{+}$-by-$\ell^{-}$ matrix all the elements of which are equal to~$-c_{\ell}$. 
\end{itemize}
By assumption on $\mu$, $b_{\ell}<0$. To complete the proof of assertions (i) and (ii) of 1-(b), we shall prove that $a_{\ell}<0$ and that ${b_{\ell} <c_{\ell}<0}$. 
That  will imply that the submatrix 
$\begin{pmatrix}\bb{B}_{\ell,\ell^{+}} & \bb{C}_{\ell}\\
\bb{C}_{\ell} & \bb{B}_{\ell,\ell^{-}}
\end{pmatrix}$  is negative definite (for more details, see Lemma~\ref{symmatrix}) hence that the Hessian matrix of $h_{n,\mu,s}$ at a point $u_{I,J}$  has $|I|$ positive eigenvalues and $n-|I|$ negative eigenvalues. \\
First, let us study the sign of $a_{\ell}=4\nu+\sum_{i=0}^{n-1}\alpha_{i}e_{n-1,i}((\frac{1}{4})^{\otimes (\ell-1)},\lambda_{\ell}^{\otimes (n-\ell)})$. As  $\phi_{\ell}(\lambda_{\ell})=0$, an application of Lemma~\ref{ESP} yields:
$$a_{\ell}=(1-4\lambda_{\ell})\sum_{i=0}^{n-2}\alpha_{i}e_{n-2,i}((\frac{1}{4})^{\otimes (\ell-1)},\lambda_{\ell}^{\otimes (n-1-\ell)}).$$
The right-hand side can be rewritten using Lemma~\ref{polydelta}:
$$
a_{\ell}=(1-4\lambda_{\ell})\sum_{i=0}^{n-2}\delta^{(1)}[m](i)B_{n-2,\ell-1,i}(2\lambda_{\ell}). 
$$
The conditions  on $\delta^{(1)}[m](i)$ imply that $a_{\ell}$ is negative. \\
Let us now study the coefficients $\tilde{b}_{\ell}=(1-4\lambda_{\ell})^{-1}b_{\ell}$ and $\tilde{c}_{\ell}=(1-4\lambda_{\ell})^{-1}c_{\ell}$. As in the study of $a_{\ell}$ we use that $\phi_{\ell}(\lambda_{\ell})=0$ and Lemma~\ref{polydelta} to write $\tilde{b}_{\ell}$ and $\tilde{c}_{\ell}$ in terms of the coefficients $\delta^{(1)}[m](s)(i)$:
\begin{align*}
\tilde{b}_{\ell}&=\frac{1}{\lambda_{\ell}}\sum_{i=0}^{n-1}\delta^{(1)}[m](i)B_{n-1,\ell,i}(2\lambda_{\ell}),\\
 \tilde{c}_{\ell}&=2\sum_{i=0}^{n-2}(\delta^{(1)}[m](i+1)-\delta^{(1)}[m](i))B_{n-2,\ell,i}(2\lambda_{\ell}).
\end{align*}
As $\delta^{(1)}[m(s)](i)$ is assumed to be a decreasing sequence, $\tilde{c}_{\ell}<0$. 
After some computations, we obtain:
$$
\lambda_{\ell}(\tilde{c}_{\ell}-\tilde{b}_{\ell})=-\sum_{i=0}^{n-2}\delta^{(1)}[m](i)B_{n-2,\ell,i}(2\lambda_{\ell}).
$$
The conditions  on $\delta^{(1)}[m](i)$ imply that $\tilde{c}_{\ell}>\tilde{b}_{\ell}$.
\item Let us prove that $0<\lambda_{n-1}<\cdots < \lambda_{0}<1/4$ , which gives the relative positions of the coordinates of the critical points. \\ 
Let $\ell\in\enum{0}{n-2}$. If we return to the expression \eqref{dfnPhi} of $\phi_{\ell}$, use 
Lemma~\ref{ESP} and then Lemma~\ref{polydelta},  we obtain:
\begin{align*}
\phi_{\ell+1}(y)-\phi_{\ell}(y)&=y(1/4-y)\sum_{i=0}^{n-2}\alpha_{i+1}e_{n-2,i}((1/4)^{\otimes \ell},y^{\otimes (n-2-\ell)})\\
&=2y(1/4-y)\sum_{i=0}^{n-2}\delta^{(2)}[m](i)B_{n-2,\ell,i}(2y).
\end{align*}
By assumption, $\delta^{(2)}[m](i)\leq 0$ for every $i\in\enum{0}{n-2}$ hence $\phi_{\ell+1}(y)\leq\phi_{\ell}(y)$  for every $y\in[0,1/4]$. As the functions $\phi_{\ell}$ are decreasing  on $[0,1/4]$, we deduce that $\lambda_{\ell+1} \leq  \lambda_{\ell}$ for every $\ell\in\enum{0}{n-2}$. As the two critical points $u_{\enu{\ell},\emptyset}$ and $u_{\enu{\ell+1},\emptyset}$ have not the same properties, they cannot coincide and thus $\lambda_{\ell+1}<\lambda_{\ell}$ for every  $\ell\in\enum{0}{n-2}$.

\item \emph{Proof of assertion $1.(b).v$:} let $h_{\ell}$ denote the value of $h_{n,\mu,s}$ at a saddle point of index $n-\ell$: $p_\ell=((1/2)^{\otimes \ell},(\xi_{\ell})^{\otimes (n-\ell)})$. To prove that $h_{\ell}> h_{\ell+1}$ for every $\ell\in\enum{0}{n-2}$, we shall use the properties of the gradient dynamical system $\frac{dx(t)}{dt}=-\nabla \tilde{h}(x)$ with  $\tilde{h}=-h_{n,\mu,s}$. Fix a positive value $M$ large enough so that  $U_M=\tilde{h}^{-1}([-M,M])$ contains all critical points of $h$ (such an $M$ exists since $\tilde{h}(x)$ tends to infinity as $x$ tends to the boundary of $[0,1]^n$). The function $\tilde{h}$ decreases along  trajectories and a trajectory of a point $x\in M$ converges to a critical point of $\tilde{h}$ as $t$ tends towards $+\infty$, since $\tilde{h}$ has only isolated critical points. 
For $k\in \{0,\ldots,n-1\}$, let $U^{(k)}_M$ denote  the subset: $$ U^{(k)}_M=\{x\in U_M,\ x_1=\cdots =x_k=1/2 \text{ and } x_i<1/2\ \forall i>k\}.$$ 
Every subset $U^{(k)}_M$ contains exactly one critical point, the saddle point $p_k$. As $\partial_i \tilde{h}(x)=0$ at points $x$ such that $x_i=1/2$, the subset $U^{(k)}_M$ is positively invariant by the gradient flow. %
Therefore, to prove that $h_{k}> h_{k+1}$, it is enough to show that there exists  $0<y_0<1/2$ such that for $y\in] y_0, 1/2 [$,\  ${\tilde{h}((1/2)^{\otimes k},y,\xi_{k+1}^{\otimes{n-k-1}})<\tilde{h}(p_{k+1})}$. \\
As $\tilde{h}((1/2)^{\otimes k},y,\xi_{k+1}^{\otimes{n-k-1}})=-\bar{h}_{n,\mu,s}((1/4)^{\otimes k},y(1-y),\lambda_{k+1}^{\otimes n-k-1})$, it is enough to show that $\partial_{k+1}\bar{h}_{n,\mu,s}((1/4)^{\otimes (k+1)},\lambda_{k+1}^{\otimes (n-k-1)})<0$. Using that $\lambda_{k+1}$ is solution of 
the equation ($\mc{E}_{k+1}$), we obtain $$\partial_{k+1}\bar{h}_{n,\mu,s}((1/4)^{\otimes (k+1)},\lambda_{k+1}^{\otimes (n-k-1)})=(1-4\lambda_{k+1})\sum_{i=0}^{n-2}\delta^{(1)}[m](i)B_{n-2,k,i}(2y)<0.$$
\end{enumerate}
%
\section{Proof of convergence to the diffusion \label{secconvproof}}
In this section, we prove convergence to the diffusion approximation in the $n$-locus case 
(Theorem~\ref{thgen}).  We also establish the two simple expressions for the drift presented
in \Sref{secconv}.\par

First, the properties of the generator $\mc{G}_{n,s}$ stated in assertion (a) of Theorem~\ref{thgen} 
can be obtained by  applying the following theorem established by Cerrai and Cl\'ement: 
\begin{thm}[\citealt{Cerrai}] 
Let $\mathcal{S}^{+}(\R^n)$ be the space of symmetric, non-negative definite, $n\times n$ matrices. 
Let $A:[0,1]^{n}\rightarrow \mathcal{S}^{+}(\R^n)$ and $b:[0,1]^n\rightarrow \R^n$ be mappings of 
class $C^2$.  For $i\in\{1,\ldots,n\}$ and $\epsilon\in\{0,1\}$, let $\nu^{i}_{\epsilon}$ denote 
the unit inward normal vector of the hypercube 
$C^{i}_{\epsilon}=\{x\in [0,1]^n,\ x_i=\epsilon\}$. Let us assume the following two conditions:
\begin{itemize}
 \item for every $i\in\{1,\ldots,n\}$, $\epsilon\in\{0,1\}$ and $x\in C^{i}_{\epsilon}$,
$A(x)\nu^{i}_{\epsilon}(x)=\bb{0}$ and ${\langle b(x),\nu^{i}_{\epsilon}(x)\rangle\geq 0}$;
\item for every $i,j\in\{1,\ldots,n\}$, $A_{i,j}(x)$ depends only on $x_i$ and $x_j$. 
\end{itemize}
Then the operator $$L=\frac{1}{2}\sum_{i=1}^{n}\sum_{j=1}^{n}A_{i,j}(x)\frac{\partial^{2}}{\partial_{x_i}\partial_{x_j}}+\sum_{i=1}^{n}b_i(x)\frac{\partial^{2}}{\partial_{x_i}}$$
is closable in $C([0,1]^n)$ and its closure is the generator of a strongly 
continuous semigroup of contractions. 
\end{thm}

To prove the convergence result, we use the following theorem, due to Ethier and Nagylaki,
on diffusion approximations for Markov chains with two time scales.
\begin{thm}[\citealt{EthierNagylaki80}, Theorem~3.3] 
\label{ethier/nagylaki thm}
For $N\in\NN^*$, let $\{Z^{N}_k,\ k\in \NN\}$  be a homogeneous Markov chain in a metric space $E_N$ with Feller transition function. Let $F_1$ and $F_2$ be compact convex subsets of $\R^n$ and $\R^m$ respectively, having non-empty interiors. Assume further that $0\in\overset{\circ}{F}_2$. Let $\Phi_N:E_N\rightarrow F_1$ and $\Psi_N:E_N\rightarrow F_2$ be continuous functions. Define $X^{N}_k=\Phi_N(Z^N_{k})$ and $Y^{N}_k=\Psi_N(Z^{N}_k)$ for each $k\in\NN$. Let $(\epsilon_N)_N$ and $(\delta_N)_N$ be two positive sequences such that $\delta_N\rightarrow 0$ and $\epsilon_N/\delta_N\rightarrow 0$. \\
 Assume that there exist continuous functions $a:F_1\times \R^m\rightarrow \R^n\otimes\R^n$, $b:F_1\times \R^m\rightarrow \R^n$ and $c:F_1\times \R^m\rightarrow \R^m$ such that for $i,j\in\enu{n}$ and $\ell\in\enu{m}$ the following properties (a)-(e) hold as $N\rightarrow +\infty$ uniformly in $z\in E_N$ where $x=\Phi_N(z)$ and $y=\Psi_N(z)$:
\begin{itemize}
 \item[(a)] $\epsilon_{N}^{-1}\EE_z[X^{N}_{1}(i)-x(i)]=b_i(x,y)+o(1)$,
\item[(b)] $\epsilon_{N}^{-1}\EE_z\big[(X^{N}_{1}(i)-x(i))(X^N_{1}(j)-x(j))\big]=a_{i,j}(x,y)+o(1)$,
\item[(c)] $\epsilon_{N}^{-1}\EE_z[(X^N_{1}(i)-x(i))^4]=o(1)$,
\item[(d)] $\delta_{N}^{-1}\EE_z[Y^N_{1}(\ell)-y(\ell)]=c_{\ell}(x,y)+o(1)$,
\item[(e)] $\delta_{N}^{-1}\EE_z[(Y^N_{1}(\ell)-y(\ell))^2]=o(1)$.
\end{itemize}
Assume further that
\begin{itemize}
 \item[(f)] $c$ is  of class $C^2$,  $c(x,0)=0$ for all $x\in\R^m$ and the solution of the differential equation 
$$ \frac{d}{dt}u(t,x,y)=c(x,u(t,x,y)), \quad u(0,x,y)=y.$$
exists for all $(t,x,y)\in [0,+\infty[\times F_1\times F_2$ and satisfies 
$$\lim_{t\rightarrow +\infty}\sup_{(x,y)\in F_1\times F_2}|u(t,x,y)|=0.$$
\item[(g)] The closure of the following operator 
$$\mathcal{L}=\frac{1}{2}\sum_{i,j=1}^{n}a_{i,j}(x,0)\frac{\partial^{2}}{\partial_{x_i}\partial_{x_j}}+\sum_{i=1}^{n}b_{i}(x,0)\frac{\partial}{\partial_{x_i}},\quad \mathcal{D}(\mathcal{L})=C^2(F_1),$$
generates a strongly continuous semigroup on $C(F_1)$ corresponding to a diffusion process $X$ in $F_1$. 
\end{itemize}
Then the following conclusions in which the symbol $\Rightarrow$ denotes convergence in distribution, hold:
\begin{itemize}
 \item[(i)] If $X^{N}_{0} \Rightarrow X(0)$ then $\{X^{N}_{[t/\epsilon_N]},t\geq 0\} \Rightarrow X(\cdot)$ in $D_{F_1}([0,+\infty[)$ (where $D_{F_1}([0,+\infty[)$ is the space of c\`adl\`ag paths $\omega: [0,\infty)\rightarrow F_1$ with the Skorohod topology),
\item[(ii)] For every positive sequence $(t_N)_N$ that converges to $+\infty$, $Y^{N}_{[t_N/\delta_N]}\Rightarrow 0$.
\end{itemize}
\end{thm}
\begin{rem}We have only stated the part of Ethier and Nagylaki's theorem that
we need. The full statement also gives a convergence result when the sequence 
$(\delta_{N})_N$ converges to a positive real number. 
\end{rem}
To apply this theorem, we consider the two sequences $\epsilon_N=N^{-2}$ and $\delta_N=N^{-1}$, 
we set $E_N=\{z\in (N^{-1}\NN)^{\mc{A}},\ \sum_{\bb{i}\in\mc{A}}z(\bb{i})=1\}$, and we define 
by $(\Phi_N,\Psi_N)$ a change of coordinates such that $\Psi_{N}^{-1}(\{0\})$ is the 
linkage equilibrium manifold: 
$$\begin{array}{ccccccccc}\Phi_N :& E_N &\rightarrow &[0,1]^n& and & \Psi_N :& E_N &\rightarrow &[-1,1]^{2^n-n-1}\\
& z &\mapsto &(u_1,\ldots,u_n)&&& z&\mapsto &(u_I,\ I\subset \enu{n} \mbox{ s. t. } |I|\geq 2 )\\
\end{array}
$$
where $u_i=\sum_{\bb{\ell},\ell_i=0}z(\bb{\ell})$ for $i\in\enu{n}$ and $u_I=\prod_{i\in I}u_{i}-\sum_{\bb{\ell},\bb{\ell}_{|I}\eq 0}z(\bb{\ell})$ for each  $I\subset\enu{n}$ having at least two elements. \\

First (in \Sref{sectproofcond}), we shall check that $X^{(N)}_1=\Phi_N(Z^{(N)}_1)$ 
and $Y^{(N)}_1=\Psi_N(Z^{(N)}_1)$ satisfy the conditions (a)-(f) of Ethier and Nagylaki's 
theorem with the following expressions for the functions $a_{i,j}(x,0)$ and $b_i(x,0)$:
\begin{alignat}{2}
 a_{i,j}(x,0)&=x(i)(1-x(i))\un_{\{i=j\}},\\
b_i(x,0)&=(1-x(i))\mu_1-x(i)\mu_0
+(1/2-x(i))x(i)(1-x(i))P_{i,s}(x),
\end{alignat}
where 
\begin{multline*}
P_{i,s}(x)=\sum_{J\subset \enu{n}\setminus\{i\}}\sum_{H\subset \enu{n}\setminus\{i\}}(s_{J\cup\{i\},H}-s_{J,H})\\
\prod_{j\in J}x(j)\prod_{h\in H}x(h)\prod_{{j\in \enu{n},}\atop{j\not\in J\cup\{i\}}}(1-x(j))\prod_{{h\in \enu{n},}\atop{h\not\in H\cup\{i\}}}(1-x(h)),
\end{multline*}
and, for two subsets $I$ and $J$ of $\enu{n}$, $s_{I,J}$
denotes the assortment parameter $s_{\bb{i},\bb{j}}$ for the types  
$\bb{i}=(\bb{0}_{I},\bb{1}_{\bar{I}})$  and $\bb{j}=(\bb{0}_J,\bb{1}_{\bar{J}})$. \\

In \Sref{sectproofdrift} we shall show that $P_{i,s}$ has the following two equivalent expressions:
\begin{alignat*}{2}
P_{i,s}(x)&=\sum_{A\subset \enu{n}\setminus\{i\}}2^{|A|}\delta_{A\cup\{i\}}[m(s)](\emptyset)\prod_{\ell\in A}x(\ell)(1-x(\ell))\\
&=\sum_{A\subset \enu{n}\setminus\{i\}}\delta_{i}[m(s)](A)\prod_{k\in A}2x(k)(1-x(k))\prod_{\ell\not\in A\cup\{i\}}\Big(1-2x(\ell)(1-x(\ell))\Big).
\end{alignat*}
\subsection{Verification of the conditions (a)-(f) of Ethier and Nagylaki's theorem  \label{sectproofcond}}
As the proportion of individuals of a given type $\bb{i}$ can only change by $\pm 1/N$ in one step:
\begin{itemize}
 \item If $r\in\NN^*$ and $\bb{i}\in\mc{A}$, then 
\begin{equation}\label{momentz1}
\EE_z\big[(Z^{(N)}_1(\bb{i})-z(\bb{i}))^r\big]=N^{-r}\sum_{\bb{j}\in\mc{A}\setminus\{ \bb{i}\}}\Big(f_N(z,\bb{j},\bb{i})+(-1)^{r}f_N(z,\bb{i},\bb{j})\Big)
\end{equation}
\item if $r,u\in\NN^*$,  $\bb{i},\bb{j}\in\mc{A}$ so that $\bb{i}\neq \bb{j}$, then 
\begin{multline}\label{momentz2}\EE_z\big[(Z^{(N)}_1(\bb{i})-z(\bb{i}))^r(Z^{(N)}_1(\bb{j})-z(\bb{j}))^u\big]\\=N^{-(r+u)}\Big((-1)^rf_N(z,\bb{i},\bb{j})+(-1)^{u}f_N(z,\bb{j},\bb{i})\Big)
\end{multline}
\item if $r\geq 3$ and $\bb{i}^{(1)}$,\ldots,$\bb{i}^{(r)}\in\mc{A}$ so that  at least three of them  are distinct, then 
\begin{equation}
\label{momentz3}\EE_z\Big[\prod_{u=1}^{r}\big(Z^{(N)}_1(\bb{i}^{(u)})-z(\bb{i}^{(u)})\big)\Big]=0.
\end{equation}
\end{itemize}
\paragraph{Condition (a).}
To show that condition (a) of Theorem~\ref{ethier/nagylaki thm}
holds, we first examine the drift of $Z^{(N)}$. 
A Taylor expansion of the transition probabilities of the Markov chain $(Z^{(N)}_t)_{t\in\NN}$ using assumption \Href{Hscaling}  yields the following formula:
\begin{lem}\label{driftZ}
For every $\bb{i}\in\mc{A}$, 
$$N^2\EE_z[Z^{(N)}_1(\bb{i})-z(\bb{i})]=NB^{(0)}_{\bb{i}}(z)+B^{(1)}_{\bb{i}}(z)+O(N^{-1}),
 \mbox{ uniformly on } z\in E_N,$$  
where
\begin{eqnarray*}
 B^{(0)}_{\bb{i}}(z)&=&\sum_{\bb{k}\in\mc{A}}\sum_{\bb{j}\in\mc{A}}z(\bb{j})z(\bb{k})q((\bb{j},\bb{k});\bb{i})-z(\bb{i})\\
 B^{(1)}_{\bb{i}}(z)&=&\sum_{\bb{k}\in\mc{A}}\sum_{\bb{j}\in\mc{A}}z(\bb{j})z(\bb{k})\Big(\sum_{u=1}^nq((\bb{j},\bb{k});(1-i_u,\bb{i}_{\enu{n}\setminus\{u\}}))\mu_{1-i_u}-q((\bb{j},\bb{k});\bb{i})\sum_{u=1}^{n}\mu_{i_u}\Big)\\
&+&\sum_{\bb{k}\in\mc{A}}\sum_{\bb{j}\in\mc{A}}s_{\bb{j},\bb{k}}z(\bb{j})z(\bb{k})q((\bb{j},\bb{k});\bb{i})-z(\bb{i})\sum_{\bb{k}\in\mc{A}}s_{\bb{i},\bb{k}}z(\bb{k})\\
&-&\sum_{\bb{k}\in\mc{A}}\sum_{\bb{j}\in\mc{A}}\sum_{\bb{h}\in\mc{A}}s_{\bb{j},\bb{h}}z(\bb{j})z(\bb{h})z(\bb{k})q((\bb{j},\bb{k});\bb{i})+z(\bb{i})\sum_{\bb{h}\in\mc{A}}\sum_{\bb{k}\in\mc{A}}s_{\bb{i},\bb{h}}z(\bb{k})z(\bb{h})
\end{eqnarray*}
\end{lem}
\begin{proof}
By assumption \Href{Hscaling}, for two different types $\bb{i},\bb{j}\in\mathcal{A}$ $$f_N(z,\bb{i},\bb{j}):=\sum_{\bb{k},\bb{\ell}\in\mathcal{A}}z(\bb{i})z(\bb{k})w^{(N)}(z,\bb{i},\bb{k})q((\bb{i},\bb{k});\bb{\ell})\mu^{(N)}(\bb{\ell},\bb{j}).$$
where $w^{(N)}(z,\bb{i},\bb{k})=1+\frac{1}{N}\big(s_{\bb{i},\bb{k}}-\sum_{\bb{h}\in \mc{A}}s_{\bb{i},\bb{h}}z(\bb{h})\big)+O(N^{-2})$  
and 
$$\mu^{(N)}(\bb{\ell},\bb{j})=\left\{\begin{array}{ll}
1-\frac{1}{N}\sum_{u=1}^{n}\mu_{j_u}+O(N^{-2}) & \text{ if } d_h(\bb{\ell},\bb{j})=0\\
\frac{1}{N}\mu_{1-j_i}+O(N^{-2}) & \text{ if } d_h(\bb{\ell},\bb{j})=1 \text{ and } \ell_i=1-j_i\\
O(N^{-2}) & \text{ if }d_h(\bb{\ell},\bb{j})\geq 2
\end{array}\right.$$ 
To prove Lemma~\ref{driftZ}, it suffices to use these expansions in  
$$\EE_z\big[Z^{(N)}_1(\bb{i})-z(\bb{i})\big]
=N^{-1}\sum_{\bb{j}\neq \bb{i}}\Big(f_N(z,\bb{j},\bb{i})-f_N(z,\bb{i},\bb{j})\Big)$$ 
and to simplify.  
\end{proof}
Let $u\in\enu{n}$. To establish an expression for the drift of  $X^{(N)}(u)$, we must 
compute $\sum_{\bb{i}\in\mc{A},i_{u}=0}B^{(0)}_{\bb{i}}(z)$ and $\sum_{\bb{i}\in\mc{A},i_{u}=0}B^{(1)}_{\bb{i}}(z)$. Direct computations yield:
\begin{lem}
For every $u\in\enu{n}$ and $z\in E_N$, 
 \begin{equation}
\label{somib0}
 \sum_{\bb{i}\in\mc{A},\ i_{u}=0}B^{(0)}_{\bb{i}}(z)=0,
\end{equation}
\begin{equation}
\label{somib1}
\sum_{\bb{i}\in\mc{A},\ i_{u}=0}B^{(1)}_{\bb{i}}(z)=(1-x(u))\mu_1-x(u)\mu_0 +\frac{1}{2}G_u(z),
\end{equation} 
where
\[x(u)=\sum_{\bb{i}\in\mc{A},\ i_u=0}z(\bb{i})\ \text{ and }\ G_u(z)=\sum_{\bb{j}\in\mc{A}\ }\sum_{\bb{h}\in\mc{A}}z(\bb{j})z(\bb{h})s_{\bb{j},\bb{h}}(\un_{\{j_u=0\}}-x(u)).
\]
\end{lem}
\begin{proof}
 For $\epsilon\in\{0,1\}$ and $\bb{i}\in\mc{A}$, let $\sigma_{u}^{(\epsilon)}(\bb{i})$ denote the type $\bb{i}$  modified  by setting the allele $\epsilon$ at the locus $u$. 
We shall use the following formula several times:
\begin{equation}
\label{somq}
\sum_{\bb{i}\in\mc{A},\ i_{u}=0}q((\bb{j},\bb{k});\sigma_{u}^{(\epsilon)}(\bb{i}))=\un_{\{j_u=\epsilon\}}+\bar{r}(u)(\un_{\{k_u=\epsilon\}}-\un_{\{j_u=\epsilon\}})
\end{equation}
with $\bar{r}(u)=\sum_{I\subset\enu{n}\setminus\{u\}}r_I=\frac{1}{2}$ by assumption \Href{Hrecombsym}. \\
 First,  formula \eqref{somq} with $\epsilon=0$ provides \[\sum_{\bb{i}\in\mc{A},\ i_{u}=0}B^{(0)}_{\bb{i}}(z)=\sum_{\bb{j}\in\mc{A},\ j_u=0}z(\bb{j})+\bar{r}(u)\sum_{\bb{j}\in\mc{A}}\sum_{\bb{k}\in\mc{A}}(\un_{\{k_u=\epsilon\}}-\un_{\{j_u=\epsilon\}})-\sum_{\bb{i}\in\mc{A},\ i_u=0}z(\bb{i})=0.
\]
Let $B^{(1,j)}_{\bb{i}}(z)$ denote  the $j$-th line of the expression of $B^{(1)}_{\bb{i}}(z)$ for $j\in\{1,2,3\}$.\\ As  $\underset{i\in\mc{A},\ i_u=0, i_x=a}{\sum}q((j,k);\sigma^{\epsilon}_x(\bb{i}))$ does not depend on the value of $a$ if  $u\neq x$:
\begin{multline*}
\sum_{\bb{i}\in\mc{A},\ i_{u}=0}B^{(1,1)}_{\bb{i}}(z)=\sum_{\bb{k}\in\mc{A}}\sum_{\bb{j}\in\mc{A}}z(\bb{j})z(\bb{k})\sum_{\bb{i}\in\mc{A},\ i_u=0}\Big(q((\bb{j},\bb{k});\sigma^{(1)}_u(\bb{i}))\mu_{1}-q((\bb{j},\bb{k});\sigma^{(0)}_u(\bb{i}))\mu_0\Big).
 \end{multline*}

Applying \eqref{somq} again, we obtain: 
$$\sum_{\bb{i}\in\mc{A},\ i_{u}=0}B^{(1,1)}_{\bb{i}}(z)=(1-x(u))\mu_1-x(u)\mu_0.$$
Due to  the  symmetry of the parameters: $s_{\bb{i},\bb{j}}=s_{\bb{j},\bb{i}}$ for $\bb{i},\bb{j}\in\mc{A}$, we have: $$\sum_{\bb{i}\in\mc{A},\ i_{u}=0}B^{(1,2)}_{\bb{i}}(z)=0.$$
Finally, computations using \eqref{somq} yet again yield:
$$
\sum_{\bb{i}\in\mc{A},\ i_{u}=0}B^{(1,3)}_{\bb{i}}(z)=\frac{1}{2}G_u(z).
$$
\end{proof}
To obtain condition (a), it remains to express $G_{u}(z)$ in the new coordinates. The following lemma describes the inverse of the change of coordinates $(\Phi_N,\Psi_N)$: 
\begin{lem}
\label{chgtvar}
For $z\in E_N$ and $L\subset \enu{n}$, set $x(L)=\sum_{\bb{i},\ i_{|L}\eq 0}z(\bb{i})$ with the convention $x(\emptyset)=1$ and  $y(L)=\prod_{\ell\in L}x({\ell})-x(L)$ if $|L|\geq 2$.  Then for every $J\subset \enu{n}$, 
\begin{equation}\label{invchgv}
 z(\bb{0}_{J},\bb{1}_{\bar{J}})=\prod_{i\in J}x(i)\prod_{i\in \bar{J}}(1-x(i))-\sum_{I\subset \enu{n} \st J\subset I,\ |I|\geq 2}(-1)^{|I|-|J|}y(I).
\end{equation}
\end{lem}
\begin{proof}
First, by induction on $n-|J|$, we  show that 
\begin{equation}\label{decomp1}
z(\bb{0}_{J},\bb{1}_{\bar{J}})=\sum_{I\subset\enu{n} \st J\subset I}(-1)^{|I|-|J|}x(I). 
 \end{equation}
Since $z(\bb{0})=x(\enu{n})$, the equality \eqref{decomp1} holds for $J=\enu{n}$. \\
Let $m\in\enu{n}$. Assume that the formula \eqref{decomp1} holds for every subset $J$ of $\enu{n}$ such that $|J|\geq m$. Let $K$ be a subset of $\enu{n}$ with $m-1$ elements. 
  $$
z(\bb{0}_{K},\bb{1}_{\bar{K}})=x(K)-\sum_{L\subset \enu{n} \st K\subsetneq  L}z(\bb{0}_{L},\bb{1}_{\bar{L}})
$$
We apply the formula  \eqref{decomp1} to every term in the sum  and we 
invert the double sum we have obtained: 
$$
z(\bb{0}_{K},\bb{1}_{\bar{K}})=x(K)-\sum_{H\subset\enu{n} \st K\subsetneq  H } x(H) \Big(\sum_{L\subset \enu{n} \st K\subsetneq  L\subset H}(-1)^{|H|-|L|}\Big).
$$
The sum between parentheses is equal to $$\sum_{v=1}^{|H|-|K|}(-1)^{|H|-|K|-v}\binom{|H|-|K|}{v}=-(-1)^{|H|-|K|}.$$
Thus the formula  \eqref{decomp1} is also satisfied for the subset $K$ 
which completes the induction. \\
To complete the proof,  we replace $x(I)$ in \eqref{decomp1} with $\prod_{i\in I}x(i)-y(I)$ for every subset $I$ having at least two elements and use the following equality:
\[
\sum_{I\subset \enu{n},\ J\subset I}\!\!\!(-1)^{|I|-|J|}\prod_{i\in I}x(i)=\prod_{j\in J}x(j)\Big(\!\sum_{L\subset \enu{n}\setminus J}\!\!\!(-1)^{|L|}\prod_{\ell\in L}x(\ell)\Big) =\prod_{j\in J}x(j)\prod_{i\in\enu{n}\setminus J}(1-x(i)).
\]
\end{proof}
To shorten the notation,  set 
\begin{itemize}
 \item $\Lambda_{u}=\enu{n}\setminus\{u\}$ for $u\in\enu{n}$
\item $\Pi_J(v)=\prod_{j\in J}v(j)$  for $v\in [0,1]^n$ and $J\in\mc{P}(\enu{n})$ with the usual convention $\Pi_{\emptyset}=1$, 
\item $s_{I,J}=s_{\bb{i},\bb{j}}$ for $\bb{i}=(\bb{0}_I,\bb{1}_{\bar{I}})$ and $\bb{j}=(\bb{0}_J,\bb{1}_{\bar{J}})$. 
\end{itemize}
With this notation, for every $J\subset \Lambda_u$, 
\begin{itemize}
\item $z(\bb{0}_{J},\bb{1}_{\bar{J}})=(1-x(u))\Pi_J(x)\Pi_{\Lambda_u\setminus J}(1-x)-R_J(y)$,
\item $ z(\bb{0}_{J\cup\{u\}},\bb{1}_{\overline{J\cup\{u\}}})=x(u)\Pi_J(x)\Pi_{\Lambda_u\setminus J}(1-x)-R_{J\cup \{u\}}(y)$, 
\end{itemize}
where $R_{J}(y)$ and $R_{J\cup \{u\}}(y)$ denote polynomial functions that vanish at  $y\equiv 0$. Therefore,  
\begin{multline*}G_u(z)=x(u)(1-x(u))\sum_{J\subset \Lambda_u}\sum_{H\subset \Lambda_u}\Pi_J(x)\Pi_H(x)\Pi_{\Lambda_u\setminus J}(1-x)\Pi_{\Lambda_u\setminus H}(1-x)\times\\\Big(x(u)(s_{J\cup\{u\},H\cup\{u\}}-s_{J,H\cup\{u\}})+(1-x(u))(s_{J\cup\{u\},H}-s_{J,H})\Big)+R_u(x,y),
\end{multline*}
where $R_{u}(x,y)$ is a polynomial function in the variables $x(1),\ldots,x(n)$ and $y(I)$ for  $I\subset\enu{n}$ such that $|I|\geq 2$, that vanishes in the equilibrium manifold: $R_u(x,0)=0$. \\
The expression for $G_{u}(z)$ can be simplified by using the two assumptions 
\Href{Hpairing} on the assortment parameters, that is $s_{J,H}=s_{H,J}$ for every 
$J,H\subset \enu{n}$ and $s_{J\cup \{u\},H\cup \{u\}}=s_{J,H}$ for every $u\in\enu{n}$ 
and $J,H\subset \Lambda_u$: 
\begin{multline*}
G_u(z)=(1-2x(u))x(u)(1-x(u))\times\\\sum_{J\subset \Lambda_u}\sum_{H\subset \Lambda_u}\Pi_J(x)\Pi_H(x)\Pi_{\Lambda_u\setminus J}(1-x)\Pi_{\Lambda_u\setminus H}(1-x)(s_{J\cup\{u\},H}-s_{J,H})+R_u(x,y).
\end{multline*}
In summary, we have established the following expansion of the drift of $X^{(N)}$: 
\begin{lem}\label{driftX}
Assume that hypotheses \Href{Hrecombsym}, \Href{Hscaling}, \Href{Hlinkage} and \Href{Hpairing} hold. 
For every $i\in\enu{n}$, 
\begin{multline}N^2\EE_{z}[X^{(N)}(i)-x(i)]=(1-x(i))\mu_1-x(i)\mu_0\\+(\frac{1}{2}-x(i))x(i)(1-x(i))P_{i,s}(x)+R_i(x,y)+O(N^{-1})
  \end{multline}
uniformly on $z\in E_N$
where $$P_{i,s}(x)=\sum_{J\subset \Lambda_u}\sum_{H\subset \Lambda_u}\Pi_J(x)\Pi_H(x)\Pi_{\Lambda_u\setminus J}(1-x)\Pi_{\Lambda_u\setminus H}(1-x)(s_{J\cup\{u\},H}-s_{J,H})$$ 
and 
$R_i(x,y)$ is a polynomial function in the variables $x(1),\ldots,x(n)$ and $y(I)$ for $I\in\mc{P}(\enu{n})$ with at least two elements such that $R_i(x,\bb{0})=0$. 
\end{lem}

\paragraph{Condition (b).}
Computations similar to those used to obtain \eqref{somib0} lead to the following expansion 
of the second moments of $X^{(N)}_1-x$, showing that condition (b) holds:
\begin{lem}
 $N^2\EE_z\big[(X^{(N)}_{1}(i)-x(i))(X^{(N)}_{1}(j)-x(j))\big]=a_{i,j}(x,y)+O(N^{-1})$,
with $$\begin{cases}a_{i,i}(x,y)=
 x(i)(1-x(i))+O(N^{-1}) \\
  a_{i,j}(x,y)=-2\Big(\underset{I\subset\enu{n}\setminus\{i,j\}}{\sum}r_I\Big) y(\{i,j\})+O(N^{-1}) \mbox{ if }i\neq j
\end{cases}$$
uniformly on $z\in E_N$. 
\end{lem}
\begin{proof}
Let $i,j\in\enu{n}$ and $z\in E_N$. By definition of $X^{(N)}$, 
\begin{multline*}
N^2\EE_z\big[(X^{(N)}_{1}(i)-x(i))(X^{(N)}_{1}(j)-x(j))\big]\\
=N^2\sum_{\bb{k}\in\mc{A},\ k_i=0\ }\sum_{\bb{\ell}\in\mc{A},\ \ell_j=0}
\EE_z\big[(Z^{(N)}_{1}(\bb{k})-z(\bb{k}))(Z^{(N)}_{1}(\bb{\ell})-z(\bb{\ell}))\big]
\end{multline*}
Using formulae \eqref{momentz1} and \eqref{momentz2} and assumption  \Href{Hscaling}, we obtain
\begin{multline*}
N^2\EE_z\big[(X^{(N)}_{1}(i)-x(i))(X^{(N)}_{1}(j)-x(j))\big]\\=\sum_{\bb{k}\in\mc{A}}\sum_{ \bb{\ell}\in\mc{A}}(f_{N}(z,\bb{\ell},\bb{k})+f_N(z,\bb{k},\bb{\ell}))(\un_{\{k_i=0, k_j=0\}}-\un_{\{k_i=0, \ell_j=0\}})\\
=T^{(1)}_{i,j}+T^{(2)}_{i,j}-T^{(3)}_{i,j}-T^{(3)}_{j,i}+O(N^{-1}),
\end{multline*}
where 
\begin{align*}  T^{(1)}_{i,j}&=
\sum_{\bb{t}\in\mc{A}}z(\bb{t})\sum_{\bb{\ell}\in\mc{A}}z(\bb{\ell})\sum_{\bb{k}\in\mc{A},\  k_i=k_j=0}q((\bb{\ell},\bb{t});\bb{k}),\\
 T^{(2)}_{i,j}&=
\sum_{\bb{t}\in\mc{A}}z(\bb{t})\sum_{\bb{k}\in\mc{A},\  k_i=k_j=0}z(\bb{k})\sum_{\bb{\ell}\in\mc{A}}q((\bb{k},\bb{t});\bb{\ell}),\\
  T^{(3)}_{i,j}&=
\sum_{\bb{t}\in\mc{A}}z(\bb{t})\sum_{\bb{\ell}\in\mc{A},\  \ell_j=0}z(\bb{\ell})\sum_{\bb{k}\in\mc{A},\ k_i=0}q((\bb{\ell},\bb{t});\bb{k}).
\end{align*}
With the convention $x(\{i,j\})=x(i)$ if $i=j$, we have 
$T^{(2)}_{i,j}=x(\{i,j\})$ and 
it follows from assumption \Href{Hrecombsym} ($r_I=r_{\bar{I}}$ for every $I\subset\enu{n}$) that 
\begin{align*}  T^{(1)}_{i,j}&=
x(i)x(j)+\sum_{I\subset \enu{n}}r_I\big(\un_{\{i\in I, j\in I\}}+\un_{\{i\not\in I, j\not\in I\}}\big)\big(x(\{i,j\})-x(i)x(j)\big)\\
&=x(i)x(j)+2\Big(\sum_{I\subset \enu{n}\setminus\{i,j\}}r_I\Big)\Big(x(\{i,j\})-x(i)x(j)\Big),\\
  T^{(3)}_{i,j}&=x(\{i,j\})+\Big(\sum_{I\subset\enu{n}\setminus\{i\}}r_{I}\Big)\big(x(i)x(j)-x(\{i,j\})\big)
=\frac{1}{2}\big(x(i)x(j)+x(\{i,j\})\big).
\end{align*}
Therefore, for every $i,j\in\enu{n}$,  
\begin{multline*}N^2\EE_z\big[(X^{(N)}_{1}(i)-x(i))(X^{(N)}_{1}(j)-x(j))\big]\\ =2\Big(\sum_{I\subset\enu{n}\setminus\{i,j\}}r_I\Big)\big(x(\{i,j\})-x(i)x(j)\big)+O(N^{-1}). \end{multline*}
If $i=j$ then $x(\{i,j\})-x(i)x(j)=x(i)(1-x(i))$ and ${\displaystyle \sum_{I\subset\enu{n}\setminus\{i,j\}}r_I=\frac{1}{2}}$. 
\end{proof}

\paragraph{Condition (d).}
Let $I$ be a subset of $\enu{n}$ with at least two elements. To compute the drift of $Y^{(N)}(I)$, 
we use the following lemma and formulae \eqref{momentz1}, \eqref{momentz2} and \eqref{momentz3} describing the moments of $Z^{(N)}_{1}-z$.
\begin{lem}\label{diffprod}
 Let $J$ be a finite set. Consider two families of reals  $\{a_j, j\in J\}$ and $\{b_j, j\in J\}$. The following identity holds:
\begin{equation}\label{eqdiffprod}
 \prod_{j\in J}a_j-\prod_{j\in J}b_j=\sum_{K\subset J,\ K\neq \emptyset\ }\prod_{k\in K}(a_k-b_k)\prod_{\ell \in J\setminus K}b_{\ell}.
\end{equation}
\end{lem}
Computations yield:
\begin{multline}
N \EE_z[Y^{(N)}_1(I)-y(I)]=\sum_{i\in I}\Big(\prod_{\ell\in I\setminus\{ i\}}x(\ell)\sum_{\bb{j}\in\mc{A},\ j_i=0}B^{(0)}_{\bb{j}}(z)\Big)\\-\sum_{\bb{j}\in\mc{A},\ \bb{j}_{|I}\eq 0}B^{(0)}_{\bb{j}}(z)+O(N^{-1}).
\end{multline}
uniformly on $z\in E_N$. 
As we have shown that $\sum_{\bb{j}\in\mc{A},\ j_i=0}B^{(0)}_{\bb{j}}(z)=0$ for every $i\in\enu{n}$ (equation \eqref{somib0}), 
\begin{equation}\label{driftY}
 N\EE_z[Y^{(N)}_1(I)-y(I)]=-\sum_{\bb{j}\in\mc{A},\ \bb{j}_{|I}\eq 0}B^{(0)}_{\bb{j}}(z)+O(N^{-1})
\end{equation}
uniformly on $z\in E_N$.\\
Direct computations provide the following expression of the sum on the right-hand side of \eqref{driftY} using the variables $x(L)=\sum_{\bb{j}\in\mc{A},\ j_{|L}\eq 0}x(j)$ for $L\in\mathcal{P}(\enu{n})$:
\begin{equation}
\label{eqBI}
 \sum_{\bb{j}\in\mc{A},\ \bb{j}_{|I}\eq 0}B^{(0)}_{\bb{j}}(z)=\sum_{L\subset \enu{n} \st I\cap L\neq \emptyset,\ I\cap \bar{L}\neq \emptyset} r_L
\Big(x(I\cap L)x(I\cap \bar{L})-x(I)\Big)
\end{equation}
To obtain an expression for $\EE_z[Y^{(N)}_1(I)-y(I)]$ in the new coordinates, it remains to 
replace each term $x(L)$ for $|L|\geq 2$ with $\prod_{\ell\in L}x(\ell)-y(L)$ in \eqref{eqBI}. 
This leads to the following lemma and shows that condition (d) holds. 
\begin{lem}
For a subset $I$ of $\enu{n}$ having at least two elements, 
\begin{equation}
N \EE_z[Y^{(N)}_1(I)-y(I)]=c_{n,I}(x,y)+O(N^{-1})
\end{equation}
where
\begin{multline*}c_{n,I}(x,y)=-\Big(\sum_{{L\subset \enu{n},}\atop {L\cap I\neq \emptyset,\  \bar{L}\cap I\neq \emptyset}}r_L\Big) y(I)-\un_{\{|I|\geq 4\}}\!\!\!\sum_{{L\subset\enu{n},}\atop {|I\cap L|\geq 2,|I\cap \bar{L}|\geq 2}} r_L y(I\cap L)y(I\cap \bar{L})\\+\un_{\{|I|\geq 3\}}\!\!\!\sum_{{L \subset\enu{n},}\atop {|I\cap L|\geq 2,|I\cap\bar{L}|\geq 1}}(r_L+r_{\bar{L}}) y(L\cap I)\prod_{\ell\in I\cap\bar{L}}x(\ell).
\end{multline*}
\end{lem}
\paragraph{Condition (f). } 
The following lemma shows that the condition (f) holds under the assumption \Href{Hlinkage}:
\begin{lem}
For two distinct loci $k,\ell$, let $r_{k,\ell}$ denote 
the probability that the offspring does not inherit the genes at the loci $k$ and $\ell$ from the same parent, $$r_{k,\ell}= \sum_{I\subset\enu{n},\ k\in I \text{ and } \ell\not\in I}(r_I+r_{\bar{I}}),$$
and set $r(n)=\min(r_{k,h}\ k,h\in\enu{n} \text{ and } h\neq k)$.\\
If $r(n)>0$ then the following system of differential equations
$$(S_{n,I})\left\{\begin{array}{ll}&
\frac{dv_{n,I}}{dt}(t,x,y)=c_{n,I}(x,v_{n,I}(t,x,y)) \\
&v_{n,I}(0,x,y)=y(I) 
\end{array}\right. \forall I\subset \enu{n} \st \ |I|\geq 2 $$
has a unique solution $v_n=\{v_{n,I},\ I\subset \enu{n} \text{ and } |I|\geq 2\}$ which is of the form: 
$$v_{n,I}(t,x,y)=\exp(-r(n)t)f_{n,I}(t,x,y),$$ where $f_{n,I}$ is a continuous and bounded function 
on $\R\times [0,1]^n\times [-1,1]^{2^n-n-1}$ so that the value of $f_{n,I}(t,x,y)$ depends on 
$x$ and $y$ only via the coordinates $x(i)$ for $i\in I$ and $y(J)$ for $J\subset I$ such that $|J|\geq 2$. 
\end{lem}
\begin{rem}
For every subset $I\subset\enu{n}$ with two elements say $k$ and $\ell$, $$\frac{dv_{n,I}}{dt}(t,x,y)=-r_{k,\ell\ } v_{n,I}(t,x,y).$$ Therefore if $r(n)=0$ then there exists a subset $I$ of $\enu{n}$ with two elements such that   $v_{n,I}(t,x,y)=y(I)$. Thus the assumption $r(n)>0$ is a necessary condition for  the solution  of $(S_{n,I})$ to converge to $0$ as $t$ tends to $+\infty$ for any initial values. 
\end{rem}
\begin{proof}
Let $n\geq 2$ and let $I\subset\enu{n}$ be such that $|I|\geq 2$. As $c_{n,I}(x,y)$ depends only on the coordinates $x(\ell)$ for $\ell\in I$ and $y(L)$ for $L\subset I$ such that $|L|\geq 2$, we shall prove by induction on the number of elements of $I$  that for any $J\subset I$, $(S_{n,J})$ has a unique solution of the form  ${v_{n,J}(t,x,y)=\exp(-r(n)t)f_{n,J}(t,x,y)}$, where $f_{n,J}$ is a continuous and bounded function on $\R\times [0,1]^n\times [-1,1]^{2^n-n-1}$ such that the value of $f_{n,J}(t,x,y)$ depends on $x$ and $y$ only through the values of the coordinates $x(j)$ for $j\in J$ and $y(L)$ for $L\subset J$ such that $|L|\geq 2$.  
\begin{itemize}
 \item If $I$ has two elements say $k$ and $\ell$, then $(S_{n,I})$ is the following differential equation: $$\left\{\begin{array}{ll}&\frac{dv_{n,I}}{dt}(t,x,y)=-r_{k,\ell\ }v_{n,I}(t,x,y)\\
&v_{n,I}(0,x,y)=y(I)
\end{array}\right.$$
It has a unique solution $v_{n,I}(t,x,y)=y(I)e^{- r(2)t}f_{n,I}(t,x,y)$ where $$f_{n,I}(t,x,y)=e^{-(r_{k,\ell}-r(2))t}y(I).$$ By assumption $r(k,\ell)\geq r(2)>0$, hence  the result holds.
\item
Let $2 \leq m<n$. Assume that the inductive hypothesis holds for any subsets $J$ with $m$ elements. Let $I$ be a subset of $\enu{n}$ with $m+1$ elements. Then  
$$\frac{dv_{n,I}}{dt}(t,x,y)=-\bar{r}_Iv_{n,I}(t,x,y) +e^{-tr(n)}g(t,x,y)$$ 
where ${\displaystyle \bar{r}_I=\sum_{{L\subset \enu{n} \st}\atop {L\cap I\neq \emptyset,\  \bar{L}\cap I\neq \emptyset}}r_L}$ and 
\begin{multline*}
g(t,x,y)=-\un_{\{|I|\geq 4\}}\sum_{{L\subset\enu{n} \st }\atop {|I\cap L|\geq 2, |I\cap \bar{L}|\geq 2 }} r_L e^{-tr(n)}f_{n,I\cap L}(t,x,y)f_{n,I\cap \bar{L}}(t,x,y)\\+\un_{\{|I|\geq 3\}}\sum_{{L \subset\enu{n} \st}\atop {|I\cap L|\geq 2, |I\cap\bar{L}|\geq 1}}(r_L+r_{\bar{L}}) f_{n,L\cap I}(t,x,y)\prod_{\ell\in I\cap\bar{L}}x(\ell).
\end{multline*}
As $\bar{r}_I$ is the probability that the offspring does not inherit 
all the genes at loci $i\in I$ from the same parent, $\bar{r}_I\geq r(n)$. 
Therefore the differential equation $(S_{n,I})$ has a unique solution: 
$$v_{n,I}(t,x,y)=y(I)e^{-\bar{r}_It}+e^{-\bar{r}_It}\int_{0}^{t}g(s,x,y)e^{(\bar{r}_I-r(n))s}ds.$$ 
By our assumptions on the functions $f_{n,J}$ for $J\subsetneq I$,  $g$ is a bounded continuous function on 
$\R_+\times[0,1]^{n}\times [-1,1]^{2^n-n-1}$ such that the value of $g(t,x,y)$ depends on $x$ and $y$  only through the coordinates $x(i)$ for $i\in I$ and $y(L)$ for $L\subset I$ such that $|L|\geq 2$. Therefore, the function $f_{n,I}(t,x,y)=e^{r(n)t}v_{n,I}(t,x,y)$ has the asserted properties. 

\end{itemize}
\paragraph{Conditions (c) and (e).} Condition (c) is easy to verify using formulae  \eqref{momentz1}, \eqref{momentz2},  \eqref{momentz3} describing  the moments of $Z^{(N)}_1-z$. This leads to: 
$$N^2\EE_z[(X_{1}^{(N)}(i)-x(i))^4]=O(N^{-2})\quad \forall i\in\enu{n}, 
\text{ uniformly on }z\in E_N.$$
Similarly, using Lemma~\ref{diffprod}, we obtain 
$$N\EE_z[(Y_{1}^{(N)}(I)-y(I))^2]=O(N^{-1})\quad \forall I\subset\enu{n},\ s. t.\ |I|\geq 2,  
\text{ uniformly on }z\in E_N.$$
\end{proof}
\subsection{Expressions for the drift\label{sectproofdrift}}
We have shown that the $i$-th coordinate of the drift of the limiting diffusion is 
$$(1-x(i))\mu_1-x(i)\mu_0+(1/2-x(i))x(i)(1-x(i))P_{i,s}(x)$$
where 
\begin{multline*}
P_{i,s}(x)=\sum_{J\subset \enu{n}\setminus\{i\}}\sum_{H\subset \enu{n}\setminus\{i\}}(s_{J\cup\{i\},H}-s_{J,H})\times\\
\prod_{j\in J}x(j)\prod_{h\in H}x(h)\prod_{{j\in \enu{n},}\atop{j\not\in J\cup\{i\}}}(1-x(j))\prod_{{h\in \enu{n},}\atop{h\not\in H\cup\{i\}}}(1-x(h)),
\end{multline*}
and, for two subsets $I$ and $J$ of $\enu{n}$, $s_{I,J}$ denotes the assortment parameter 
$s_{\bb{i},\bb{j}}$ for the types  $\bb{i}=(\bb{0}_{I},\bb{1}_{\bar{I}})$  and $\bb{j}=(\bb{0}_J,\bb{1}_{\bar{J}})$. 
The following lemma states that  $P_{i,s}(x)$ is actually a polynomial function in the variables $x(i)(1-x(i))$ for $i\in\enu{n}\setminus\{u\}$:
\begin{lem}\label{lemsimppoly}
 Let $\Lambda$ be a finite subset of $\NN$. 
Consider a family of reals $\beta=\{\beta_{I,J},\ I,J\subset \Lambda\}$ such that 
$\beta_{I,J}=\beta_{I\setminus J,J\setminus I}$ for every $I,J\subset \Lambda$. 
Then, 
\begin{multline}\label{simppoly}
\sum_{J\subset \Lambda}\sum_{H\subset \Lambda}\left(\beta_{J,H}\prod_{j\in J}x(j)\!\prod_{h\in H}x(h)\!\!\prod_{j\in \Lambda\setminus J}(1-x(j))\!\!\prod_{h\in \Lambda\setminus  H}(1-x(h))\right)\\=\sum_{L\subset \Lambda}C_{L}(\beta)\prod_{\ell\in L}x(\ell)(1-x(\ell))
\end{multline}
where $$C_{L}(\beta)=\sum_{T\subset L}(-2)^{|T|-|L|}\sum_{A\subset T}\beta_{A,T\setminus A}.$$
\end{lem}
\begin{proof}
 Let $P_{\Lambda}(\beta)$ denote the polynomial function on the right-hand side. The proof is by induction on $|\Lambda|$. First,  $P_{\emptyset}(\beta)(x)=\beta_{\emptyset,\emptyset}=C_{\emptyset}(\beta)$. \\
Let $n\in\NN$. Assume that the equality \eqref{simppoly} holds for every subset $\Lambda$ of $\NN$ with at most $n$ elements and every  family of reals $\beta$ satisfying the assumptions of the lemma. \\ 
Let $\Lambda$ be a  subset of $\NN$ with $n+1$ elements, let $j$ be an element of  $\Lambda$ and let $\eta=\{\eta_{I,J},\ I,J\subset \Lambda\}$ be a family of reals such that 
$\eta_{I,J}=\eta_{I\setminus J,J\setminus I}$ for every $I,J\subset \Lambda$. 
We split $P_{\Lambda}(\eta)$ into a sum over the subsets of $\Lambda$ containing $j$ and a sum over the subsets of $\Lambda\setminus\{j\}$ to obtain the following expression:
\begin {multline*}
P_{\Lambda}(\eta)(x)=\!\!\sum_{K\subset \Lambda\setminus\{j\}}\sum_{L\subset \Lambda\setminus\{j\}}\prod_{k\in K}x(k)\!\prod_{\ell\in L}x(\ell)\!\!\prod_{k\in \Lambda\setminus K}(1-x(k))\!\!\prod_{h\in \Lambda\setminus  L}(1-x(h))\times\\ \Big(x(j)^{2}\eta_{K\cup\{j\},L\cup\{j\}}+(1-x(j))^2\eta_{K,L}+x(j)(1-x(j))(\eta_{K\cup\{j\},L}+\eta_{K,L\cup\{j\}})\Big).
\end{multline*}
This expression can be simplified by using that $\eta_{K\cup\{j\},L\cup\{j\}}=\eta_{K,L}$:
\begin{multline*}P_{\Lambda}(\eta)(x)=P_{\Lambda\setminus\{j\}}(\eta^{(0)})(x)\\+x(j)(1-x(j))\big(P_{\Lambda\setminus\{j\}}(\eta^{(1)})(x)+P_{\Lambda\setminus\{j\}}(\eta^{(2)})(x)-2P_{\Lambda\setminus\{j\}}(\eta^{(0)})(x)\big),
\end{multline*}
where $\eta^{(0)}$, $\eta^{(1)}$ and $\eta^{(2)}$ are the following three families of reals indexed by the pairs of subsets of $\Lambda\setminus\{j\}$:  $$\eta^{(0)}_{A,B}=\eta_{A,B},\quad \eta^{(1)}_{A,B}=\eta_{A\cup\{j\},B}\text{ and } \eta^{(2)}_{A,B}=\eta_{A,B\cup\{j\}}\text{ for every } A,B\subset\Lambda\setminus\{j\}.$$ 
The inductive hypothesis applies to $\Lambda\setminus\{j\}$ and the three families of reals $\eta^{(0)}$, $\eta^{(1)}$ and $\eta^{(2)}$:
\[
 P_{\Lambda}(\eta)(x)=\sum_{L\subset \Lambda\setminus\{j\}}C_{L}(\eta)\prod_{\ell\in L}x(\ell)(1-x(\ell))+\sum_{L\subset \Lambda,\ j\in L}\tilde{C}_{L}\prod_{\ell\in L}x(\ell)(1-x(\ell)),
\]
where \[\tilde{C}_{L}=\sum_{T\subset L\setminus\{j\}}(-2)^{|L|-1-|T|}\sum_{A\subset T}(\eta_{A\cup\{j\},T\setminus A}+\eta_{A,(T\cup\{j\})\setminus A}-2\eta_{A,T\setminus A}).
\]
The  double sum of the terms $\eta_{A\cup\{j\},T\setminus A}+\eta_{A,(T\cup\{j\})\setminus A}$  is equal to:
$$
\sum_{T\subset L,\ j\in T}(-2)^{|L|-|T|}\sum_{A\subset T}\eta_{A,T\setminus A}. 
$$
Therefore, $\tilde{C}_{L}=C_{L}(\eta)$ and $P_{\Lambda}(\eta)(x)=\sum_{L\subset \Lambda}C_{L}(\eta)\prod_{\ell\in L}x(\ell)(1-x(\ell))$ which completes the proof by induction.
\end{proof}
By Lemma~\ref{lemsimppoly}, the expanded form of $P_{i,s}$ as a polynomial function of the $n-1$ variables $x(j)(1-x(j))$, $j\neq i$ is: 
\begin{equation}
\label{expandP}
 P_{i,s}(x)=\sum_{L\subset \enu{n}\setminus\{i\}}\alpha_{i,L}(s)\prod_{\ell\in L}x(\ell)(1-x(\ell))
\end{equation}
where $$\alpha_{i,L}(s)=\sum_{T\subset L}(-2)^{|L|-|T|}\sum_{A\subset T}(s_{A\cup\{i\},T\setminus A}-s_{A,T\setminus A}).$$
The coefficient $\alpha_{i,L}(s)$ can be rewritten in terms of
the mean values of the assortment parameters  $m_{T}(s)$ for $T\subset L$:
\[\alpha_{i,L}(s)= 2^{|L|}\sum_{T\subset L}(-1)^{|L|-|T|}(m_{T\cup\{i\}}(s)-m_{T}(s))
= 2^{|L|}\sum_{T\subset L}(-1)^{|L|-|T|}\delta_{i}[m(s)](T).
\]
Indeed, it follows from the assumption \Href{Hpairing} that for every 
$i\in\enu{n}$ and $T\subset\enu{n}\setminus\{i\}$, 
$$m_{T}(s) =2^{-|T|}\sum_{A\subset T}s_{A,T\setminus A}\text{ and }
m_{T\cup\{i\}}(s) =2^{-|T|}\sum_{A\subset T}s_{A\cup \{i\},T\setminus A}.$$ 
Using formula \eqref{expdiffset}, we obtain $\alpha_{i,L}(s)=2^{|L|}\delta_{L\cup \{i\}}[m(s)](\emptyset)$.

The following factorised form of the polynomial function $P_{i,s}$ can be derived from a 
general identity stated in Lemma~\ref{multivariate}:
\[
P_{i,s}(x)=\sum_{A\subset \enu{n}\setminus\{i\}}\delta_{i}[m(s)](A)\prod_{k\in A}2x(k)(1-x(k))\prod_{\ell\not\in A\cup\{i\}}\big(1-2x(\ell)(1-x(\ell))\big).
\]
%
\appendix
\section{Appendix}
\subsection{Combinatorial formulae for difference operators\label{appendInv}}
This section collects some combinatorial formulae used to study the 
limiting diffusion. 
Let $E$ be a finite set and $t$ be a real. For a function $f$ defined on $\mc{P}(E)$, we set $$S_t(f)(A)=\sum_{B\subset A}t^{|A|-|B|}f(B) \text{ for every } A\in\mc{P}(E)$$ (with the usual convention $a^{0}=1$ for every $a\in \R$). 
Most of the combinatorial formulae used in the paper can be deduced from this general identity: 
\begin{lem}\label{multivariate}
 Let $U$ be a subset of $E$ and let $\{x_{u}, u\in U\}$ be a family of reals.\\
 \begin{equation}
  \sum_{A\subset U}S_t(f)(A)\prod_{i\in A}x_i= \sum_{B\subset U}f(B)\prod_{i\in B}x_i\prod_{j\in U\setminus B}(1+tx_j). 
 \end{equation}
\end{lem}
\begin{proof}
One way to derive this equality is to interchange the sum on the right-hand side of the 
equation with the sum that appears in the definition of $S_t(f)(A)$, to use the new summation 
index $C=A\setminus B$ and to recognize the following expansion of the product of the 
terms $1+tx_i$:
$$\prod_{i\in U\setminus B}(1+tx_i)=\sum_{C\subset U\setminus B}t^{|C|}\prod_{i\in C}x_i.$$
\end{proof}
As $S_{-1}(f)(A)$ is nothing other than  $\delta_{A}[f](\emptyset)$ by \eqref{expdiffset}, 
if we apply Lemma~\ref{multivariate} with $t=-1$, $f(A)=\delta_{i}[m(s)](A)$ and the family of 
reals $\{2x(j)(1-x(j)),\ j\in\enu{n}\setminus\{i\}\}$, we obtain the following equality 
\begin{multline*}\sum_{A\subset \enu{n}\setminus\{i\}}2^{|A|}\delta_{A\cup \{i\}}[m(s)](\emptyset)\prod_{\ell\in A}x(\ell)(1-x(\ell))
\\=\sum_{A\subset \enu{n}\setminus\{i\}}\delta_{i}[m(s)](A)\prod_{k\in A}2x(k)(1-x(k))\prod_{\ell\not\in A\cup\{i\}}\Big(1-2x(\ell)(1-x(\ell))\Big).
\end{multline*}
This shows the equality between the expanded form \eqref{Pgeneral} and factorised 
form \eqref{expresGn} of the polynomial term $P_{i,s}(x)$ appearing in the drift of 
the limiting diffusion.

By taking $x_i=-1/t$ for every $i\in U$ in Lemma~\ref{multivariate}, we can deduce the inverse 
of the operator $S_t$. This gives a useful formula for inverting a relation between two sequences 
indexed by the subsets of a finite set. 
\begin{cor}\label{invformula}
The inverse of the operator $S_t$ is $S_{-t}$, that is 
$$f(A)=\sum_{B\subset A}(-t)^{|A|-|B|}S_t(f)(B) \text{ for every } A\subset E.$$
\end{cor}
From Corollary~\ref{invformula} we can deduce the following identity for the finite difference 
operator: 
 \begin{equation}\label{invdifset}
 f(A)=\sum_{B\subset A}\delta_B[f](\emptyset) \text{ for every }A\in\mc{P}(E).
\end{equation} 

By considering the operator $S_t$ for a function $f$ which is constant on subsets having the 
same number of elements, we can rewrite the previous relations to obtain useful formulae
relating two sequences indexed by the integers $0,1,\ldots,n$. 
\begin{cor}\label{invformulacst}
Let $t$ be a real number.  Let $n\in\NN^*$.  For a function $f$ defined on $\enum{0}{n}$, 
let $s_t(f)$ be the function defined by: 
$$s_t(f)(k)=\sum_{\ell=0}^{k}\binom{k}{\ell}t^{k-\ell}f(\ell)\text{ for every }k\in\enu{n}.$$  Then,
\begin{enumerate} \item For every  $x\in \R^n$ 
$$\sum_{j=0}^{n}s_t(f)(j)e_{n,j}(x)=\sum_{\ell=0}^{n}f(\ell)\sum_{L\subset \enu{n} \st |L|=\ell\ }\prod_{i\in L}x_i\prod_{j\in \enu{n}\setminus L}(1+tx_j)$$
where $e_{n,j}$ denotes the elementary polynomial of degree $j$ in $n$ variables: 
$$e_{n,j}(x)=\sum_{J\subset \enu{n} \st |J|=j \ }\prod_{i\in J}x_i. $$
 \item The operator 
 $s_{-t}$ is the inverse of the operator $s_t$:  
$$f(k)=\sum_{\ell=0}^{k}\binom{k}{\ell}(-t)^{k-\ell}s_t(f)(\ell)\text{ for every }k\in\enu{n}.$$
\end{enumerate}
\end{cor}
This corollary provides identities for the forward finite difference operators of any orders 
since  $s_{-1}(f)(k)=\delta^{(k)}[f](0)$ for every $k\in\enum{0}{n}$.  In particular, this leads 
to the following formula used in the proof of Proposition~\ref{propindep}:
\begin{equation}
 \label{deltarel}
\sum_{\ell=0}^{k}\binom{k}{\ell}\delta^{(\ell)}[f](0)=f(k) \text{ for every }k\in\enu{n} 
\end{equation}
and Lemma~\ref{polydelta} used in the proof of Proposition~\ref{descloirev}.
\subsection{Example \ref{exquadraticassort}\label{annexexquadraticassort}}
Under the hypotheses of the assertion 1-(b) of Proposition~\ref{descloirev}, the logarithm of the stationary density $h_{n,s,\mu}$  takes its maximum value  in $[0,1/2]^n$ at a unique point $(\xi_0,\ldots,\xi_0)$ such that $\lambda_0=\xi_0(1-\xi_0)$ is the unique solution in $]0,1/4[$ of the equation $\mathcal{E}^{'}_{0}$:
$$2\mu-1+\sum_{k=0}^{n-1}2^k\delta^{(k+1)}[m](0)\binom{n-1}{k}y^{k+1}=0.$$
In $[0,1/2]^n$ the  saddle points of index $n-1$  has one coordinate equal to 1/2 and $(n-1)$ coordinates equal to $\xi_1$ where $\lambda_1=\xi_1(1-\xi_1)$ is the unique solution  in $]0,1/4[$ of the equation $\mathcal{E}^{'}_{1}$:
$$2\mu-1+\sum_{k=0}^{n-2}\binom{n-2}{k}\left(2^{k-1}\delta^{(k+2)}[m](0)+2^k\delta^{(k+1)}[m](0)\right)y^{k+1}=0.$$
If we denote by $h_{n,i}$ the value of  $h_{n,s,\mu}$ at a critical point of index $n-i$ then  
\begin{alignat*}{2}h_{n,0}-h_{n,n}=&(2\mu+1)n\ln(4\lambda_0)+\sum_{k=0}^{n-1}2^{k}\delta^{(k+1)}[m](0)\binom{n}{k+1}(\lambda^{k+1}_0-(1/4)^{k+1}),\\
h_{n,0}-h_{n,1}=&(2\mu+1)\Big(n\ln(\frac{\lambda_0}{\lambda_1})+\ln(4\lambda_1)\Big)\\
&+\sum_{k=0}^{n-1}2^{k}\delta^{k+1}[m](0)\left(\binom{n-1}{k+1}(\lambda^{k+1}_0-\lambda^{k+1}_{1})\un_{\{k\leq n-2\}}+\binom{n-1}{k}(\lambda_{0}^{k+1}-\frac{1}{4}\lambda^{k}_1)\right).
\end{alignat*}
If  we define the assortment by means of the Hamming criterion with the quadratic sequence of parameters: $s_k=s_0-(bk+ck^2)$ $\forall k\in\enum{0}{n}$ with $c>0$ and $b+c>0$, 
then $$\delta^{(1)}[m](k)=-(b+c+2kc)\ \forall k\in\enum{0}{n-1},\    \delta^{(2)}[m](0)=-2c \text{ and }\delta^{(r)}[m](0)=0\ \forall r\geq 3.$$ In this case, $\lambda_0$ and $\lambda_1$ are solutions of quadratic  functions: $2\mu-1 -(b+c)\lambda_0-4c(n-1)\lambda^{2}_0=0$ and $2\mu-1 -(b+2c)\lambda_1-4c(n-2)\lambda^{2}_1=0$. 
 After some  computations, we obtain: 
$h_{n,0}-h_{n,n}\underset{n\rightarrow +\infty}{\sim} \frac{c}{8}n^2$ and $h_{n,0}-h_{n,1}\underset{n\rightarrow+\infty}{\sim} n^{1/2}1/2\sqrt{c(2\mu-1)}$. 
\subsection{Property of a symmetric matrix}
The following lemma is used to determine the nature of the critical points of the density of 
the invariant measure (Proposition~\ref{descloirev}). 
\begin{lem}\label{symmatrix}
 For a real $a$ and two integers $k$ and $n$ so that $n\geq 1$ and $0\leq k\leq n$, let $M_{n,k}(a)$ denote the following symmetric matrix: 
$$M_{n,k}=\begin{pmatrix}\bb{A}_{k} & \bb{B}_{k,n-k}\\
\bb{B}_{n-k,k} & \bb{A}_{n-k}
\end{pmatrix}$$
where
\begin{itemize}
 \item $\bb{A}_{k}$ denotes  the following $k$-by-$k$ matrix: $\bb{A}_{k}=\begin{pmatrix} 1&a&\cdots&a\\
a&\ddots&\ddots &\vdots\\
\vdots &\ddots& \ddots & a\\
a&\cdots&a&1\\
\end{pmatrix}$
\item $\bb{B}_{k_1,k_2}$ denotes the $k_1$-by-$k_2$ matrix all the elements of which are equal to $-a$. 
\end{itemize}
If $0\leq a < 1$ then $M_{n,k}(a)$ is  positive definite. 
\end{lem}
\begin{proof}
 Let $Q_{n,k,a}$ denote the quadratic form  with matrix $M_{n,k}(a)$ in the canonical basis. For every $x\in\R^n$, $Q_{n,k,a}(x)=\sum_{i=1}^{n}x^{2}_{i}+2a\sum_{1\leq i<j\leq n}\epsilon_i\epsilon_jx_ix_j$, where $\epsilon_{1}=\ldots=\epsilon_k=1$ and 
 $\epsilon_{k+1}=\ldots=\epsilon_{n}=-1$. 
This lemma can be established by induction on $n$ by using the following decomposition of $Q_{n,k,a}(x)$:
$$Q_{n,k,a}(x)=(x_n+a\epsilon_n\sum_{i=1}^{n-1}\epsilon_ix_i)^2+(1-a^2)\Big(\sum_{i=1}^{n-1}x_{i}^{2}+2b\sum_{1\leq i<j\leq n-1}\epsilon_i\epsilon_jx_ix_j\Big).$$ 
where $b=\frac{a}{1+a}\in[0,1[$.
\end{proof}
%
\fancyhead[LO]{}

\end{document}